\pdfoutput=1
\documentclass[reqno]{amsart}
\usepackage{etex}
\usepackage[a4paper,hmarginratio=1:1]{geometry}
\usepackage[OT2, T1]{fontenc}
\usepackage{amssymb,amsfonts,amsmath}
\usepackage{comment} 
\usepackage[all,arc]{xy}
\usepackage{enumerate}
\usepackage{mathrsfs,mathtools}
\usepackage[disable]{todonotes}
\usepackage{stmaryrd}
\usepackage{quiver}
\usepackage{marvosym}
\usepackage{graphicx}
\usepackage{pdflscape} 	
\usepackage{bbm}
\usepackage[utf8]{inputenc}
\usepackage{tikz}
\usepackage{tikz-cd}
\usepackage{caption}
\usetikzlibrary{trees}
\usetikzlibrary[shapes]
\usetikzlibrary[arrows]
\usetikzlibrary{patterns}
\usetikzlibrary{fadings}
\usetikzlibrary{backgrounds}
\usetikzlibrary{decorations.pathreplacing}
\usetikzlibrary{decorations.pathmorphing}
\usetikzlibrary{positioning}
\usepackage[symbol]{footmisc}
	
    \def\biblio{\bibliography{duality}\bibliographystyle{alpha}}

\usepackage{xcolor} 
\usepackage{eucal} 
\usepackage{graphicx}
\SelectTips{cm}{10}
\usepackage{float}

\usepackage[pagebackref]{hyperref}
\definecolor{DefColor}{rgb}{0.6,0.15,0.25}

\definecolor{dark-red}{rgb}{0.5,0.15,0.15}
\definecolor{dark-blue}{rgb}{0.15,0.15,0.6}
\definecolor{dark-green}{rgb}{0.15,0.6,0.15}
\hypersetup{
    colorlinks, linkcolor=dark-red,
    citecolor=dark-blue, urlcolor=dark-green
}

\newcommand{\an}{\mathrm{an}}

\renewcommand*{\backref}[1]{}
\renewcommand*{\backrefalt}[4]{%
  \ifcase #1 %
No citations.
  \or
(cit. on p. #2).%
  \else
(cit on pp. #2).%
  \fi%
}

\usepackage[nameinlink,capitalise,noabbrev]{cleveref}
\newtheorem{thm}{Theorem}[subsection]

\newtheorem{cor}[thm]{Corollary}

\newtheorem{prop}[thm]{Proposition}

\newtheorem{lem}[thm]{Lemma}
\newtheorem{lemma}[thm]{Lemma}
\newtheorem{conj}[thm]{Conjecture}

\theoremstyle{definition}
\newtheorem{defn}[thm]{Definition}

\newtheorem{example}[thm]{Example}

\newtheorem{notation}[thm]{Notation}

\theoremstyle{remark}
\newtheorem{rem}[thm]{Remark}

\theoremstyle{theorem}
\newtheorem*{thm*}{Theorem}

\newtheorem*{quest*}{Question}
\newtheorem*{cor*}{Corollary}

\newtheorem{thmx}{Theorem}

\makeatletter
\let\c@equation\c@thm
\makeatother
\numberwithin{equation}{subsection}

\DeclareMathOperator{\Sp}{Sp}

\DeclareMathOperator{\Spd}{Spd}
\DeclareMathOperator{\Hom}{Hom}
\DeclareMathOperator{\RHom}{RHom}
\DeclareMathOperator{\End}{End}

\DeclareMathOperator{\RZ}{RZ}
\DeclareMathOperator{\height}{ht}
\DeclareMathOperator{\Sht}{Sht}

\DeclareMathOperator{\univ}{univ}

\DeclareMathOperator{\cC}{\mathcal{C}}

\DeclareMathOperator{\LT}{LT}

\DeclareMathOperator{\Spec}{Spec}
\DeclareMathOperator{\Mod}{Mod}

\DeclareMathOperator{\fib}{fib}

\newcommand{\Q}{\mathbb{Q}}

\DeclareMathOperator{\Pic}{Pic}

\DeclareMathOperator{\Spa}{Spa}
\DeclareMathOperator{\Cond}{Cond}
\DeclareMathOperator{\cond}{cond}

\DeclareMathOperator{\Ab}{Ab}
\DeclareMathOperator{\ab}{ab}

\DeclareMathOperator{\FG}{FG}
\newcommand{\ZZ}{\mathbb{Z}}
\newcommand{\QQ}{\mathbb{Q}}

\newcommand{\PP}{\mathbb{P}}

\DeclareMathOperator{\unit}{\mathbbm{1}}

\DeclareMathOperator{\Perf}{Perf}

\DeclareMathOperator{\Aut}{Aut}

\DeclareMathOperator{\Lie}{Lie}

\DeclareMathOperator{\Def}{Def}
\DeclareMathOperator{\coker}{coker}

\newcommand{\et}{\mathrm{\acute{e}t}}
\newcommand{\proet}{\mathrm{pro\acute{e}t}}

\newcommand{\M}{\mathcal{M}}

\newcommand{\xx}{\ast \ast}

\newcommand{\Oxx}{\OO^{ \xx}}

\newcommand{\E}{\mathcal{E}}
\newcommand{\G}{\mathbb{G}}

\DeclareMathOperator{\LC}{LC}

\newcommand{\Z}{\mathbb{Z}}

\newcommand{\HH}{\mathcal{H}}

\newcommand{\SW}{\mathfrak{M}}

\renewcommand{\diamond}{\diamondsuit}
\Crefname{figure}{Figure}{Figures}
\Crefname{assu}{Assumption}{Assumptions}
\Crefname{lem}{Lemma}{Lemmas}
\Crefname{thm}{Theorem}{Theorems}
\Crefname{thma}{Theorem}{Theorems}
\Crefname{prop}{Proposition}{Propositions}

\DeclareMathOperator{\cts}{cts}

\DeclareMathOperator{\BTT}{BT}

\usepackage{lipsum}
\usepackage{adjustbox}

\newcommand{\recollement}[5]{
\xymatrix{{#1} \ar[r]|-{#2} & #3 \ar[r]|-{#4} \ar@<1ex>[l]^-{{#2}_!} \ar@<-1ex>[l]_-{{#2}^*} & #5, \ar@<1ex>[l]^-{{#4}!} \ar@<-1ex>[l]_-{{#4}^*}
}}
\let\lim\relax

\DeclareMathOperator{\lim}{lim}

\newcommand{\cZ}{\mathcal{Z}}

\newcommand{\bG}{\mathbb{G}}

\newcommand{\F}{\mathbb{F}}

\newcommand{\OO}{\mathcal{O}}

\DeclareMathOperator{\GL}{GL}
\DeclareMathOperator{\SL}{SL}

\DeclareMathOperator{\Spf}{Spf}
\DeclareMathOperator{\BT}{BT}
\DeclareMathOperator{\St}{St}
\DeclareMathOperator{\cl}{cl}
\DeclareMathOperator{\alg}{alg}

\DeclareMathOperator{\Gal}{Gal}

\DeclareMathOperator{\Isom}{Isom}
\DeclareMathOperator{\QIsog}{QIsog}
\DeclareMathOperator{\Bun}{Bun}
\newcommand{\abs}[1]{\left\lvert #1 \right\rvert}

\newcommand{\set}[1]{\left\{ #1 \right\}}
\newcommand{\powerseries}[1]{\llbracket #1 \rrbracket}

\newcommand{\isom}{\cong}
\newcommand{\from}{\colon}

\def\isomto{\stackrel{\sim}{\longrightarrow}}

\DeclareMathOperator{\can}{can}

\title{On Hopkins' Picard group}

\author[Barthel]{Tobias Barthel}
\address{Max Planck Institute for Mathematics, Vivatsgasse 7, 53111 Bonn, Germany}
\email{tbarthel@mpim-bonn.mpg.de}

\author[Schlank]{Tomer M.~Schlank}
\address{The Hebrew University of Jerusalem}
\email{tomer.schlank@mail.huji.ac.il}

\author[Stapleton]{Nathaniel Stapleton}
\address{University of Kentucky}
\email{nat.j.stapleton@uky.edu}

\author[Weinstein]{Jared Weinstein}
\address{Boston University Department of Mathematics and Statistics, 665 Commonwealth Avenue, Boston, MA, USA}
\email{jsweinst@bu.edu}

\date{\today}
\begin{document}

\begin{abstract}
    We compute the algebraic Picard group of the category of $K(n)$-local spectra, for all heights $n$ and all primes $p$. In particular, we show that it is always finitely generated over $\Z_p$ and, whenever $n \geq 2$, is of rank $2$, thereby confirming a prediction made by Hopkins in the early 1990s. In fact, with the exception of the anomalous case $n=p=2$, we provide a full set of topological generators for these groups. Our arguments rely on recent advances in $p$-adic geometry to translate the problem to a computation on Drinfeld's symmetric space, which can then be solved using results of Colmez--Dospinescu--Nizio\l.
\end{abstract}

\maketitle

\begin{center}  
\includegraphics[scale=0.4]{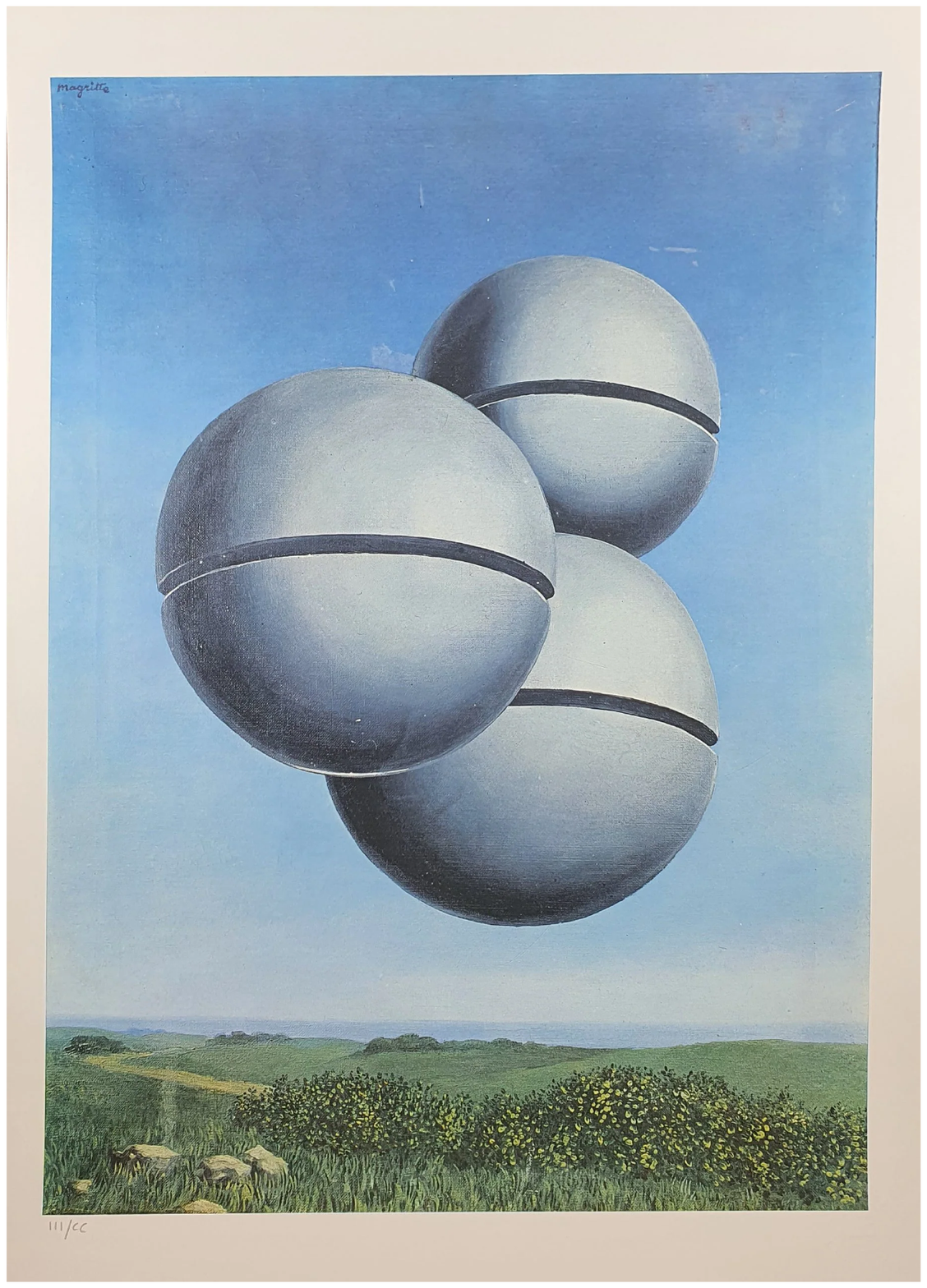} \\
Ren\'e Magritte, {\em The Voice of Space}, 1931
\end{center}

\newpage

\section{Introduction}

\subsection{Results}
This article settles a question about the $K(n)$-local Picard group $\Pic_{n,p}$, relative to a prime $p$ and an integer $n\geq 1$, formulated by Hopkins in the early 1990s (\cite{Strickland1992}). Here, $K(n)$ denotes Morava $K$-theory at height $n$, the prime $p$ being implicit.  Also known as Hopkins' Picard group, $\Pic_{n,p}$ is the group of isomorphism classes of invertible $K(n)$-local spectra, under the operation of $K(n)$-local smash product.  The study of $\Pic_{n,p}$ was initiated by Hopkins--Mahowald--Sadovsky \cite{HopkinsMahowaldSadofsky}, who defined an approximation to $\Pic_{n,p}$ in the form of a comparison map
    \[ 
        \varepsilon\from \Pic_{n,p} \to \Pic_{n,p}^{\alg}. 
    \]
The target $\Pic_{n,p}^{\alg}$ of this map is defined in a purely algebraic manner in terms of equivariant line bundles on Lubin--Tate space (reviewed below).  The map $\varepsilon$ is an isomorphism for $p$ large enough with respect to $n$, but generally it has a kernel $\kappa_{n,p}$ which is a (possibly infinite) product of cyclic $p$-groups with bounded exponent (see \cref{rem:cases} and references therein).  Our main theorem gives a complete description of $\Pic_{n,p}^{\alg}$ together with the image of $\varepsilon$ (up to a $\Z/2$ ambiguity in the case $(n,p)=(2,2)$). 

\begin{thmx}[\Cref{thm:picalg}] \label{ThmMainTopologicalResult} Let $p$ be a prime, assume $n\geq 2$, and write $\mathcal{Z}_{n,p}$ for the pro-cyclic group\footnote[4]{This notation allows us to deal with all primes uniformly. When $p$ is odd, $\mathcal{Z}_{n,p} \cong \Z_p \oplus \Z/(2p^n-2)$, while $\mathcal{Z}_{n,2} \cong \Z_2 \oplus \Z/(2^n-1)$.} $\lim_k \Z/p^k(2p^n-2)$. 
There is an isomorphism
\[\Pic_{n,p}^{\alg}\isom \mathcal{Z}_{n,p}\oplus \Z_p  \oplus (\Z/2)^{\oplus e(n,p)},\]
where
\[ 
e(n,p)=
    \begin{cases}
        0, & p\neq 2; \\
        2, & p=2\text{ and } n\geq 3; \\
        3, & p=2\text{ and } n=2.
    \end{cases}
\]
In all cases except possibly $(n,p)=(2,2)$, the map $\varepsilon$ is surjective. 
In the case $(n,p)=(2,2)$, the image of $\varepsilon$ is a direct summand of
$\Pic_{n,p}^{\alg}$ of index at most 2.
\end{thmx}

The case $n=1$ has long been understood due to \cite{HopkinsMahowaldSadofsky}, so we have excluded it from the statement of the theorem.  The computation for $n=2$ has also been known by work of Hopkins and Lader (\cite{Lader}, $p \geq 5$), Karamanov as well as Goerss--Henn--Mahowald--Rezk (\cite{Karamanov2010, GHMR2015}, $p=3$), and Henn ($p=2$, unpublished). These computations rely on explicit resolutions such as those constructed in \cite{GHMR2005} in the case $n=2$ and $p=3$, and thus seem very difficult to extend to larger heights. Indeed, prior to the present work, there was not a single pair $(n,p)$ with $n\geq 3$ for which $\Pic_{n,p}^{\alg}$ had been computed. 

A precise version of \cref{ThmMainTopologicalResult} appears as  \cref{thm:picalg}, which in fact provides explicit topological generators for $\Pic_{n,p}^{\alg}$. The case $p=2$ being rather delicate, for the moment we focus on what happens when $p$ is odd.  As a profinite group, $\Pic_{n,p}^{\alg}\isom \mathcal{Z}_{n,p}\oplus \Z_p$ is generated by the images under $\varepsilon$ of the following two classes in $\Pic_{n,p}$:  the suspended $K(n)$-local sphere $\Sigma \mathbb{S}_{K(n)}$, and the $(p-1)$-fold power of the determinant sphere $\mathbb{S}_{K(n)}[\det]$, defined in \cref{ssec:pic_twospheres}.  (In the case $n=1$, $\mathbb{S}_{K(1)}[\det]$ and $\Sigma^2\mathbb{S}_{K(1)}$ have the same image in $\Pic_{1,p}^{\alg}$, and $\Pic_{1,p}^{\alg}\isom \mathcal{Z}_{1,p}$ is generated by the image of $\Sigma\mathbb{S}_{K(1)}$.) Combined with the known isomorphism range of the comparison map $\varepsilon$, for all heights and generic primes $p$, we obtain a full computation of the topological $K(n)$-local Picard group, stated here only for the previously unknown cases:

\begin{thmx}[\Cref{cor:pictop}]\label{thmx:pic_genericprimes}
     Let $n > 2$ and $p>(n^2+1)/2$. There is an isomorphism 
        \[
            \Pic_{n,p} \cong \Z_p \oplus \Z_p \oplus \Z/(2p^n-2).
        \]
   As a profinite group, $\Pic_{n,p}$ is generated by $\Sigma \mathbb{S}_{K(n)}$ and $\mathbb{S}_{K(n)}[\det]^{\otimes (p-1)}$.
\end{thmx}

\Cref{ThmMainTopologicalResult} can be reduced to a purely algebraic result concerning deformations of formal groups, as we now explain.  Let $H_n$ be a 1-dimensional formal group of height $n$ over $\overline{\F}_p$, and define the Lubin--Tate ring $A_n$ as the deformation ring of $H_n$.  Then $A_n$ is non-canonically isomorphic to a power series ring in $n-1$ variables over the ring of Witt vectors of $\overline{\F}_p$.  Finally, define the Morava stabilizer group $\G_n$ as the group of automorphisms of the pair $(H_n,\overline{\F}_p)$.  Then $\G_n$ is a profinite group;  it is an extension of $\Gal(\overline{\F}_p/\F_p)$ by the locally pro-$p$ group $\Aut H_n$.  Then $\G_n$ is a profinite group acting continuously on the topological ring $A_n$.  

The ring $A_n$ admits a spectral incarnation $E=E(H_n,\overline{\F}_p)$, known as Morava $E$-theory or the Lubin--Tate spectrum.  This is a $K(n)$-local commutative ring spectrum with homotopy ring $E_* = A_n[u^{\pm 1}]$ for a class $u$ of degree $-2$.  The spectrum $E$ is constructed functorially from the pair $(H_n,\overline{\F}_p)$, so that it admits an action of $\G_n$.  The completed $E$-homology of a spectrum is a graded $E_*$-module which is derived complete with respect to the maximal ideal of $\pi_0 E=A_n$ and equipped with a continuous action of $\G_n$ lying over the action on $E_*$. We call such a structure a Morava module, and refer to \cref{ssec:localPic_morava} for more details.

The algebraic Picard group $\Pic_{n,p}^{\alg}$ is defined as the group of isomorphism classes of invertible Morava modules. The comparison map $\varepsilon$ appearing above is given by applying completed $E$-homology to invertible $K(n)$-local spectra. The algebraic Picard group sits in an exact sequence
\begin{equation}
    \label{eq:even_picard_intro}
 0 \to H^1_{\cts}(\G_n,A_n^\ast)\to \Pic_{n,p}^{\alg} \to \Z/2\to 0.
 \end{equation}
This is due to the fact that $E_*$ is 2-periodic, so an invertible Morava module is supported in either odd or even degrees, giving the map to $\Z/2$;  since the shift $E_*[1]$ is supported in odd degrees, the map to $\Z/2$ is surjective.  The kernel of this map is isomorphic to the group of isomorphism classes of invertible $A_n$-modules carrying an equivariant continuous action of $\G_n$.  Since $A_n$ is a local ring, its Picard group is trivial.  Equivariant continuous actions of $\G_n$ on a free $A_n$-module of rank 1 correspond exactly to classes in the continuous cohomology $H^1_{\cts}(\G_n,A_n^\ast)$, where $A_{n}^{\ast}$ is the group of units in $A_n$.  

The exact sequence \eqref{eq:even_picard_intro} allows us to reduce \cref{ThmMainTopologicalResult} to a computation of $H^1_{\cts}(\G_n,A_n^\ast)$.  
Once again we assume $n\geq 2$, the case $n=1$ being already well-understood.
 
\begin{thmx}
    \label{ThmMainAlgebraicResult}
    Assume $n\geq 2$.  There is an isomorphism
    \[ H^1_{\cts}(\G_n,A_n^\ast)\isom \Z_p^{\oplus 2} \oplus \Z/(p^n-1) \oplus (\Z/2)^{\oplus e(n,p)},\]
    where $e(n,p)$ is as in \cref{ThmMainTopologicalResult}.    
\end{thmx}
\Cref{ThmMainAlgebraicResult} implies the description of $\Pic_{n,p}^{\alg}$, at least up to the question of which extension of $\Z/2$ by $H^1_{\cts}(\G_n,A_n^\ast)$ is represented by \eqref{eq:even_picard_intro}.  
What is more, we show in the course of proving \cref{ThmMainAlgebraicResult} that every element of $H^1_{\cts}(\G_n,A_n^\ast)$ (or rather an index 2 subgroup of it, in the case $n=p=2$) lifts to an element of $\Pic_{n,p}$ of even degree;  this implies the properties of $\varepsilon$ claimed by \cref{ThmMainTopologicalResult}.  

Let $A_n^{\xx}\subset A_n^\ast$ be the subgroup of principal units (i.e., topologically unipotent elements).  
It is convenient to divide \cref{ThmMainAlgebraicResult} into pro-$p$ and prime-to-$p$ parts using the $\G_n$-equivariant decomposition $A_n^\ast\isom A_n^{\xx}\oplus  \overline{\F}_p^\ast$.  The prime-to-$p$ part poses no difficulty (\Cref{prop:evenPic_cohom}), so \cref{ThmMainAlgebraicResult} is reduced to the claim $H^1_{\cts}(\G_n,A_n^{\xx})\isom \Z_p^{\oplus 2}\oplus (\Z/2)^{\oplus e(n,p)}$.  We prove a precise statement (\Cref{thm:A**main}) about $H^1_{\cts}(\G_n,A_n^{\xx})$, and show that it implies \Cref{thm:picalg}.

\subsection{Main ideas of the proof}

An even invertible Morava module (meaning a class in $\Pic_{n,p}^{\alg}$ which maps to $0\in \Z/2$ in \eqref{eq:even_picard_intro}) corresponds to a free $A_n$-module with a continuous equivariant $\G_n$-action.  We may view such modules geometrically, as line bundles on the stack $[\LT_n/\G_n]$, where $\LT_n=\Spf A_n$ is the Lubin--Tate deformation space, i.e., the affine formal scheme parameterizing deformations of the formal group $H_n$.  
The stack $[\LT_n/\G_n]$ arises naturally;  it is isomorphic to the completion of the moduli space of formal groups at the height $n$ locus in characteristic $p$ (see \cref{PropUniformizationByLT}).  (Note that $[\LT_n/\G_n]$ is fibered over the base $\Spf \Z_p$ and not over the Witt vectors of $\overline{\F}_p$.)

To access the Picard group of $[\LT_n/\G_n]$, we rely heavily on Faltings' isomorphism between the Lubin--Tate and Drinfeld towers, as interpreted as a relation between perfectoid spaces \cite{ScholzeWeinstein2013}.  This isomorphism is only available after passage to the generic fiber.  Therefore let $\LT_{n,\eta}$ be the adic generic fiber of $\LT_n$, so that $\LT_{n,\eta}$ is isomorphic to a rigid-analytic $(n-1)$-dimensional open ball over $K=W(\overline{\F}_p)[1/p]$.  

In \cref{ThmEquivalenceOfStacks} we interpret Faltings' isomorphism as a relation between stacks on Scholze's category of diamonds:
\begin{equation}
    \label{EqIsomOfStacksIntro}
    [\LT_{n,\eta}^{\diamond}/\G_n] \isom [\HH^{n-1,\diamond}/\GL_n(\Z_p)]. 
\end{equation}
On the left side of the isomorphism is the ``diamond generic fiber'' of the stack $[\LT_n/\G_n]$.  Appearing on the right side is (the diamond version of) Drinfeld's symmetric space $\HH^{n-1}$, defined as the complement in rigid-analytic $\PP^{n-1}_{\Q_p}$ of all $\Q_p$-rational hyperplanes.  

The isomorphism \eqref{EqIsomOfStacksIntro} is indispensable for the proof of \cref{ThmMainAlgebraicResult}, because it trades the opaque action of $\G_n$ on $A_n$ for the transparent (indeed linear) action of $\GL_n(\Z_p)$ on $\HH^{n-1}$.  We offer a summary of the proof of \cref{ThmMainAlgebraicResult}.  Recall we have already reduced 
\cref{ThmMainAlgebraicResult} to the calculation of $H^1_{\cts}(\G_n,A_n^{\xx})$.  In particular, we want to show that for $n\geq 2$, the torsion-free part of $H^1_{\cts}(\G_n,A_n^{\xx})$ is $\Z_p^{\oplus 2}$, generated by the pro-$p$ parts of the even Morava modules $\varepsilon(\Sigma^2\mathbb{S}_{K(n)}[\det])$ and $\varepsilon(\mathbb{S}_{K(n)}[\det])$.  The proof brings together two important constructions, accounting for each of these two generators.

The first construction involves determinants of formal groups.  The idea is that in certain contexts, it is possible to take the determinant of a height $n$ formal group to obtain a height 1 formal group.  See \cref{ThmDeterminant} for the precise statement.  In particular, the determinant of $H_n$ is isomorphic to $H_1$.  Since the determinant is functorial, we obtain a determinant homomorphism $\det\from \G_n\to \G_1$.  Finally, consider the determinant of the universal deformation of $H_n$ over $A_n$:  this is a deformation of $H_1$, and so it is classified by a map $A_1\to A_n$ which is compatible with $\det\from \G_n\to \G_1$.  Consequently there is an induced map on cohomology: 
\[ \det\!_{\LT}^*\from H^1_{\cts}(\G_1,A_1^{\xx})\to H^1_{\cts}(\G_n,A_n^{\xx}). \]
The source of this map can be calculated directly,
\begin{equation}
    \label{EqH1G1}
H^1_{\cts}(\G_1,A_1^{\xx}) \isom 
    \begin{cases} 
        \Z_p, & p\neq 2; \\ 
        \Z_2\oplus (\Z/2)^{\oplus 2}, & p= 2,
    \end{cases} 
\end{equation}
and the image of the generator of the free part of this group under $\det_{\LT}^*$ is the pro-$p$ part of $\varepsilon(\mathbb{S}_{K(n)}[\det])$, the Morava module of the determinant sphere.  At the prime $2$ the images of the torsion classes under the map $\det_{\LT}^*$  can be lifted to $\mathrm{Pic}_{p,n}$ via a construction from \cite[Section 5.3]{CarmeliSchlankYanovski2024}. 

The second construction is the one that leverages \eqref{EqIsomOfStacksIntro}.  In \cref{sec:exactsequence} we define a sheaf $\Oxx$ of principal units on an adic space, and we observe that $A_n^{\ast\ast}\isom H^0(\LT_{n,\eta},\OO^{\ast\ast})$.  Let $C$ be the completion of an algebraic closure $\overline{\Q}_p$ of $\Q_p$, and let $\HH_C^{n-1}$ be the base change of $\HH^{n-1}$ to $C$.  Another version of \eqref{EqIsomOfStacksIntro} is 
\[[\LT_{n,\eta}^{\diamond}/\G_n] \isom [\HH^{n-1,\diamond}_C/(\Gal(\overline{\Q}_p/\Q_p)\times \GL_n(\Z_p))]. \]
As a consequence of this isomorphism, we obtain a map 
\[ b \from H^1_{\cts}(\G_n,A_n^{\xx})\to H^1_{\proet}(\HH_C^{n-1},\Oxx)^{\Gal(\overline{\Q}_p/\Q_p)\times \GL_n(\Z_p)}. \]
The following theorem appears as \cref{ThmFundamentalExactSequence} combined with \cref{ThmSphereClassIsPrimitive}.

\begin{thmx}
    \label{ThmFundamentalExactSequenceIntro} Assume $(n,p)\neq (2,2)$.  The maps $\det_{\LT}^*$ and $b$ fit into an exact sequence
    \[ 
        0\to H^1_{\cts}(\G_1,A_1^{\xx}) \xrightarrow{\det\!_{\LT}^\ast} H^1_{\cts}(\G_n,A_n^{\xx})\stackrel{b}{\to} H^1_{\proet}(\HH_C^{n-1},\Oxx)^{\Gal(\overline{\Q}_p/\Q_p)\times \GL_n(\Z_p)} \to 0.
    \]
If $(n,p)=(2,2)$, there is a similar exact sequence, except the first nonzero term in the sequence is $H^1_{\cts}(\G_1,A_1^{\xx})\oplus \Z/2$.
\end{thmx}
Interestingly, the reason that $(n,p)=(2,2)$ is treated differently in \cref{ThmMainTopologicalResult} traces back to the fact that it is the only case where there are nontrivial continuous homomorphisms $\SL_n(\Z_p)\to \Q_p^{\xx}$ (see \cref{lem:SLncohom}).

At this point we apply the recent advances of Colmez--Dospinescu--Nizio\l $ $    
\cite{CDNStein, CDNIntegral} on the cohomology of $\HH_C^{n-1}$, which seem to be tailor-made for our purposes.  From their work we deduce (see \cref{CorRank1Invariants}) that if $n\geq 2$, then 
$H^1_{\proet}(\HH_C^{n-1},\Oxx)^{\Gal(\overline{\Q}_p/\Q_p)\times \GL_n(\Z_p)}$ is canonically isomorphic to $\Z_p$ .  The final challenge (\cref{ThmSphereClassIsPrimitive}) is to show that the map $b$ carries the pro-$p$ part of $\varepsilon(\Sigma^2\mathbb{S}_{K(n)})$ onto a generator of this $\Z_p$, so that $b$ is surjective.  

Combining \cref{ThmFundamentalExactSequenceIntro} with \eqref{EqH1G1} shows that $H^1_{\cts}(\G_n,A_n^{\xx})\isom \Z_p^{\oplus 2}\oplus (\Z/2)^{\oplus e(n,p)}$, which completes the proof of \cref{ThmMainAlgebraicResult}.

\subsection*{Acknowledgements}

We thank Agn{\`e}s Beaudry, Pierre Colmez, Ehud de Shalit, Shai Evra, Paul Goerss, Drew Heard, Hans-Werner Henn, Mike Hopkins, Wies{\l}awa Nizio{\l}, Ori Parzanchevski, and Peter Scholze for many helpful discussions. We are also grateful to Guchuan Li, Ningchuan Zhang, and the anonymous referees for comments on a preliminary draft of this paper.  TB, TMS, and NS are grateful to JW for implementing the bulk of the edits to the article following referee review.

TB is supported by the European Research Council (ERC) under Horizon Europe (grant No.~101042990) and would like to thank the Max Planck Institute for its hospitality. NS and TMS were supported by the US-Israel Binational
Science Foundation under grant 2018389. NS was supported by a Sloan research fellowship, NSF Grant DMS-2304781, and a grant from the Simons Foundation (MP-TSM-00002836, NS).  TMS was supported by ISF1588/18 and the
ERC under the European Union’s Horizon 2020 research and innovation program (grant agreement
No. 101125896).  JW was supported by NSF Grant DMS-2401472.  

\setcounter{tocdepth}{1}
\tableofcontents
\def\biblio{}

\vspace{-1cm}
\section{The $K(n)$-local Picard group}\label{sec:localPic}

The purpose of this section is two-fold. First, it reviews some background material on Picard groups in chromatic homotopy theory. We then explain how to reduce the determination of Hopkins' Picard group to a cohomological computation (\cref{thm:A**main}), which will then be tackled in the remainder of this paper. 

\subsection{Motivation: invertible spectra and degree}\label{ssec:localPic_deg}

Let $\cC = (\cC,\otimes,\unit)$ be a symmetric monoidal $\infty$-category. The \emph{Picard group} $\Pic(\cC)$ of $\cC$ is defined as the collection of isomorphism classes of $\otimes$-invertible objects in $\cC$, i.e., those objects $X \in \cC$ for which there exists $Y \in \cC$ with $X \otimes Y \simeq \unit$. If $\cC$ is presentably symmetric monoidal, then $\Pic(\cC)$ is a set \cite[Remark 2.1.4]{MathewStojanoska2016}. The symmetric monoidal product on $\cC$ then equips $\Pic(\cC)$ with the structure of a group, with unit $\unit$.

For example, if $D(\Z)$ is the derived $\infty$-category of $\Z$-modules, then the invertible objects in $D(\Z)$ are exactly those complexes $X$ whose homology is isomorphic to $\Z$ at a single degree and 0 elsewhere; thus $\Pic D(\Z)\isom\Z$. (See \cite[\href{https://stacks.math.columbia.edu/tag/0FNP}{Tag 0FNP}]{stacks-project} for a description of $\Pic D(R)$ for a general ring $R$.)  For $\Sp=(\Sp,\otimes,\mathbb{S})$ the symmetric monoidal $\infty$-category of spectra, we have a symmetric monoidal functor $\Sp\to D(\Z)$ given by $X\mapsto X\otimes H\Z$, where $H\Z$ is the Eilenberg-MacLane spectrum.  This functor preserves invertible objects, and so it induces a homomorphism
\[
        \xymatrix{\deg\colon\Pic(\Sp) \ar[r] & \Z}
    \]
which sends $X$ to the degree in which $X\otimes H\Z$ is supported.  A Postnikov tower argument from \cite[Section 1]{HopkinsMahowaldSadofsky} shows that $\deg$ is an isomorphism, with inverse $n\mapsto \mathbb{S}^n$.  In particular, $\Pic(\Sp)$ is generated by $\Sigma\mathbb{S}$.

Hopkins, Mahowald, and Sadofsky \cite{HopkinsMahowaldSadofsky} studied the Picard group of the category of $K(n)$-local spectra.  Here $K(n)$ denotes Morava $K$-theory of height $n$ at a prime number $p$, with $0\leq n\leq \infty$.  This $K(n)$ is an associative ring spectrum with coefficients 
    \[
        \pi_*K(n) \cong 
            \begin{cases}
                \Q & \text{if } n = 0; \\
                \F_p[v_n^{\pm 1}] & \text{if } 0 < n < \infty; \\
                \F_p & \text{if }  n = \infty,
            \end{cases} 
    \]
with $v_n$ of degree $2(p^n-1)$. By convention, $K(0) = H\Q$ independently of $p$, while $K(\infty) = H\F_p$.   (The prime $p$, being fixed throughout, will be suppressed from the notation.)  The Morava $K$-theories are examples of fields in $\Sp$, as discussed in \cite[\S2.3]{BarthelBeaudry2020}.  A ring spectrum $K$ is said to be a field if every $K$-module is isomorphic to a direct sum of suspensions of $K$.  Consequently, field spectra obey a K\"unneth formula 
    \begin{equation}\label{eq:kunneth}
        K_*(X \otimes Y) \cong K_*(X) \otimes_{K_*} K_*(Y)
    \end{equation}
for any $X, Y \in \Sp$.  In fact, the Morava $K$-theories constitute a complete list of ``prime fields'' in $\Sp$, in the sense that any field spectrum is isomorphic to a direct sum of shifted copies of $K(n)$.   This is a consequence of the nilpotence theorem due to Devinatz, Hopkins, and Smith \cite{DevinatzHopkinsSmith}, see \cite[Proposition 1.9]{HopkinsSmith1998}. For a general introduction to chromatic homotopy
theory, we refer the reader to \cite{BarthelBeaudry2020}.

For a spectrum $K$, we write $\Sp\to \Sp_K$, $X\mapsto X_K$, for the Bousfield localization of $\Sp$ at $K$.  This is the universal localization which inverts all maps $f$ of spectra for which $K\otimes f$ is an equivalence.  The localization functor $\Sp\to \Sp_K$ admits a fully faithful right adjoint $\Sp_K \subset \Sp$, so we may view a $K$-local spectrum as a spectrum. The smash product on $\Sp$ descends to a symmetric monoidal structure on $\Sp_{K}$; explicitly, the product of two $K$-local spectra $X,Y$ is given by 
    \[
        X \hat{\otimes} Y = (X\otimes Y)_{K}. 
    \]
Equipped with this structure, $\Sp_{K} = (\Sp_{K}, \hat{\otimes},\mathbb{S}_{K})$ forms a presentably symmetric monoidal stable $\infty$-category.

This article focuses on the Picard group of the $K(n)$-local category $\Sp_{K(n)}$ in intermediate characteristics $0 < n < \infty$. For $n=0$, we have $\Pic(\Sp_{H\Q}) \cong \Pic(\mathrm{D}(\Q)) \cong \Z$ generated by the $\Q[1]$. For $n=\infty$, the Picard group of the local category has also been computed by Kamiya and Shimomura. In \cite[Theorem 4.3]{KamiyaShimomura2007}, they show that $\Pic(\Sp_{H\F_p}) \cong \Z$, again generated by (the localization of) $\mathbb{S}^1$.

\begin{defn} Let $p$ be a prime number, and let $0< n<\infty$ be a height.  
  We define the \emph{$K(n)$-local Picard group} as 
        \[
            \Pic_{n,p} \coloneqq \Pic(\Sp_{K(n)}).
        \]
\end{defn}

By the K\"unneth formula \eqref{eq:kunneth}, the Morava $K$-homology $K(n)_*(X)$ of any invertible $X \in \Sp_{K(n)}$ is an invertible graded module over the graded field $\pi_*K(n)\isom \F_p[v_n^{\pm 1}]$.  Therefore there exists a \emph{local degree} $\deg_{n,p}(X)\in \Z/2(p^n-1)$ such that $\pi_*K(n)$ is supported in degrees congruent to $\deg_{n,p}(X)$ modulo $2(p^n-1)$. The next result is proven in \cite[Proposition 14.3]{HoveyStrickland}, relying on \cite{HopkinsMahowaldSadofsky}.

\begin{prop}[Hopkins--Mahowald--Sadofsky, Hovey--Strickland]\label{prop:pic_localdegree}
    Let $n\geq 1$. The $K(n)$-local degree provides a surjective map 
        \[
            \xymatrix{\deg_{n,p}\colon \Pic_{n,p} \ar@{->>}[r] & \Z/2(p^n-1)}
        \]
    with kernel an infinite abelian pro-$p$-group. 
\end{prop}

In particular, in contrast to the case of $\Pic(\Sp)$, the naive notion of degree in the $K(n)$-local setting is only a rather coarse approximation to the Picard group. In the next subsection, we will review a refined invariant which provides a much closer algebraic approximation to $\Pic_{n,p}$.

\subsection{Morava modules}\label{ssec:localPic_morava}

As the starting point for the construction of a good algebraic approximation to $\Pic_{n,p}$, we take the following fundamental result of Devinatz--Hopkins \cite{DevinatzHopkins}. A general account of the Galois theory of commutative ring spectra was developed by Rognes \cite{Rognes}.  

Let $H_n$ be a 1-dimensional formal group of height $n$ over $\overline{\F}_p$.  Let $E = E_n$ be Morava $E$-theory of height $n$ associated with $H_n$. This is a $K(n)$-local commutative ring spectrum with
    \[ 
        \pi_*E_n = A_n[u^{\pm 1}],\;\abs{u}=-2.
    \]
Here $\pi_0E_n=A_n$ is the deformation ring of $H_n$, known as the Lubin--Tate deformation ring.  There is a (non-canonical) isomorphism $A_n \isom W(\bar{\F}_p)\llbracket u_1,\ldots, u_{n-1}\rrbracket$, with $W(\bar{\F}_p)$ the ring of Witt vectors of $\bar{\F}_p$.  Let $H_n^{\univ}$ be the universal deformation of $H_n$;  then $\pi_{-2}E_n=\Lie H_n^{\univ}=A_nu$ is the Lie algebra of $H_n^{\univ}$. 

The {\em Morava stabilizer group} is the profinite group $\G_n:=\Aut(H_n,\overline{\F}_p)$.  It is an extension
        \[ 
            0 \to \OO_D^\ast\to \G_n\to \Gal(\overline{\F}_p/\F_p)\to 0,
        \]
    where $\OO_D=\End H_n$ is the ring of integers in a central simple algebra $D$ of invariant $1/n$ over $\Q_p$.  The group $\G_n$ acts continuously on the spectrum $E_n$ and the ring $A_n$.  

\begin{thm}[Devinatz--Hopkins]
    The unit map $\mathbb{S}_{K(n)} \to E_n$ witnesses the target as a $K(n)$-local Galois extension with Galois group $\bG_n$. In particular, there is a canonical equivalence $\mathbb{S}_{K(n)} \simeq E_n^{h\bG_n}$, where the homotopy fixed points must be understood in a continuous sense.
\end{thm}

We can twist the equivalence of this theorem by any $X \in \Sp_{K(n)}$, resulting in a new equivalence $X \simeq (E_n \hat{\otimes} X)^{h\bG_n} = (L_{K(n)}(E_n \otimes X))^{h\bG_n}$, where $X$ is equipped with trivial $\bG_n$-action. This in turn gives rise to a (Bousfield--Kan) descent spectral sequence of signature 
    \begin{equation}\label{eq:descentss}
     E_2^{s,t}(X):= H_{\cts}^s(\G_n,E_*(X)_t)\implies \pi_{t-s}X 
    \end{equation}
that has excellent convergence properties. (For conditions on $X$ that guarantee this identification of the $E_2$-page, see \cite[Theorem 4.3]{BarthelHeard2016}.) Indeed, since $\mathbb{S}_{K(n)} \to E_n$ is a descendable morphism of commutative ring spectra by the smash product theorem \cite[Theorem 7.5.6]{RavenelBook1992} (see also \cite[Lecture 31]{LurieLectures}), the spectral sequence converges strongly and collapses on a finite page with a horizontal vanishing line.

The input to the spectral sequence \eqref{eq:descentss} is the continuous $\bG_n$-cohomology of the completed $E_n$-homology of $X$:
\[ E_*(X) := \pi_*(E_n \hat \otimes X).\] 
Here, $E_*(X)$ is given the $\frak{m}$-adic topology, where 
$\frak{m}= (p,u_1,\ldots, u_{n-1}) \subset \pi_0E_n$ is the maximal ideal of $\pi_0E_n$.  
This structure is captured algebraically in the category of \emph{Morava modules}, $\Mod_{E_*}^{\circlearrowleft \bG_n}$, defined as follows.  A Morava module is a derived $\mathfrak{m}$-complete graded $E_*$-module equipped with a continuous semi-linear action of $\bG_n$. Morphisms of Morava modules are continuous equivariant maps. The category of Morava modules can be modeled as $(\pi_*E_n,E_*E)$-comodules whose underlying $E_*$-module is derived $\frak{m}$-complete. Completing the usual tensor product induces a symmetric monoidal structure on $\Mod_{E_*}^{\circlearrowleft \bG_n}$, denoted $\hat{\otimes}$. With these definitions, we obtain a canonical lift 
    \begin{equation}\label{eq:moravamodules}
        \xymatrix{& \Mod_{E_*}^{\circlearrowleft \bG_n} \ar[d]^U \\
        \Sp_{K(n)} \ar[r]_-{E_*(-)} \ar[ru]^-{E_*(-)} & \Mod_{E_*},}
    \end{equation}
where $U$ denotes the evident forgetful functor, which forgets the action by $\bG_n$ on the given Morava module. Note that $E_*(-)$ is in general not symmetric monoidal. However, when restricted to spectra $X$ with $E_*(X)$ projective (as $E_*$-module), it is symmetric monoidal. In particular, this applies to invertible $K(n)$-local spectra, in light of the following characterization \cite[Theorem 1.3]{HopkinsMahowaldSadofsky}:

\begin{lem}[Hopkins--Mahowald--Sadofsky]
    For $X \in \Sp_{K(n)}$, the following conditions are equivalent:
        \begin{enumerate}
            \item $X$ is invertible;
            \item $E_*(X)$ is free of rank 1 as a $E_*$-module;
            \item $K(n)_*(X)$ is free of rank 1 as a $K(n)_*$-module.
        \end{enumerate}
\end{lem}

\begin{defn}
    The \emph{algebraic Picard group} $\Pic_{n,p}^{\alg}$ is defined as the group of isomorphism classes of $\hat{\otimes}$-invertible Morava modules. The Morava module functor \eqref{eq:moravamodules} then induces a \emph{comparison map}
        \begin{equation}\label{eq:picalg_comparison}
            \varepsilon\colon \Pic_{n,p} \longrightarrow \Pic_{n,p}^{\alg} \coloneqq \Pic(\Mod_{E_*}^{\circlearrowleft \bG_n}), \quad X \mapsto E_*(X).
        \end{equation}
\end{defn}

The comparison map $\varepsilon$ turns out to be a close approximation to the topological Picard group, essentially due to the excellent properties of the descent spectral sequence \eqref{eq:descentss}. The next result summarizes the situation:

\begin{prop}[Hopkins--Mahowald--Sadofsky, Pstragowski, Goerss--Hopkins]\label{prop:comparisontheorem}
    Let $p$ be prime such that $2p-2 > \max\{n^2-1,2n\}$, then the comparison map is bijective:
        \[
            \varepsilon\colon \Pic_{n,p} \xrightarrow{\cong} \Pic_{n,p}^{\alg}.
        \] 
\end{prop}

The injectivity part of this result was established in \cite{HopkinsMahowaldSadofsky}, while surjectivity was proven under slightly stronger hypothesis in \cite{Pstragowski2022}; the version given here is due to Goerss and Hopkins, \cite[Remark 2.4]{GoerssHopkinsComparingDualities}.

In general, however, the map $\varepsilon$ is not injective; non-trivial elements in the kernel are called {\em exotic}. This can only occur in the non-algebraic realm, i.e., when the descent spectral sequence \eqref{eq:descentss} for invertible $X \in \Sp_{K(n)}$ does not collapse for cohomological reasons. 

\begin{defn}\label{def:exoticpic}
    The kernel of $\varepsilon$ is the \emph{exotic Picard group} $\kappa_{n,p}$, fitting into an exact sequence
        \begin{equation}\label{eq:picardsequence}
        \xymatrix{0 \ar[r] & \kappa_{n,p} \ar[r] & \Pic_{n,p}\ar[r]^-{\varepsilon} &  \Pic_{n,p}^{\alg}.}
    \end{equation}
\end{defn}

\begin{rem}\label{rem:exotic}
    For any $n$ and $p$, the group $\kappa_{n,p}$ is an abelian pro-$p$ group of finite exponent.  It is pro-$p$ as a consequence of \cref{prop:pic_localdegree}.  The finiteness of its exponent can be seen as follows: Consider the descent spectral sequence \eqref{eq:descentss}. Since it collapses at a finite page with a horizontal vanishing line, uniformly for any $X \in \Pic_{n,p}$, there is an induced descent filtration on $\kappa_{n,p}$ with graded pieces that are subgroups of $E_r^{r,r-1}(\mathbb{S})$, see for example \cite[Lemma 3.29]{BBGHPS2022pp}. By \cite[Remark 1.8]{GoerssHopkinsComparingDualities}, the latter groups are all $p$-torsion of finite exponent, hence so is $\kappa_{n,p}$. It follows from the Pr\"ufer--Baer theorem that $\kappa_{n,p}$ is then a (possibly infinite) direct product of finite cyclic $p$-groups. For odd $p$, this has already been observed by Heard \cite[Theorem 4.4.1]{HeardThesis}.  
\end{rem}

\subsection{Two distinguished spheres}\label{ssec:pic_twospheres}

We discuss here the two distinguished classes in the local Picard group $\Pic_{n,p}$ which make an appearance in \cref{thm:picalg}.

The first of these classes is simply the suspension $\Sigma\mathbb{S}_{K(n)}$.  This has local degree 1, and so its tensor power $\Sigma^{2p^n-2}\mathbb{S}_{K(n)}$ has local degree 0. In light of \cref{prop:pic_localdegree}, this implies that the assignment $k \mapsto \Sigma^{(2p^n-2)k}\mathbb{S}_{K(n)}$ extends to a homomorphism
    \[
        \xymatrix{\iota_{\can}\colon \Z_p \ar[r] & \Pic_{n,p} \ar[r] & \Pic_{n,p}^{\alg}.}
    \]
In fact, we can do better, extending to a larger procyclic group generated by $\Sigma\mathbb{S}_{K(n)}$. To do so, we will introduce an auxiliary piece of notation (inspired by \cite{GoerssHopkinsComparingDualities}) that allows us to treat the cases of odd and even primes uniformly:

\begin{notation}\label{nota:zpn}
    For each prime $p$ and height $n$, we set $\cZ_{n,p} \coloneqq \lim_k \Z/(p^k(2p^n-2))$.
\end{notation}

\begin{rem}\label{rem:zpn_isoclass}
    For any $n \geq 1$ and prime $p$, the group $\cZ_{n,p}$ is procyclic, isomorphic to
        \[
            \cZ_{n,p} \cong 
                \begin{cases}
                    \Z_p \oplus \Z/2(p^n-1) & \text{if } p>2; \\
                    \Z_2 \oplus \Z/(2^n-1) & \text{if } p=2.
                \end{cases}
        \]
    For example, if $p$ is odd, a generator for the torsion summand in this presentation is given by the element $(p^{nk})_k$. 
\end{rem}

Based on the periodicity theorem \cite{HopkinsSmith1998}, \cite{HopkinsMahowaldSadofsky} uses ``$v_n$-adic interpolation'' to construct an extension of $\iota_{\can}$, which we denote by the same symbol:
    \begin{equation}\label{eq:iotacan}
        \xymatrixcolsep{4pc}{
        \xymatrix{\iota_{\can}\colon \cZ_{n,p} \ar[r]^-{1 \mapsto \Sigma\mathbb{S}_{K(n)}} & \Pic_{n,p}^{\alg}.}}
    \end{equation}
Here, the map $\Z_p \to \cZ_{n,p}$ is induced by maps $\Z/p^k \to \Z/(p^k(2p^n-2))$ determined by $1 \mapsto 2p^n-2$. In particular, for $p$ odd, this splits the presentation of \cref{rem:zpn_isoclass}. However, for $p=2$, this embedding includes $\Z_2$ as an index 2 subgroup of the $\Z_2$ appearing in \cref{rem:zpn_isoclass}.

The second important class is the \emph{determinant sphere}.  Let $\mathbb{S}(1)$ be the $p$-complete sphere with $\bG_n$-action given through the determinant map $\det\colon \bG_n \to \Z_p^{\ast}$. The determinant sphere is defined as 
    \begin{equation}\label{eq:detsphere}
        \mathbb{S}_{K(n)}[\det] \coloneqq (E_n \otimes \mathbb{S}(1))^{h\bG_n},
    \end{equation}
where $\bG_n$ acts diagonally on the $p$-completed smash product $E_n \otimes \mathbb{S}(1)$. As shown in \cite[Theorem 1.1]{BBGS2022}, the determinant sphere topologically realizes the Morava module $A_n[\det]$, meaning the free $A_n$-module of rank 1 where the $\G_n$-action has been twisted by $\det$.  Thus:
    \[ 
        \varepsilon(\mathbb{S}_{K(n)}[\det]) = A_n[\det]. 
    \]

\begin{rem}\label{rem:ghduality}
    The significance of having a concrete topological model for the determinant sphere comes from Gross--Hopkins duality, which describes the Morava module of the dualizing module for the $K(n)$-local category. The latter is by definition the Brown--Comenetz dual of the $n$-th monochromatic layer of the sphere $M_n\mathbb{S} = \fib(L_n\mathbb{S} \to L_{n-1}\mathbb{S})$, denoted $I_n \coloneqq IM_nS^0$. Here, Brown--Comenetz duality $I$ is a lift of Pontryagin duality to the category of spectra. In \cite[ Theorem 6.]{HopkinsGross1994a} and \cite{HopkinsGross1994b} Gross and Hopkins prove that $I_n \in \Pic_{n,p}$ with Morava module $E_*(I_n) \simeq \pi_{n-n^2}E_{n}[\det]$, see also \cite{Strickland2000}. We may then restate this algebraic identification as a $K(n)$-local equivalence 
        \[
            I_n \simeq \Sigma^{n^2-n}\mathbb{S}_{K(n)}[\det] \hat{\otimes} P_{n,p},
        \]
    where $P_{n,p} \in \kappa_{n,p}$. This formula is a source for the construction of exotic elements in the Picard group, see for example \cite{BarthelBeaudryStojanoska2019,HeardLiShi2021}. 
\end{rem}

\subsection{Conjectures and results}\label{ssec:localPic_sota}

The following conjecture, dating from the early 1990s and due to Hopkins, summarizes the expected behaviour of the sequence \eqref{eq:picardsequence}:

\begin{conj}\label{conj:picard}
    For all heights $n \geq 1$ and primes $p$:
        \begin{enumerate}
            \item $\kappa_{n,p}$ is a finite $p$-group;
            \item $\Pic_{n,p}^{\alg}$ is finitely generated over $\Z_p$.
        \end{enumerate}
\end{conj}
Taken together, the two parts of the conjecture imply that $\Pic_{n,p}$ is finitely generated over $\Z_p$.  

\begin{rem}\label{rem:cases}
\Cref{conj:picard} has been confirmed at heights $1$ and $2$ and all primes $p$ through explicit computations; the following table summarizes the state of the art before the present work: 
\begin{table}[H]\label{table:1}
\centering
\resizebox{\textwidth}{!}{  
\begin{tabular}{|c | c | c | c | c | c | }
\hline
$n$ & $p$ & $\Pic_{n,p}$ & $\Pic_{n,p}^{\mathrm{alg}} $ & $\kappa_{n,p}$  &  Reference \\
\hline \hline
$1$ & $\geq 3$ & $\Z_p \oplus \Z/2(p-1)$ &  $\Z_p \oplus \Z/2(p-1)$  & $0$  & \cite{HopkinsMahowaldSadofsky}  \\
\hline
$1$ & $ 2$ & $\Z_2 \oplus \Z/2 \oplus \Z/4$ &  $\Z_2 \oplus (\Z/2)^2$  & $\Z/2$  & \cite{HopkinsMahowaldSadofsky}     \\
\hline
$2$ & $\geq 5$ & $\Z_p^2 \oplus \Z/2(p^2-1)$ &  $\Z_p^2 \oplus \Z/2(p^2-1)$  & $0$  &  Hopkins, \cite{Lader} \\
\hline
$2$ & $3$ & $\Z_3^2 \oplus \Z/16 \oplus (\Z/3)^2$ &  $\Z_3^2 \oplus \Z/16 $  & $(\Z/3)^2$ &  \cite{Karamanov2010,GHMR2015,KamiyaShimomura2004}  \\
\hline
$2$ & $2$ & (?) & $\Z_2^2 \oplus \Z/3 \oplus (\Z/2)^3$ & $(\Z/8)^2 \oplus (\Z/2)^3$ & \cite{BBGHPS2022pp}, Henn \\
\hline
$\geq 3$ & $p$ & ?? & ?? & ?? & \\
\hline
\end{tabular}
}
\caption{Comments: The case $n=2$ and $p\geq 5$ is due to Hopkins and was later part of Lader's thesis \cite{Lader}. The computation of $\Pic_{2,2}^{\alg}$ has been completed by Hans-Werner Henn, but at the time of writing, remains unpublished. Additional partial information about $\kappa_{n,p}$ can be found in \cite{BarthelBeaudryStojanoska2019,HeardLiShi2021,CulverZhang2022pp}.}
\end{table}
\end{rem}

The main result of the present paper is a complete calculation of the algebraic Picard group of $\Sp_{K(n)}$, for all heights $n$ and all primes $p$; in particular, we recover the the computations of $\Pic_{n,p}^{\alg}$ for $n\leq 2$ summarized in the table above. In addition, we provide explicit topological elements which under the comparison map $\varepsilon$ generate the algebraic Picard group with the exception of the case $n=p=2$. The following is \cref{ThmMainTopologicalResult} from the introduction:

\begin{thm}\label{thm:picalg}
    For any prime $p$ and any height $n\ge 2$, the algebraic Picard group of the category of $K(n)$-local spectra is given by:
    \[
        \Pic_{n,p}^{\alg} \cong
            \begin{cases}
                \cZ_{n,p} \oplus \Z_p & \text{if }  p>2 \text{ and } n \geq 2;\\
                \cZ_{n,2} \oplus \Z_2 \oplus (\ZZ/2)^{\oplus 2} & \text{if }  p=2  \text{ and } n>2; \\
                \cZ_{n,2} \oplus \Z_2 \oplus (\ZZ/2)^{\oplus 2} \oplus \Z/2 & \text{if }  p=2  \text{ and } n=2,
            \end{cases}
    \]
    with $\cZ_{n,p}$ as in \cref{nota:zpn}. For all primes $p$, the summands $\cZ_{n,p}$ and $\Z_p$ are generated by $\varepsilon(\Sigma \mathbb{S}_{K(n)})$ and $\varepsilon(\mathbb{S}_{K(n)}[\det]^{\otimes (p-1)})$, respectively. For $p=2$, generators for the additional $(\ZZ/2)^{\oplus 2}$ summand are described in \cref{ssec:evenprime}. 
\end{thm}

In the range where the comparison map $\varepsilon$ is an isomorphism (\cref{prop:comparisontheorem}), we thus obtain a complete computation of the Picard group of $\Sp_{K(n)}$; this is \cref{thmx:pic_genericprimes}.

\begin{cor}\label{cor:pictop}
    When $2p-2 > \max(n^2-1,2n)$, there is an isomorphism 
        \[
            \Pic_{n,p} \cong \Z_p \oplus \Z_p \oplus \Z/(2p^n-2),
        \]
    with generators given by $\Sigma \mathbb{S}_{K(n)}$ and $\mathbb{S}_{K(n)}[\det]^{\otimes (p-1)}$.
\end{cor}

Another direct consequence of \cref{thm:picalg} is the surjectivity of the comparison map in all cases except for possibly the anomalous one when $n=p=2$; see also \cref{rem:nonsurjectivity}.

\begin{cor}
    For all heights $n \geq 1$ and all primes $p$ with the possible exception of $n=p=2$, the comparison map $\varepsilon\colon \Pic_{n,p} \to \Pic_{n,p}^{\alg}$ is surjective.
\end{cor}

The goal of the remainder of this section is to reduce \cref{thm:picalg} to a computation in $p$-adic geometry, stated as \cref{thm:A**main} below, which we then carry out in the rest of this paper.

\subsection{The even algebraic local Picard group}\label{ssec:pic_alg}

We now begin our reduction of \cref{thm:picalg} to a purely cohomological statement. Observe that the $K(n)$-local Picard group and its algebraic counterpart are naturally $\Z/2$-graded, since $E_*$ is concentrated in even degrees and is 2-periodic. Let $\Pic_{n,p}^{\alg,0}$ be the \emph{even algebraic Picard group}, consisting of those invertible Morava modules whose underlying $E_*$-module is \emph{even}, i.e., concentrated in even degrees. Using 2-periodicity, from now on we also restrict attention to the degree $0$ part of the even Morava module under consideration. Thus there is a short exact sequence
    \begin{equation}
        \label{eq:even_picard}
        \xymatrix{0 \ar[r] & \Pic_{n,p}^{\alg,0} \ar[r] & \Pic_{n,p}^{\alg} \ar[r] & \Z/2 \ar[r] & 0.}
    \end{equation}

The even algebraic Picard group has a simple description in terms of continuous group cohomology. 
Let $M \in \Pic_{n,p}^{\alg,0}$ be an even invertible Morava module.  We can think of $M$ as an invertible $A_n$-module with continuous $\G_n$-action.  Since $A_n$ is a local ring, the underlying $A_n$-module of $M$ is free of rank $1$, say with generator $m$.  Define a continuous 1-cocycle $\phi_M\from \G_n\to A_n^\ast$ by
    \[
        g \cdot m = \phi_M(g)m
    \]
for any $g \in \bG_n$. The class of $\phi_M$ in $H^1_{\cts}(\bG_n,A_n^\ast)$ does not depend on the chosen generator.  Conversely, a continuous cocycle in $H^1_{\cts}(\bG_n,A_n^\ast)$ determines an even Morava module by the same formula, and we obtain:

\begin{prop}\label{prop:evenPic_cohom}
    The assignment $M \mapsto \phi_M$ induces an isomorphism
        \[
            \Pic_{n,p}^{\alg,0} \cong H^1_{\cts}(\bG_n,A_n^\ast).
        \]
\end{prop}

\begin{rem}\label{rem:cocyclespheres}
    Let $\Pic_{n,p}^0$ be the even part of the local Picard group, i.e., those invertible $K(n)$-local spectra $X$ with even Morava module $E_*(X)$. We can translate the Morava modules of the two distinguished spheres of \cref{ssec:pic_twospheres} into a 1-cocycle description. Indeed, for each prime $p$ and height $n$, under the composite 
        \[
            \xymatrix{\Pic_{n,p}^0 \ar[r]^-{\varepsilon} & \Pic_{n,p}^{\alg,0} \ar[r]^-{\phi}_-{\cong} & H^1_{\cts}(\bG_n,A_n^\ast)}
        \]
    we have:
        \begin{enumerate}
            \item since $\varepsilon(\Sigma^2\mathbb{S}_{K(n)}) \cong \Lie H_n^{\univ}$, it maps under $\phi$ to the class of the cocycle $t_0\colon \bG_n  \to A_n^{\ast}$, where $t_0$ is determined by the formula $g\cdot u=t_0(g)u$ for $g\in \bG_n$ and $u$ a generator of $\pi_{-2}E_n \cong \Lie H_n^{\univ}$.  
            \item by construction, $\mathbb{S}_{K(n)}[\det]$ maps to the composition of the determinant map $\det\colon \bG_n \to \Z_p^\ast$ with the inclusion map $\Z_p^\ast\to A_n^\ast$.  
        \end{enumerate}
\end{rem}

In order to better understand these elements and in particular their relation, it will be useful to separate the \emph{prime-to-$p$} part of $\Pic_{n,p}^{\alg,0}$ from its \emph{principal} part. Recall that we consider Morava $E$-theory based on $\bar{\F}_p$. As a topological ring, 
    \[ 
        A_n\isom W\powerseries{u_1,\dots,u_{n-1}},
    \]
where $W=W(\overline{\F}_p)$ is the ring of Witt vectors. The ring $A_n$ is endowed with the $\frak{m}$-adic topology, where $\frak{m}=(p,u_1,\dots,u_{n-1})$ is the maximal ideal. The quotient map $A_n\to A_n/\frak{m}\isom \overline{\F}_p$ admits a multiplicative splitting via the Teichm\"uller lift $\overline{\F}_p\to W$.  Thus there is a $\G_n$-equivariant isomorphism of topological groups:
    \[ 
        A_n^\ast\isom A_n^{\ast\ast}\oplus\overline{\F}_p^\ast, 
    \]
where $A_n^{\ast\ast}=1+\frak{m}$.  We call $A_n^{\ast\ast}$ the group of {\em principal units}; principal units will appear in various contexts throughout the article;  we will always use the double-star notation to refer to them. 

Accordingly, the even algebraic Picard group decomposes into prime-to-$p$ and principal parts:
    \begin{equation}\label{eq:picalgeven_decomposition} 
        \Pic_{n,p}^{\alg,0}\isom H^1_{\cts}(\G_n,\overline{\F}_p^\ast)\oplus H^1_{\cts}(\G_n,A_n^{\ast\ast}). 
    \end{equation}
With respect to this decomposition, let us write 
    \[ 
        \varepsilon = \varepsilon^p \oplus \varepsilon_p, 
    \]
where $\varepsilon^p$, resp. $\varepsilon_p$, refers to the composition of $\varepsilon$ with the projection onto the prime-to-$p$ part, resp.~the principal part. 

\subsection{The prime-to-$p$ part and a relation}\label{ssec:prime-to-p-part}

We begin by determining the prime-to-$p$ part of $\Pic_{n,p}^{\alg,0}$. 

\begin{lem}\label{lem:primetop}
    There are isomorphisms 
        \[
            H^1_{\cts}(\G_n,\overline{\F}_p^\ast)\isom \Hom_{\Gal(\overline{\F}_p/\F_p)}(\F_{p^n}^\ast,\overline{\F}_p^\ast) \cong \Z/(p^n-1).
        \]
\end{lem}
\begin{proof}
    For the proof, let us abbreviate $\Gamma=\Gal(\overline{\F}_p/\F_p)$.  We write $\sigma$ for the Frobenius automorphism, so that $\sigma$ generates $\Gamma$ as a profinite group.  The ring $\OO_D=\End H_n$ has a unique two-sided maximal ideal $J$, with quotient ring $\F_{p^n}$. Let $\OO_D^{\ast\ast}=1+J$ be the group of principal units;  then $\OO_D^{\ast\ast}$ is a normal pro-$p$ subgroup of $\G_n$, with quotient
        \[  
            \G_n/\OO_D^{\xx}\isom \F_{p^n}^\ast \rtimes\Gamma. 
        \]
    The pro-$p$ subgroup $\OO_D^{\ast\ast}\subset \G_n$ acts trivially on $\overline{\F}_p^\ast$, on which $p$ acts invertibly, so we have an isomorphism for all $i\geq 0$:
        \[
            H^i_{\cts}(\G_n,\overline{\F}_p^\ast)\isom H^i_{\cts}(\F_{p^n}^\ast \rtimes\Gamma, \overline{\F}_p^\ast).
        \]
    Consider the Lyndon--Hochschild--Serre spectral sequence
        \[ 
            H^i_{\cts}(\Gamma,H^j(\F_{p^n}^\ast,\overline{\F}_p^\ast)) \implies H^{i+j}_{\cts}(\G_n,\overline{\F}_p^\ast). 
        \]
    Since $\Gamma$ is pro-cyclic, its cohomology vanishes in degrees $\geq 2$, so we have a short exact sequence
        \[ 
            0\to H^1_{\cts}(\Gamma,\overline{\F}_p^\ast)\to H^1_{\cts}(\G_n,\overline{\F}_p^\ast) \to H^1(\F_{p^n}^\ast,\overline{\F}_p^\ast)^{\Gamma} \to 0.
        \]
    The term  $H^1_{\cts}(\Gamma,\overline{\F}_p^\ast)$ is the cokernel of $\sigma-1$ on $\overline{\F}_p^\ast$, which is 0.  The term $H^1(\F_{p^n}^\ast,\overline{\F}_p^\ast)^{\Gamma}$ is the group of $\Gamma$-equivariant homomorphisms $\F_{p^n}^\ast\to \overline{\F}_p^\ast$.  This is a cyclic group of order $p^n-1$, generated by the inclusion map $\F_{p^n}^\ast\hookrightarrow \overline{\F}_p^\ast$.
\end{proof}

Next, we can determine the fate of $\Sigma^2\mathbb{S}_{K(n)}$ and $\mathbb{S}_{K(n)}[\det]$ in the prime-to-$p$ part.

\begin{prop} 
\label{prop:primetop}
As a cyclic group of order $p^n-1$, $H^1_{\cts}(\G_n,\overline{\F}_p^\ast)$ is generated by $\varepsilon^p(\Sigma^2\mathbb{S}_{K(n)})$. Furthermore, the following relation holds in $H^1_{\cts}(\G_n,\overline{\F}_p^\ast)$:
    \[ 
        \varepsilon^p(\mathbb{S}_{K(n)}[\det])=\frac{p^n-1}{p-1}\varepsilon^p(\Sigma^2\mathbb{S}_{K(n)}).
    \]
\end{prop}
\begin{proof}  
We use the identification of \cref{lem:primetop} and claim that this group is generated by the prime-to-$p$ part of $\varepsilon(\Sigma^2\mathbb{S}_{K(n)})$.  Recall from \cref{ssec:pic_twospheres} that $\varepsilon(\Sigma^2\mathbb{S}_{K(n)})$ is the Morava module $\Lie H_n^{\univ}$, where $H_n^{\univ}$ is the universal deformation of the (unique up to isomorphism) formal group $H_n \isom H_n^{\univ}\otimes_{A_n} \overline{\F}_p$ of height $n$ over $\overline{\F}_p$.  The prime-to-$p$ part of $\varepsilon(\Sigma^2\mathbb{S}_{K(n)})$ is the class in $H^1_{\cts}(\G_n,\overline{\F}_p^\ast)$ represented by the $\G_n$-equivariant $\overline{\F}_p$-vector space $\Lie H_n^{\univ}\otimes_{A_n} \overline{\F}_p\isom \Lie H_n$.  

Consider the derivative of the action of $\OO_D=\End H_n$ on $H_n$, meaning the action on $\Lie H_n$:  this is a ring homomorphism $\OO_D\to \overline{\F}_p$ identifying the residue field of $\OO_D$ with $\F_{p^n}\subset \overline{\F}_p$. Therefore the subgroup $\OO_D^\ast=\Aut H_n$ of $\G_n$ acts on $\Lie H_n$ through the composition $\OO_D^\ast\to \F_{p^n}^\ast\hookrightarrow\overline{\F}_p^\ast$.  Tracing through the calculation of $H^1_{\cts}(\G_n,\overline{\F}_p^\ast)$ in the proof of \cref{lem:primetop}, we see that it is generated by the class of $\Lie H_n^{\univ}$.  

With respect to the isomorphism $H^1_{\cts}(\G_n,\overline{\F}_p^\ast)\isom \Hom_{\Gamma}(\F_{p^n}^\ast,\overline{\F}_p^\ast)$, the 
class $\varepsilon^p(\mathbb{S}_{K(n)}[\det])$ corresponds to the norm map $\F_{p^n}^\ast \to \F_p^\ast\hookrightarrow\overline{\F}_p^\ast$.  The norm map is $(p^n-1)/(p-1)$ times the inclusion map, giving the relation claimed in the second part of the proposition. 
\end{proof}

The proposition also allows us to compute the local degree of $\mathbb{S}_{K(n)}[\det]$, as we shall explain next. First, we will denote the even Picard group of the $K(n)$-local category by $\Pic_{n,p}^0$; by definition, this is the subgroup of $\Pic_{n,p}$ consisting of those elements that have even Morava module. Restricting the local degree from \cref{prop:pic_localdegree} to the even part, it factors through the inclusion $\Z/(p^n-1) \subseteq \Z/2(p^n-1)$ determined by $1 \mapsto 2$. Therefore, we obtain the \emph{even local degree} homomorphism
    \[
        \xymatrix{\deg_{n,p}^0\colon \Pic_{n,p}^{\alg,0} \ar[r] & \Z/(p^n-1).}
    \]
We observe that the target is normalized so that $\deg_{n,p}^0\varepsilon(\Sigma^2\mathbb{S}_{K(n)}) = 1$. Moreover, the quotient map $A_n^\ast \to (A_n/\frak m)^\ast \cong \bar{\F}_p^\ast$ is $\G_n$-equivariant, and hence induces a reduction map on continuous cohomology, $q\colon H^1_{\cts}(\bG_n,A_n^\ast) \to H^1_{\cts}(\bG_n,\bar{\F}_p^\ast)$. The next result provides an alternative perspective on the even local degree:

\begin{prop}\label{prop:degformula}
    There is a commutative diagram:
        \[
            \xymatrix{\Pic_{n,p}^{0} \ar[r]^-{\deg_{n,p}^0} \ar[d]_{\phi_{(-)}\circ \varepsilon} & \Z/(p^n-1) \ar[d]^{\cong} \\
            H^1_{\cts}(\bG_n,A_n^\ast) \ar[r]^-{q} & H^1_{\cts}(\bG_n,\bar{\F}_p^\ast).}
        \]
\end{prop}
\begin{proof}
    Recall that the kernel of the even local degree map is a pro-$p$-group by \cref{prop:pic_localdegree}, while $H^1_{\cts}(\bG_n,\bar{\F}_p^\ast)$ is a group of order prime to $p$, see \cref{lem:primetop}. It follows that the composite $q \circ \phi_{(-)}\circ \varepsilon$     factors through $\deg_{n,p}^0$, i.e., there is a commutative square
        \[
            \xymatrix{\Pic_{n,p}^{0} \ar[r]^-{\deg_{n,p}^0} \ar[d]_{\phi_{(-)}\circ \varepsilon} & \Z/(p^n-1) \ar[d] \\
            H^1_{\cts}(\bG_n,A_n^\ast) \ar[r]^-{q} & H^1_{\cts}(\bG_n,\bar{\F}_p^\ast).}
        \]
    We claim that the induced right vertical map is an isomorphism. Indeed, it suffices to verify that on a generator of $\Z/(p^n-1)$, which we may take to be $\deg_{n,p}^0(\Sigma^2\mathbb{S}_{K(n)}) = 1$. As already observed in the proof of \cref{prop:primetop}, we have that  $q \circ \phi_{(-)}\circ \varepsilon(\Sigma^2\mathbb{S}_{K(n)}) = \varepsilon^p(\Sigma^2\mathbb{S}_{K(n)})$ is the canonical inclusion $\F_{p^n}^\ast\hookrightarrow \overline{\F}_p^\ast$. This is our preferred generator of the cyclic group $H^1_{\cts}(\bG_n,\bar{\F}_p^\ast)$, which verifies the claim.
\end{proof}

\begin{cor}\label{cor:detdeg}
    The even local degree of $\mathbb{S}_{K(n)}[\det]$ is $\frac{p^n-1}{p-1}$. In particular, we have 
        \[
            \varepsilon^p(\mathbb{S}_{K(n)}[\det]^{\otimes (p-1)}) = 0.
        \]
\end{cor}
\begin{proof}
    This follows from \cref{prop:primetop} and \cref{prop:degformula}.
\end{proof}

\subsection{A reduction of \cref{thm:picalg}}

The next result, whose proof will occupy the remaining sections of this paper, describes the principal part of the even algebraic Picard group. Assuming it, we can then deduce \cref{thm:picalg} up to some special features for the prime 2 that are treated separately in \cref{ssec:evenprime}.

\begin{thm}\label{thm:A**main} Assume $n\geq 2$.  There are isomorphisms
        \[
            H^1_{\cts}(\bG_n,A_n^{\ast\ast}) \cong 
                \begin{cases}
                    \Z_p^2  & \text{if } p>2; \\
                    \Z_2^2 \oplus (\Z/2)^{\oplus 2} & \text{if } p=2 \text{ and } n>2; \\
                    \Z_2^2 \oplus (\Z/2)^{\oplus 3} & \text{if } p=2 \text{ and } n=2.
                \end{cases}
        \]
    For all primes $p$, the free summands are generated by the principal parts of the 1-cocycles $t_0$ and $\det$ of \cref{rem:cocyclespheres}, respectively. For $p=2$, generators for the $(\Z/2)^{\oplus 2}$-summand are described in \cref{ssec:evenprime}. 
\end{thm}

\begin{proof}[Proof of \cref{thm:picalg}, assuming \cref{thm:A**main}]
    By \cref{prop:evenPic_cohom}, \cref{thm:A**main}, \cref{lem:primetop}, and \eqref{eq:picalgeven_decomposition}, we have isomorphisms
        \[
            \psi_{n,p}^0\colon\Pic_{n,p}^{\alg,0} \cong H^1_{\cts}(\bG_n,A_n^{\ast}) \cong 
                \begin{cases}
                    \Z_p^2 \oplus \Z/(p^n-1) & \text{if } p>2; \\
                    \Z_2^2 \oplus \Z/(2^n-1) \oplus (\Z/2)^{\oplus 2} & \text{if } p=2 \text{ and } n>2; \\
                    \Z_2^2 \oplus \Z/(2^n-1) \oplus (\Z/2)^{\oplus 3} & \text{if } p=2 \text{ and } n=2.
                \end{cases}
        \]
    Since the 2-torsion part for $p=2$ is dealt with separately in \cref{ssec:evenprime} below, we will ignore these summands for the remainder of this proof. Keeping in mind \cref{rem:cocyclespheres}, it follows from the identifications of the generators in \cref{thm:A**main} together with \cref{cor:detdeg} that the group $\Z_p \oplus \Z/(p^n-1)$ is generated by $\varepsilon(\Sigma^2\mathbb{S}_{K(n)})$, while the other $\Z_p$ summand is generated by $\varepsilon(\mathbb{S}_{K(n)}[\det])$. 

    In order to deduce \cref{thm:picalg} from this description, it remains to solve the extension problem of \eqref{eq:even_picard}. To this end, recall from \cref{ssec:pic_twospheres} that we have constructed a map
        \[
            \xymatrixcolsep{8pc}{
            \xymatrix{\psi_{n,p}\colon \cZ_{n,p} \oplus \Z_p \ar[r]^-{(\Sigma\mathbb{S}_{K(n)},\mathbb{S}_{K(n)}[\det]^{\otimes (p-1)})} & \Pic_{n,p}^{\alg},}
            }
        \] 
    whose two components send the distinguished generators of $\cZ_{n,p}$ and $\Z_p$ to $\varepsilon(\Sigma\mathbb{S}_{K(n)})$ and $\varepsilon(\mathbb{S}_{K(n)}[\det]^{\otimes (p-1)})$, respectively. This map fits into a commutative diagram of short exact sequences
        \[
            \xymatrix{0 \ar[r] & (\Z_p \oplus \Z/(p^n-1)) \oplus \Z_p \ar[r]^-{(2,1)} \ar[d]_{\psi_{n,p}^0} & \cZ_{n,p} \oplus \Z_p \ar[r] \ar[d]^{\psi_{n,p}} & \Z/2 \ar[r] \ar[d]^-{=} & 0 \\ 
            0 \ar[r] & \Pic_{n,p}^{\alg,0} \ar[r] & \Pic_{n,p}^{\alg} \ar[r] & \Z/2 \ar[r] & 0,}
        \]
    where the top left map is induced by multiplication by $2$ resp.~the identity map on the two cyclic summands. By \cref{thm:A**main}, the left vertical map is split injective with cokernel $(\Z/2)^{\oplus e(n,p)}$ for 
        \[ 
            e(n,p)=
                \begin{cases}
                    0, & p\neq 2;\\
                    2, & p=2\text{ and } n\geq 3;\\
                    3, & p=2\text{ and } n=2.
                \end{cases}
        \]
    Tracing $\varepsilon(\Sigma\mathbb{S}_{K(n)})$ in the right square, we deduce that the right vertical map is an isomorphism. Therefore, the snake lemma shows that $\psi_{n,p}$ is a split injection with cokernel $(\Z/2)^{\oplus e(n,p)}$, as desired. 
\end{proof}

\subsection{Addenda and anomalies for $p=2$}\label{ssec:evenprime}

The case of the even prime has some oddities, which we discuss separately in this subsection. The description of $\Pic_{n,p}^{\alg,0}$ in \cref{thm:picalg} implies that, at least if $p\neq 2$, the comparison map $\varepsilon\from \Pic_{n,p}\to \Pic_{n,p}^{\alg}$ is surjective, since $\Pic_{n,p}^{\alg}$ is generated as a profinite group by the images of $\Sigma\mathbb{S}_{K(n)}$ and $\mathbb{S}_{K(n)}[\det]$ under $\varepsilon$.  For completeness' sake we discuss the surjectivity of $\varepsilon$ when $p=2$, referring to \cref{rem:nonsurjectivity} for the anomalous case $n=p=2$.

\begin{prop}\label{prop:surjectivity_p=2}
    The comparison map $\varepsilon\colon \Pic_{n,2} \to \Pic_{n,2}^{\alg}$ is surjective if $n>2$.
\end{prop}
\begin{proof}
    Let $n>2$. In light of \cref{thm:picalg}, it remains to find $K(n)$-local spectra which realize the 2-torsion in $\Pic_{n,2}^{\alg}$, i.e., to find topological lifts for the non-trivial elements in $\Z/2 \oplus \Z/2  = \Pic_{n,2}^{\alg}[2]$. This has already been achieved in \cite{CarmeliSchlankYanovski2024}; for the convenience of the reader, we sketch the construction. There is a cyclotomic $(\Z/8)^{\ast}$-Galois extension $\mathbb{S}_{K(n)} \to \mathbb{S}_{K(n)}[\omega_{n}^8]$, obtained from $\mathbb{S}_{K(n)}[B^n\Z/8]$ by splitting off a certain idempotent. By \cite[Section 5.3]{CarmeliSchlankYanovski2024}, the cofibers of the unit maps of the three $\Z/2$-Galois subextensions of $\mathbb{S}_{K(n)}[\omega_{n}^8]$ then give the three non-zero elements in $\Pic_{n,2}^{\alg}[2]$.     
\end{proof}

\begin{rem}\label{rem:nonsurjectivity}
    Combining \cref{thm:picalg} and \cref{prop:surjectivity_p=2} show that the comparison map $\varepsilon$ is surjective in all cases except potentially when $(n,p) = (2,2)$. We currently do not know if the generator of the third, ``anomalous'' $\Z/2$-summand that arises in the algebraic Picard group $\Pic_{2,2}^{\alg}$ can be realized topologically. This question has also been studied in unpublished work of Hans-Werner Henn. 
\end{rem}

\subsection{A toy case: height 1}\label{ssec:height1}

For completeness and because we will make use of the following lemma later in the proof of \cref{thm:A**main}, we briefly include a discussion of the case $n=1$.

\begin{lem}\label{lem:height1contribution}
There are isomorphisms 
    \[ 
        H^1_{\cts}(\G_1,A_1^{\xx}) \cong \Hom_{\cts}(\ZZ_p^{\ast},\ZZ_p^{\xx}) \cong  
            \begin{cases}
                \Z_p & \text{if } p>2; \\
                \Z_2 \oplus (\Z/2)^{\oplus 2} & \text{if } p=2.
            \end{cases}
    \]
\end{lem}
\begin{proof}
Recall that $A_1\cong W(\overline{\F}_p)$.  The Morava stabilizer group is $\G_1\isom \Z_p^\ast \times \hat{\ZZ}$;  it acts on $A_1$ through the $\hat{\ZZ}$ quotient.   We first claim that the inclusion $\ZZ_p^{\xx} \to A_1^{\xx}$ induces an equivalence of graded rings $\ZZ_p^{\xx} \cong H^*_{\cts}(\hat{\ZZ},A_1^{\xx})$.  
Indeed, taking $\phi$ to be a topological generator of $\hat{\ZZ}$, what we need to prove is that there is a short exact sequence
    \[
        0\to \ZZ_p^{\xx} \to W(\overline{\F}_p)^{\xx} \xrightarrow{\mathrm{Id}-\phi}  W(\overline{\F}_p)^{\xx} \to 0.
    \]
Using the compatible $p$-adic filtrations on $\ZZ_p$ and $A_1$, it suffices to show that the corresponding sequence of  associated graded pieces 
    \[
        0\to \F_p \to \overline{\F}_p \xrightarrow{\mathrm{Id}-\phi}  \overline{\F}_p \to 0
    \]
is short exact, which is indeed the case. Using the Lyndon--Hochschild--Serre spectral sequence, we find there are isomorphisms
    \[ 
        H^1_{\cts}(\G_1,A_1^{\xx})\isom H^1_{\cts}(\Z_p^\ast \times \hat{\ZZ},A_1^{\xx})\isom H^1_{\cts}(\Z_p^\ast,\Z_p^{\xx})
        \isom \Hom_{\cts}(\Z_p^\ast,\Z_p^{\xx}).
    \]
The calculation of $\Hom_{\cts}(\Z_p^\ast,\Z_p^{\ast\ast})$ is routine.
\end{proof}

\begin{rem}\label{rem:height1}
    The calculation in the proof of \cref{lem:height1contribution} recovers the computation of the algebraic Picard group at height $1$ due to Hopkins, Mahowald, and Sadofsky \cite{HopkinsMahowaldSadofsky}. Indeed, combined with \cref{lem:primetop} and \cref{prop:evenPic_cohom}, we get the even algebraic Picard group $\Pic_{1,p}^{0,\alg}$. The extension problem \eqref{eq:even_picard} is then resolved as in the proof of \cref{thm:A**main}. We obtain isomorphisms
        \[
            \Pic_{1,p}^{\alg} \cong   
                \begin{cases}
                    \Z_p \oplus \Z/(2p-2) & \text{if } p >2; \\
                    \Z_2 \oplus (\Z/2)^{\oplus 2} & \text{if } p = 2.
                \end{cases}
        \]
    For odd primes, the algebraic Picard group is generated by $\varepsilon(\Sigma^1 \mathbb{S}_{K(n)})$; for the even prime, the same class generates the $\Z_2$-summand. Explicit generators for the 2-torsion part of $\Pic_{1,2}^{\alg}$ can again be constructed as in the proof of \cref{prop:surjectivity_p=2}; the original reference is \cite[Sections 4 and 5]{HopkinsMahowaldSadofsky}. For $p>2$, the comparison map gives an isomorphism $\Pic_{1,p} \cong \Pic_{1,p}^{\alg}$, while the exotic part for $p=2$ is a $\Z/2$, see \cite{HopkinsMahowaldSadofsky}. 
\end{rem}

\section{Moduli of formal groups}\label{sec:moduli}

We have reduced our main topological result to an algebraic one, namely \cref{thm:A**main}, which describes the continuous cohomology group $H^1_{\cts}(\G_n,A_n^{\xx})$.  Its proof will require a dive into the theory of moduli of formal groups, especially those modern aspects of the theory which meet $p$-adic geometry. The star player in the proof is $\widehat{\M}_{\FG}^n$, the completion of the moduli stack of formal groups along the height $n$ locus in characteristic $p$.  It enters the picture via \cref{PropUniformizationByLT}, which states that $\widehat{\M}_{\FG}^n$ is the quotient of the formal scheme $\Spf A_n$ by $\G_n$.  

In this section we offer the reader an exposition of the theory of moduli of formal groups and present two results required for the proof of \cref{thm:A**main}.   Neither result is entirely original;  our work here is more or less a direct adaptation of material appearing in \cite{ScholzeWeinstein2013} and \cite{ScholzeLectures}.  Both results concern the ``diamond generic fiber'' of the stack $\widehat{\M}_{\FG}^n$, which we call $\mathfrak{M}_n$, see  \cref{def:diamondstack}.  This $\mathfrak{M}_n$ is a sheaf of groupoids on the v-topology of perfectoid spaces in characteristic 0.  

The first result is \cref{ThmEquivalenceOfStacks}, which gives an alternative description of $\mathfrak{M}_n$ in terms of linear algebra objects (``modifications of vector bundles on the Fargues--Fontaine curve'').  Faltings' isomorphism between the Lubin--Tate and Drinfeld towers appears within \cref{ThmEquivalenceOfStacks}.

The second result is \cref{ThmDeterminant}, which leverages the linear-algebraic description of $\mathfrak{M}_n$ to construct a {\em determinant morphism} $\mathfrak{M}_n\to\mathfrak{M}_1$.  This result states that if $H$ is a 1-dimensional formal group of height $n$ over a perfectoid space $S$, there exists a functorially associated 1-dimensional formal group $\bigwedge^n H$ of height 1 over $S$, whose associated linear algebra objects (the Tate and Dieudonn\'e modules) are the top exterior powers of those of $H$.  Ultimately this result will help us understand the class of the determinant sphere in $H^1_{\cts}(\G_n,A_n^{\xx})$.

\subsection{The stack $\widehat{\M}_{\FG}^n$ of height $n$ formal groups}

For an introduction to the theory of moduli of formal groups aimed at the algebraic topologist, see \cite{Goerss_MFG}.  Unless specified otherwise, all formal groups in this discussion have dimension 1.

Let $\M_{\FG}$ be the moduli stack of 1-dimensional commutative formal groups on the category of schemes with the fpqc topology.  Then $\M_{\FG}$ admits a uniformization by an affine scheme $\Spec L$, where $L=\Z[x_1,x_2,\dots]$ is the Lazard ring.  The scheme $\Spec L$ parametrizes formal group laws, which carry a choice of coordinate $T$.  For each $n\in \Z$, let $[n](T)\in L\powerseries{T}$ be the multiplication-by-$n$ series in the universal formal group law over $L$.

Fix a prime number $p$.  The base change $\M_{\FG}\otimes\Z_{(p)}$ 
admits a well-known height stratification by closed subsets $\M_{FG}^{\geq n}$ for $0\leq n\leq\infty$, characterized as follows.  For $n\geq 0$, let $v_n$ be the coefficient of $T^{p^n}$ in $[p](T)$.  In particular $v_0=p$.  For $1\leq n\leq \infty$, the vanishing locus of $(v_0,v_1,\dots,v_{n-1})$ in $\Spec L$ is coordinate-invariant, meaning that it descends to a closed substack of $\M_{\FG}\otimes \Z_{(p)}$, to wit $\M_{\FG}^{\geq n}$.   We refer to the locally closed substack $\M_{\FG}^n\coloneqq \M_{FG}^{\geq n+1}\setminus \M_{FG}^{\geq n}$ as the moduli stack of formal groups of height $n$.  Given a characteristic $p$ scheme $S$, the $S$-points of $\M_{\FG}^n$ classify formal groups $H$ over $S$ whose $p$-torsion $H[p]$ is a locally free group scheme of rank $p^n$ over $S$.  

Formal groups of height $n$ obey an important rigidity condition with respect to isomorphisms. This is proven by Goerss in \cite[Theorem 5.23]{Goerss_MFG}, who attributes it to Lazard; see also \cite[Lecture 14]{LurieLectures}.

\begin{lemma} \label{LemmaRigidity}
Let $R$ be a ring of characteristic $p$, and let $H,H'$ be formal groups over $R$ of height $n$, where $1\leq n < \infty$.  Let $\underline{\Isom}(H,H')$ denote the functor sending an $R$-algebra $R'$ to the set of isomorphisms $H\otimes_R R'\to H'\otimes_R R'$.  Then $\underline{\Isom}(H,H')$ is represented by an affine scheme $\Spec R(\infty)$, where $R(\infty)$ is the inductive limit of injective finite \'etale maps $R=R(1)\to R(2) \to\cdots$.
\end{lemma}

In the context of Lemma \ref{LemmaRigidity}, the morphism $\underline{\Isom}(H,H')\to \Spec R$ admits a section over any geometric point of $\Spec R$. Therefore this morphism is surjective.  

For each $1\leq n<\infty$, there exists a formal group $H_n$ of height $n$ over $\overline{\F}_p$ (see \cite[Lecture 13, Corollary 2]{LurieLectures}), which is then unique up to isomorphism.  Recall the Morava stabilizer group $\G_n=\Aut(H_n,\overline{\F}_p)$.  The point $\Spec \overline{\F}_p\to \M_{\FG}^n$ classifying $H_n$ is a $\G_n$-torsor.  In other words, the choice of $H_n$ determines an isomorphism of stacks over $\Spec \F_p$:
\[ \M_{\FG}^n \isom [\Spec \overline{\F}_p/ \G_n]. \]

For $1\leq n<\infty$, let $\widehat{\M}_{\FG}^n$ be the completion of $\M_{\FG}$ along the locally closed substack $\M_{\FG}^n$.  This is a formal stack over $\Spf \Z_p$.  We may describe it by giving its values on schemes on which $p$ is locally nilpotent.   If $R$ is a ring in which $p$ is nilpotent, the $\Spec R$-points of $\widehat{\M}_{\FG}^n$ classify formal groups $H$ over $R$ such that $H\otimes_R R/I$ has height $n$ for some nilpotent ideal $I$ containing $p$.  

An equivalent formulation of the completion $\widehat{\M}_{\FG}^n$ extends its domain to more general formal schemes over $\Spf \Z_p$.  An adic ring \cite[\href{https://stacks.math.columbia.edu/tag/07E8}{Tag 07E8}]{stacks-project} is a topological ring which is complete for the $I$-adic topology, where $I$ is an ideal.  An adic $\Z_p$-algebra is then an adic ring $R$ admitting a continuous structure homomorphism $\Z_p\to R$; equivalently, it is an adic ring for which one can take $I$ to contain $p$.  Then for an adic $\Z_p$-algebra $R$, the $\Spf R$-points of $\widehat{\M}_{\FG}^n$ classify formal groups $H$ over $R$ such that $H\otimes_R R/I$ has height $n$ for some open ideal $I\subset R$ containing $p$.  

The stack $\widehat{\M}_{\FG}^n$ has a standard presentation as a pro-\'etale quotient of a finite-type formal scheme.  Let 
\[ \LT_n=\Def(H_n)\]
be the deformation space of $H_n$.  This means that for an Artinian local ring $R$ with residue field $\overline{\F}_p$, the $R$-points of $\LT_n$ classify 
pairs $(H,\rho)$, where $H$ is a formal group over $R$, and $\rho\from H_n\to H\otimes_R \overline{\F}_p$ is an isomorphism. 
Then $\LT_n$ is pro-representable by an affine formal scheme:  $\LT_n=\Spf A_n$, where $A_n$ is an adic $\Z_p$-algebra, isomorphic to $W(\overline{\F}_p)\powerseries{u_1,\dots,u_{n-1}}$ with its $(p,u_1,\dots,u_{n-1})$-adic topology \cite{LubinTate}.  The formal scheme $\LT_n$ admits an action of $\G_n$, which is compatible with the action of $\Gal(\overline{\F}_p/\F_p)$ on $W(\overline{\F}_p)$.

\begin{prop}  \label{PropUniformizationByLT}  We have an isomorphism of stacks of formal schemes in the pro-\'etale topology over $\Spf \Z_p$:
\[ \widehat{\M}_{\FG}^n\isom [\LT_n/\G_n]. \]
\end{prop}

\Cref{PropUniformizationByLT} is standard, but we discuss its proof in some detail as a means of building up the necessary machinery for what follows.  To do so, it will be important to have a moduli interpretation for the $R$-points of $\LT_n$, where $R$ is a general ring (not Artinian or even Noetherian) in which $p$ is nilpotent.  This interpretation, appearing in the moduli spaces studied by Rapoport--Zink \cite{RapoportZink}, shifts the focus from formal groups to the larger category of $p$-divisible groups, which we review in the next subsection.

\subsection{$p$-divisible groups}

These were introduced by Tate \cite{TatePDivisibleGroups}.  

\begin{defn} Let $R$ be a ring.  A {\em $p$-divisible group} (or {\em Barsotti--Tate group}) of height $n$ over $R$ is an inductive system $G=\varinjlim_{\nu\geq 1} G^\nu$ of commutative group schemes $G^\nu$ over $R$, such that each $G^\nu$ is locally free over $R$ of rank $p^{n\nu}$.  It is required that $G^\nu\to G^{\nu+1}$ induces an isomorphism of $G^\nu$ with the $p^\nu$-torsion in $G^{\nu+1}$.  Let $\BTT_R$ denote the category of $p$-divisible groups over $R$, with the evident notion of morphism;  then $\BTT_R$ is a $\Z_p$-linear category.  
\end{defn}

A $p$-divisible group $G=\varinjlim_\nu G^\nu$ over a ring $R$ determines an fppf sheaf of abelian groups on the category of $R$-algebras $S$, via its functor of points:  $G(S)=\varinjlim_\nu G^{\nu}(S)$.  In this way $\BT_R$ can be seen as a full subcategory of the category of fppf sheaves of abelian groups on $\Spec R$.  The functor of points extends to the situation where $G$ is a $p$-divisible group over an adic ring $R$:  Whenever $S$ is an adic $R$-algebra, one defines $G(S)=\varprojlim G(S/I)$, where $I$ runs over ideals of definition of $S$.  

There is a close link between formal groups and $p$-divisible groups.  Let $R$ be an adic $\Z_p$-algebra, and let $H/R$ be a formal group representing a $\Spf R$-point of $\widehat{\M}_{\FG}^n$.  Then $H[p^\infty]:=\varinjlim_{\nu\geq 1} H[p^\nu]$ is a $p$-divisible group over $R$ of height $n$. A theorem of Tate \cite[\S2, Proposition 1]{TatePDivisibleGroups} states that if $R$ is a complete Noetherian local ring with residue field $k$ of characteristic $p$, then the functor $H\mapsto H[p^\infty]$ is an equivalence between the category of formal groups over $R$ (of whatever dimension) which are $p$-divisible (meaning that the $p$-torsion is a locally free group scheme) and the full subcategory of $\BT_R$ whose objects $G$ are connected $p$-divisible groups (meaning, each $G^\nu$ is 
connected).  

\begin{defn}  Let $G$ and $G'$ be $p$-divisible groups over a ring $R$.  An {\em isogeny} $f\from G\to G'$ between $p$-divisible groups over $R$ is a morphism which is surjective (as a map between fppf sheaves) whose kernel is a finite locally free group scheme over $R$.  If $f$ is an isogeny, its kernel $\ker f$ has rank $p^n$, where $n$ is a locally constant function on $\Spec R$;  we call $n=\height(f)$ the {\em height} of $f$.  Let $\BTT_R^0=\BTT_R\otimes\Q$ and call this the category of $p$-divisible groups up to isogeny. An isomorphism in $\BTT_R^0$ is called a {\em quasi-isogeny}.
\end{defn}

Thus the multiplication-by-$p$ map on $G$ is an isogeny whose height equals the height of $G$.  Note that if $f\from G\to G'$ is an isogeny of height $n$, then $\ker f$ is contained in $G[p^n]$, meaning that $f$ factors through multiplication by $p^n$, so that $f$ is also a quasi-isogeny.  Conversely, a quasi-isogeny becomes an isogeny when multiplied by a sufficient power of $p$.  If a composition $f'\circ f$ of isogenies is defined, then $\ker(f'\circ f)$ is an extension of $\ker f'$ by $\ker f$, whence $\height(f'\circ f)=\height(f')+\height(f)$.  There is a unique extension of the height function to quasi-isogenies preserving this relation.

Quasi-isogenies satisfy a rigidity property similar to what appears in \cref{LemmaRigidity}.  For $p$-divisible groups $G$ and $G'$ over a ring $R$, let $\underline{\QIsog}(G,G')$ be the functor which associates to an $R$-algebra $S$ the set $\QIsog(G_S,G'_S)$, meaning the set of quasi-isogenies $G\otimes_R S\to G'\otimes_R S$.  Then $\underline{\QIsog}(G,G')$ is formally \'etale over $\Spec R$:

\begin{prop}[{\cite{DrinfeldCoverings}}]
    \label{PropRigidityOfQIsog} Let $R$ be a ring in which $p$ is nilpotent, and let $I\subset R$ be a nilpotent ideal.  Let $G$ and $G'$ be $p$-divisible groups over $R$.  The natural map $\QIsog(G,G')\to \QIsog(G_{R/I},G'_{R/I})$ is a bijection. 
\end{prop}

Consequently, given $p$-divisible groups $G$ and $G'$ over $R$, an isogeny $G\otimes_R R/I\to G'\otimes_R R/I$ will lift to an isogeny $G\to G'$ once multiplied by a sufficiently high power of $p$.

\subsection{Lubin--Tate space as a moduli space of $p$-divisible groups}
\label{SecLTRZ}
Recall that $H_n$ is the (unique up to isomorphism) 1-dimensional formal group of height $n$ over $\overline{\F}_p$.  Let
\[ G_n = H_n[p^\infty], \]
so that $G_n$ is a $p$-divisible group over $\overline{\F}_p$ of height $n$.  Let $\OO_{D_n}=\End H_n\isom \End G_n$, and let $D_n=\OO_{D_n}[1/p]$.  Then $D_n$ is a division algebra over $\Q_p$ of invariant $1/n$, and $\OO_{D_n}$ is its ring of integers.  Therefore the group of automorphisms of $G_n$ (as an object in $\BT_{\overline{\F}_p}$) is $\OO_{D_n}^\ast$, and the group of quasi-isogenies of $G_n$ (that is, automorphisms in $\BT^0_{\overline{\F}_p}$) is $D_n^\ast$.  By \cref{PropRigidityOfQIsog}, the functor $\QIsog(G_n,G_n)$ is formally \'etale over an algebraically closed field; it is therefore the constant group scheme associated to $D_n^\ast$.  

We can view the Morava stabilizer group $\G_n$ as the group of automorphisms of the pair $(G_n,\overline{\F}_p)$.  Let $\G_n^0$ denote the group of automorphisms of the pair $(G_n,\overline{\F}_p)$ in the isogeny category.  Explicitly, an element of $\G_n^0$ is a pair $(\tau,u)$, where $\tau\in \Gal(\overline{\F}_p/\F_p)$, and $u\from G_n\to G_n\otimes_{\overline{\F}_p,\tau} \overline{\F}_p$ is a quasi-isogeny.  Let $\height\from\G_n^0\to\Z$ be the homomorphism sending $(\tau,u)$ to the height of the quasi-isogeny $\tau$.  

The groups discussed here fit into the diagram
\[
\xymatrix{
& 0 \ar[d] & 0 \ar[d] & & \\
0 \ar[r] & \OO_{D_n}^\ast \ar[r] \ar[d] & D_n^\ast \ar[r]^{\height}\ar[d] & \Z \ar[r] & 0 \\
0 \ar[r] &  \G_n \ar[r] \ar[d] & \G_n^0 \ar[r]^{\height} \ar[d] & \Z \ar[r] & 0\\
& \Gamma_{\F_p} \ar[d] \ar[r]_{=} & \Gamma_{\F_p} \ar[d] & & \\
& 0 & 0 &&
}
\]
in which rows and columns are exact, and $\Gamma_{\F_p}=\Gal(\overline{\F}_p/\F_p)$.  

The following definition is a special case of the constructions appearing in \cite{RapoportZink}. 

\begin{defn} Let $\RZ_n$ be the functor which inputs a ring $R$ in which $p$ is nilpotent and outputs the set of isomorphism classes of triples $(G,\iota,\rho)$, where $G/R$ is a $p$-divisible group, $\iota$ is a ring homomorphism $\overline{\F}_p\to R/p$, and 
\[ \rho\from G_n \otimes_{\overline{\F}_p,\iota} R/p\to G\otimes_R R/p \]
is a quasi-isogeny. 
\end{defn}

Let $W=W(\overline{\F}_p)$ be the ring of Witt vectors.  The homomorphism $\iota$ determines a homomorphism $W\to R$, so that there is a natural transformation $\RZ_n\to \Spf W$.  $\RZ_n$ may be viewed as a sheaf on the Zariski site of $\Spf W$, and as such there is a decomposition:
\[\RZ_n=\coprod_{h\in\Z} \RZ_n^{(h)},\] where $\RZ_n^{(h)}$ classifies those $(H,\iota,\rho)$ where $\rho$ has constant height $h$.  

There is an action of the group $\G_n^0$ on $\RZ_n$ lying over the action of $\Gamma_{\F_p}$ on $\Spf W$:    An element $g=(\tau,u)\in \G_n^0$ sends $(G,\iota,\rho)$ to $(G,\iota\circ \tau^{-1},\rho\circ u^{-1})$.  Under this action, an element $g\in \G_n^0$ induces an isomorphism $\RZ_n^{(h)}\isomto \RZ_n^{(h-\height(g))}$.  Therefore all components $\RZ_n^{(h)}$ are isomorphic, and $\RZ_n$ is isomorphic to the ``induction'' $(\RZ_n^{(0)}\times \G_n)/\G_n^0$.  

Recall that $A_n$ is the universal deformation ring of $H_n$.  Let $\mathfrak{m}$ be the maximal ideal of $A_n$, and let $H_n^{\univ}$ be the universal deformation of $H_n$, so that there is an isomorphism $H_n\isom H_n^{\univ} \otimes_{A_n} A_n/\mathfrak{m}$.  Let $G_n^{\univ}=H_n^{\univ}[p^\infty]$, so that we have an isomorphism 
\[  G_n\isom G_n^{\univ}\otimes_{A_n} A_n/\mathfrak{m}. \]
For each $k\geq 1$, we apply \cref{PropRigidityOfQIsog} over the ring $A_n/(p,\mathfrak{m}^k)$;  the above isomorphism lifts to a quasi-isogeny
\[ \rho_k\from G_n\otimes_{A_n} A_n/(p,\mathfrak{m}^k) \to G_n^{\univ}\otimes_{A_n} A_n/(p,\mathfrak{m}^k) \]
of height 0.  Thus one obtains a morphism $\Spec A_n/\mathfrak{m}^k\to \RZ_n^{(0)}$ for all $k\geq 1$, or all the same, a morphism $\LT_n=\Spf A_n\to \RZ_n^{(0)}$.  This morphism is $\G_n$-equivariant.  In fact:

\begin{prop}[{\cite[Proposition 3.79]{RapoportZink}}]  
\label{PropRZvsLT} 
    The morphism $\LT_n\to \RZ_n^{(0)}$ is an isomorphism.
\end{prop}

We can now complete the proof of  \cref{PropUniformizationByLT}, which claims that $\widehat{\M}_{\FG}^n\isom [\LT_n/\G_n]$.

\begin{proof}[Proof of \cref{PropUniformizationByLT}] The universal deformation of $H_n$ determines a $\G_n$-equivariant morphism $\LT_n\to \widehat{\M}_{\FG}^n$.  Composing with the isomorphism in \cref{PropRZvsLT}, we obtain a $\G_n$-equivariant morphism $\RZ_n^{(0)}\to \widehat{\M}_{\FG}^n$.  As we observed above, $\RZ_n\isom (\RZ_n^{(0)}\times \G_n)/\G_n^0$,
so this morphism extends uniquely to a $\G_n^0$-equivariant morphism $\RZ_n\to \widehat{\M}_{\FG}^n$.  To complete the proof, we need to show that $\RZ_n\to \widehat{\M}_{\FG}^n$ is a pro-\'etale $\G_n^0$-torsor.

Suppose $R$ is a ring in which $p$ is nilpotent, and that we are given a morphism $\Spec R\to \widehat{\M}_{\FG}^n$, corresponding to a formal group $H$ over $R$
such that $H\otimes_R R/I$ has height $n$ for a nilpotent ideal $I$ containing $p$. The fiber $F:=\Spec R\times_{\widehat{\M}_{\FG}^n} \RZ_n$ is the functor which sends an $R$-algebra $S$ to the set of pairs $(\iota,\rho)$, where $\iota\from \overline{\F}_p\to S/p$ is a ring homomorphism, and 
\[ \rho \from G_n\otimes_{\overline{\F}_p,\iota} S/p \to H[p^\infty]\otimes_R S/p \]
is a quasi-isogeny.  Our objective is to show that $F$ is a pro-\'etale $\G_n^0$-torsor over $\Spec R$.

First we show that $F\to\Spec R$ has pro-\'etale local sections.  Let $R'=R\otimes_{\Z_p} W$, so that $\Spec R'\to \Spec R$ is pro-\'etale.  By \cref{LemmaRigidity}, 
\[ \underline{\Isom}(H_n\otimes_{\overline{\F}_p} R'/I, H\otimes_R R'/I) \]
is represented by an affine scheme which is pro-finite \'etale and surjective over $\Spec R'/I$;  let $\Spec S$ be the lift of this affine scheme to $\Spec R'$.    
Let $H_n\otimes_{\overline{\F}_p} S/I\to H\otimes_R S/I$ be the universal isomorphism. By \cref{PropRigidityOfQIsog} applied to the ring $S/p$, the associated isomorphism of $p$-divisible groups lifts to a quasi-isogeny $G_n\otimes_{\overline{\F}_p} S/p\to H[p^\infty]\otimes_R S/p$ of height 0, determining a morphism $\Spec S\to F$.  Therefore $F\to \Spec R$ admits pro-\'etale local sections.

Any two sections of $F$ over an $R$-algebra $S$ differ by a pair $(\tau,u)$, where $\tau\in \Gamma_{\F_p}$ and $u\from G_n\otimes_{\overline{\F}_p} S\to \tau^*G_n\otimes_{\overline{\F}_p} S$ is a quasi-isogeny.  (Here $\tau^*G_n$ means $G_n\otimes_{\overline{\F}_p,\tau}\overline{\F}_p$.)  The quasi-isogeny $u$ represents an $S$-point of the $\overline{\F}_p$-scheme $\QIsog(G_n,\tau^*G_n)$.  By \cref{PropRigidityOfQIsog}, $\QIsog(G_n,\tau^*G_n)$ is formally \'etale over $\overline{\F}_p$, which (since the latter is algebraically closed) must be the constant scheme associated to the set of quasi-isogenies $G_n\to\tau^*G_n$.  Thus any two sections of $F$ differ by an $S$-point of the constant scheme associated with $\G_n^0$. We have shown that $F\to \Spec R$ is a $\G_n^0$-torsor, which completes the proof.
\end{proof}

\subsection{Perfectoid spaces, diamonds, and v-stacks}

The stack $\widehat{\M}_{\FG}^n$ is fibered over $\Spf \Z_p$.  The next step is to pass to the generic fiber to obtain an object living over $\Q_p$.  The generic fiber of general Rapoport-Zink spaces was studied in \cite{ScholzeWeinstein2013} and \cite{ScholzeLectures}, which found an alternate description not in terms of $p$-divisible groups, but rather in terms of certain linear-algebra data called {\em mixed-characteristic shtukas}.  We tell the story in a way adapted to our purposes, starting with some recollections from \cite{ScholzeDiamonds} on perfectoid spaces and v-stacks. 

Let $\Perf$ denote the category of perfectoid spaces in characteristic $p$.  Then $\Perf$ carries two important topologies:  the {\em pro-\'etale topology} and the finer {\em v-topology}.   These are the respective analogues of the \'etale and fppf topologies on the category of schemes.  Importantly, a representable presheaf on $\Perf$ is a sheaf for the v-topology \cite[Theorem 1.2]{ScholzeDiamonds}.  We write $\Spd \F_p$ for the final object in the category of v-sheaves, and $\Spd \overline{\F}_p$ for the v-sheaf taking an object $S$ in $\Perf$ to $\Hom(S,\Spa \overline{\F}_p)$ in the category of adic spaces.

A {\em diamond} is a v-sheaf which admits a presentation as a quotient of a representable sheaf by a pro-\'etale equivalence relation. There is a functor $X\mapsto X^{\diamond}$ from analytic adic spaces over $\Z_p$ to diamonds \cite[Theorem 1.5]{ScholzeDiamonds}.  An important example is $X=\Spa\Q_p$, in which case $X^\diamond=\Spd \Q_p$ takes an object $S$ in $\Perf$ to the set of isomorphism classes of untilts $S^\sharp$ in characteristic 0 \cite[Theorem 9.4.4]{ScholzeLectures}. 

A {\em v-stack} is a sheaf of groupoids on the v-site of $\Perf$.  A v-stack fibered over $\Spd\Q_p$ is the same thing as a sheaf of groupoids on the v-site of perfectoid spaces over $\Q_p$.

We introduce notation extending this construction to formal schemes.

\begin{defn} Suppose $X$ is a presheaf of groupoids on the category of adic $\Z_p$-algebras.  The {\em diamond generic fiber} $X^{\diamond}_\eta$ of $X$ is the v-stack over $\Spd\Q_p$ defined as the v-sheafification of the presheaf which has the value 
\[ (R,R^+)\mapsto X(R^+) \]
on affinoid perfectoid $\Q_p$-algebras.
\end{defn}
Thus for a perfectoid space $S$ over $\Q_p$, an object of $X_{\eta}^{\diamond}(S)$ is given by a v-cover of $S$ by affinoids $\Spa(R_i,R^+_i)$, and for each $i$ an object of $X(R^+_i)$, together with a descent datum to $S$.

Passage to the diamond generic fiber commutes with pro-\'etale quotients, as the following lemma shows.

\begin{lemma}
\label{LemFormalToDiamond}
Let $X$ be a formal scheme over $\Spf \Z_p$ admitting a continuous action of a profinite group $G$, and let $[X/G]$ denote the quotient stack in the pro-\'etale topology.  There is a natural isomorphism of v-stacks $[X/G]^{\diamond}_{\eta}\isom [X^{\diamond}_{\eta}/G].$
\end{lemma}

\begin{proof}  A v-stack is determined by its values on $S$-points, where $S=\Spa(R,R^+)$ is strictly totally disconnected \cite[Lemma 1.16]{ScholzeDiamonds}.  In our situation, the v-stacks in question are fibered over $\Spd \Q_p$, so we may assume these $S$ live in characteristic 0.  For such an $S$, consider a $\Spf R^+$-point of $[X/G]$:  this means a pro-\'etale $G$-torsor $\tilde{R}^+/R^+$ in adic $\Z_p$-algebras together with a $G$-equivariant map $\Spf \tilde{R}^+\to X$.  Letting $\tilde{R}=\tilde{R}^+[1/p]$ and $\tilde{S}=\Spa(\tilde{R},\tilde{R}^+)$, we find a pro-\'etale $G$-torsor in perfectoid spaces $\tilde{S}\to S$, and a $G$-equivariant map $\tilde{S}\to X^{\diamond}_{\eta}$.  The defining property of strictly totally disconnected spaces is that every pro-\'etale $\tilde{S}\to S$ is split, and therefore the $G$-torsor is split:  $\tilde{S}=S\times G$.  So in fact we have a map $S\to X^{\diamond}_{\eta}$, well-defined up to translation by $G(R^+)=G(S)$. 

The diamond generic fiber $[X/G]_{\eta}^{\diamond}$ is therefore the v-sheafification of the presheaf sending a strictly totally disconnected $S$ to $X^{\diamond}_{\eta}(S)/G(S)$, which is to say it is isomorphic to $[X^{\diamond}_{\eta}/G]$.
\end{proof}

\begin{defn} \label{def:diamondstack}
Let 
    \[ 
        \SW_n = \left(\widehat{\M}_{FG}^n\right)^{\diamond}_{\eta}, 
    \]
and call this the {\em v-stack of 1-dimensional formal groups of height $n$}.  
\end{defn}

Applying Lemma \ref{LemFormalToDiamond} to the isomorphism in Proposition \ref{PropUniformizationByLT} gives an isomorphism 
of v-stacks over $\Spd \Q_p$:
\[ \SW_n\isom [\LT_{n,\eta}^{\diamond}/\G_n]. \] 

\subsection{The Fargues--Fontaine curve, and the moduli space of local shtukas}
The theory of the Fargues--Fontaine curve was introduced in \cite{FarguesFontaine} in the absolute setting, and developed in \cite{KedlayaLiu} in the relative setting.  Further preparations for this material may be found in \cite{ScholzeLectures} and \cite{FarguesScholze}.  

Let $S=\Spa(R,R^+)$ be an affinoid object of $\Perf$, with pseudo-uniformizer $\varpi\in R^+$.  Let $Y_S$ be the open locus in $\Spa W(R^+)$ where both $p$ and $[\varpi]$ are nonzero.  The {\em Fargues--Fontaine curve} $X_S$ is defined as the quotient of $Y_S$ by the Frobenius automorphism of $S$.  Then $X_S$ is an analytic adic space over $\Q_p$ \cite[Proposition II.1.1]{FarguesScholze}.  The construction of $X_S$ extends to general objects $S$ of $\Perf$ \cite[Proposition II.1.3]{FarguesScholze}.  Given an untilt $S^\sharp$ of $S$ in characteristic 0, which is to say a morphism $S\to \Spd \Q_p$, there is an associated closed immersion $i\from S^\sharp\to X_S$ which presents $S^\sharp$ as a Cartier divisor of $X_S$ \cite[Proposition II.1.4]{FarguesScholze}.

In the case that $C$ is an algebraically closed perfectoid field of characteristic $p$ with ring of integers $C^\circ$, one refers to $X_C:=X_{\Spa(C,C^\circ)}$ as an absolute curve. 
A major accomplishment of \cite{FarguesFontaine} is the classification of vector bundles over $X_C$, by means of a Harder--Narasimhan formalism.  Every vector bundle $\E$ over $X_C$ is isomorphic to a direct sum of vector bundles $\OO(\lambda)$ parameterized by rational numbers $\lambda$, known as the slopes of $\E$.  We call $\E$ isoclinic if it has only one slope.  For each $\lambda$, the degree and rank of $\OO(\lambda)$ are the numerator and denominator of $\lambda$ in lowest terms, and the endomorphism algebra of $\OO(\lambda)$ is a central division algebra over $\Q_p$ of invariant $\lambda$.  (In particular $\End \OO=\Q_p$.) 

We now turn to the relative setting.  Define a morphism $\Bun_n\to \Spd \overline{\F}_p$ as follows:  Given a perfectoid space $S$ over $\overline{\F}_p$, let $\Bun_n(S)$ be the groupoid of vector bundles of rank $n$ over $X_S$.  For a rational number $\lambda$ whose denominator $d\geq 1$ (in lowest terms) is a divisor of $n$, let $\Bun_n^{\lambda}(S)\subset \Bun_n(S)$ classify those vector bundles which are isoclinic of slope $\lambda$ at every geometric point of $S$.

\begin{prop}[{\cite[Theorem 8.5.12]{KedlayaLiu}, \cite[I.4.1(v)]{FarguesScholze}}] \label{PropIsoclinic} Let $n\geq 0$.
\begin{enumerate}
    \item $\Bun_n$ is a v-stack.  
    \item Let $\lambda$ be a rational number whose denominator $d\geq 1$ is a divisor of $n$.  Then $\Bun_n^{\lambda}$ is isomorphic to the classifying stack $[\Spd \overline{\F}_p/J_\lambda]$, where $J_{\lambda}$ is the group of nonzero elements of the unique central simple algebra over $\Q_p$ of degree $n$ and invariant $\lambda$.  
\end{enumerate}
\end{prop}

We spell out the consequences of Proposition \ref{PropIsoclinic} in the two cases relevant to us.  Note that for a perfectoid space $S$ over $\overline{\F}_p$ and a profinite group $\Gamma$, an $S$-point of the classifying stack $[\Spd \overline{\F}_p/\Gamma]$ (commuting with the structure maps to $\Spd \overline{\F}_p$) is the same as a pro-\'etale $\Gamma$-torsor on $S$.

In the case $\lambda=0$, we have $J_\lambda=\GL_n(\Q_p)$.  Note that a pro-\'etale $\GL_n(\Q_p)$-torsor is the same thing as a $\Q_p$-local system of rank $n$.  If $T$ is a $\Z_p$-local system of rank $n$, let us write $T\otimes_{\Z_p} \OO_{X_S}$ for the vector bundle over $X_S$ corresponding to the $\Q_p$-local system $T[1/p]$.  

In the case $\lambda=1/n$, we have that $J_\lambda=D_n^\ast$ is a subgroup of the extended Morava stabilizer group $\G_n^0$ described in \S\ref{SecLTRZ}.  As a corollary of Proposition \ref{PropIsoclinic}, we have the following statement for perfectoid spaces $S$ over $\F_p$:  There is an equivalence of groupoids between $S$-points of $[\Spd \overline{\F}_p/\G_n^0]$ and rank $n$ vector bundles over $X_S$ which have slope $1/n$ at every geometric point of $S$.

Let $K=W(\overline{\F}_p)[1/p]$.  We define a morphism $\Sht_n \to \Spd K$ of sheaves on $\Perf$ as follows.  Given a morphism $S\to \Spd K$, which is to say a perfectoid space $S$ in characteristic $p$ equipped with an untilt $S^\sharp$ over $K$, let $\Sht_n(S)$ be the set of modifications 
\begin{equation}
    \label{DefnShtukas}
    0\to T\otimes_{\Z_p}\OO_{X_S}\to \OO_{X_S}(1/n)\to i_*L\to 0,
\end{equation}
    where $T$ is a pro-\'etale $\Z_p$-local system of rank $n$ on $S$, and $L$ is a locally free $\OO_{S^\sharp}$-module of rank 1.  
Then $\Sht_n$ is an example of a local shtuka space as introduced in \cite{ScholzeLectures}.  There is an action of $\G_n^0$ on $\Sht_n$ lying over the action on $\Spd K$.  

The following theorem relates $p$-divisible groups to shtukas. 

\begin{thm}
\label{ThmShtukas} There is a natural $\G_n^0$-equivariant isomorphism of diamonds $\RZ_n^\diamond\to \Sht_n$ commuting with the maps to $\Spd K$.  
\end{thm}

This is a special case of \cite[Theorem 24.2.5]{ScholzeLectures}.  We only comment to identify the objects $T$ and $L$ in the shtuka corresponding to a triple $(G,\iota,\rho)$ over $R^+$, when $S^\sharp=\Spa(R,R^+)$ is a perfectoid space in characteristic 0:  The pro-\'etale $\Z_p$-local system $T$ is the Tate module of the \'etale $p$-divisible group $G\otimes_{R^+} R$, and $L=\Lie G \otimes_{R^+} R$.

\subsection{The isomorphism between the Lubin--Tate and Drinfeld towers}

Recall that $\SW_n$ is the $v$-stack of 1-dimensional formal groups of height $n$.  The following theorem is an equivalent way of stating the ``isomorphism between the Lubin--Tate and Drinfeld towers'' appearing in \cite{ScholzeWeinstein2013}.  

\begin{thm}
\label{ThmEquivalenceOfStacks} There are isomorphisms between $\SW_n$ and the following v-stacks:
    \begin{enumerate}
        \item The quotient stack $[\LT_{n,\eta}^{\diamond}/\G_n]$;
        \item the quotient stack $[\HH^{n-1,\diamond}/\GL_n(\Z_p)]$.
    \end{enumerate}
    Here, $\LT_{n,\eta}$ is the rigid-analytic fiber of $\LT_n$, and $\HH^{n-1}$ is Drinfeld's symmetric space, defined as the complement in rigid $\PP^{n-1}_{\Q_p}$ of all $\Q_p$-rational hyperplanes.  

With respect to the isomorphism $
[\LT_{n,\eta}^{\diamond}/\G_n] \isom 
[\HH^{n-1,\diamond}/\GL_n(\Z_p)]$, the following two invertible sheaves are identified:
\begin{enumerate}
    \item[(a)] $(\Lie H_n^{\univ})^{\diamond}_{\eta}$, the diamond generic fiber of the Lie algebra of the universal deformation $H_n^{\univ}$ over $[\LT_n/\G_n]$. 
    \item[(b)] $\OO(-1)^{\diamond}\otimes_{\Q_p} \Q_p(-1)$, where $\OO(-1)$ is the restriction to $\HH^{n-1}$ of the dual tautological bundle on the ambient projective space, $\OO(-1)^{\diamond}$ represents the descent of the diamond version of this module to the quotient $[\HH^{n-1,\diamond}/\GL_n(\Z_p)]$, and $\Q_p(-1)$ is a Tate twist.  
\end{enumerate}
\end{thm}

\begin{proof}   
Combining \cref{PropRZvsLT} with \cref{ThmShtukas} gives isomorphisms
\[ \SW_n\isom [\LT_{n,\eta}^{\diamond}/\G_n] \isom [\RZ_{n,\eta}^{\diamond}/\G_n^0] \isom [\Sht_n/\G_n^0]. \] 
Given a perfectoid space $S$, the $S$-points of the v-stack $[\Sht_n/\G_n^0]$ classify untilts $S^\sharp$ together with modifications    
\[  
                 0\to T\otimes \OO_{X_S} \to \mathcal{E}\to i_*L \to 0,
             \]
as in \eqref{DefnShtukas}, except that $\mathcal{E}$ is a variable vector bundle of rank $n$ which is isoclinic of slope $1/n$ at every geometric point of $S$. 

There is a morphism $[\Sht_n/\G_n^0]\to [\Spd \overline{\F}_p/\GL_n(\Z_p)]$ which outputs the \'etale $\Z_p$-local system $T$.  The pullback of this morphism along the $\overline{\F}_p$-point represented by the trivial local system classifies modifications
\begin{equation}
    \label{EqModif}
 0\to \OO_{X_S}^{\oplus n} \to \E\to i_*L \to 0.
 \end{equation}
 We will show that such modifications correspond to $S$-points of $\HH^{n-1,\diamond}$, and then quotienting by $\GL_n(\Z_p)$ would give the equivalence between (1) and (2).

Modifications as in \eqref{EqModif} are in duality with modifications 
\begin{equation}
\label{EqModif2}
0 \to \check{\E}\to \OO_{X_S}^{\oplus n} \to i_*\ell \to 0, 
\end{equation}
where $\check{\E}=\Hom(\E,\OO_{X_S})$, and $\ell$ is a locally free $\OO_{S^\sharp}$-module.  Disregarding any condition on $\E$, these are parameterized by $S$-points of $\PP^{n-1,\diamond}$.  Indeed, an $S$-point of $\PP^{n-1,\diamond}$ is the same as an $S^\sharp$-point of $\PP^{n-1}$, and this is in turn the same as a  locally free $\OO_{S^\sharp}$-module $\ell$ together with a surjective map $\OO_{S^{\sharp}}^{\oplus n}\to \ell$.  Such a map determines an exact sequence as in \eqref{EqModif2}.

Finally, we claim that $\E$ is isoclinic of slope $1/n$ at every point of $S$ if and only if the corresponding map $S\to \PP^{n-1,\diamond}$ factors through $\HH^{n-1,\diamond}$.  For this it suffices to assume that $S=\Spa(C,C^\circ)$ is a geometric point, so that $C^\sharp/\Q_p$ is an algebraically closed perfectoid field, and $\ell$ is a 1-dimensional quotient of $(C^\sharp)^{\oplus n}$.  The following statements are equivalent:
\begin{enumerate}
    \item $\E$ has 0 as a slope.
    \item $\check{\E}$ has 0 as a slope.
    \item $\check{\E}$ admits $\OO_{X_C}$ as a summand.
    \item There is a morphism $\OO_{X_C}\to \OO_{X_C}^{\oplus n}$, determined by a nonzero vector $v\in \Q_p^{\oplus n}$, whose composition with $\OO_{X_C}^{\oplus n} \to i_*\ell$ is zero.
    \item The image of the morphism $\Spa(C,C^\circ)\to \PP^{n-1,\diamond}$ lies in the $\Q_p$-rational hyperplane defined by the condition that $(C^\sharp)^{\oplus n}\to \ell$ vanishes on some nonzero vector $v\in \Q_p^{\oplus n}$.   
\end{enumerate}
Negating these conditions, we find that $S\to \PP^{n-1,\diamond}$ factors through $\HH^{n-1,\diamond}$ if and only if $\E$ has only positive slopes.  Since $\E$ has degree 1 and rank $n$, the only possibility is that $\E$ is isoclinic of slope $1/n$.

In the duality between modifications \eqref{EqModif} and \eqref{EqModif2}, the relationship between the locally free $\OO_{S^\sharp}$-modules $L$ and $\ell$ is
\[ \ell = \check{L} \otimes T_{S^\sharp},\] 
where $T_{S^\sharp}$ is the tangent space of $X_S$ along $i\from S^\sharp\to X_S$;  this relation holds when dualizing modifications of vector bundles over any regular curve.  For $S=\Spa(R,R^+)$ affinoid, the completed local ring of $\OO_{X_S}$ along $S^\sharp$ is the de Rham period ring $B_{\mathrm{dR}}^+(R^\sharp)$, say with uniformizing parameter $(\xi)$;  then $T_{S^\sharp}$ is the $R^\sharp$-linear dual of $(\xi)/(\xi^2) \isom R^\sharp \otimes_{\Q_p}\Q_p(1)$.  

When the shtuka in \eqref{EqModif} arose from a formal group $H$ over $R^{\sharp+}$, the invertible sheaf $L$ has been identified with $\Lie H[1/p]$, cf. the comment following Theorem \ref{ThmShtukas}.  On the other hand, $\ell$ is the pullback to $S$ of the tautological bundle $\OO(1)^{\diamond}$ through $S\to\PP^{n-1,\diamond}$.  Taking into account the relation between $L$ and $\ell$, we find the relation between the invertible modules claimed in the theorem.
\end{proof}

\subsection{The determinant map $\SW_n\to \SW_1$.}

To access the determinant sphere $\mathbb{S}_{K(n)}[\det]$ as a Picard class on $\SW_n$, we will need to understand the role of the determinant in the study of 1-dimensional formal groups.  

\begin{thm} \label{ThmDeterminant} 
There exists a determinant morphism
\[ \det \from \SW_n\to \SW_1 \]
making the following diagram commute:
    \[
        \xymatrix{
        [\LT_{n,\eta}^{\diamond}/\G_n] \ar[rr]^-{\isom} \ar[d]_{\det_{\LT}} && \SW_n \ar[rr]^-{\isom} \ar[d]_{\det} && [\HH^{n-1,\diamond}/\GL_n(\Z_p) ] \ar[d]^{\det_{\HH}} \\
        [\LT_{1,\eta}^{\diamond}/\G_1] \ar[rr]_-{\isom} && \SW_1 \ar[rr]_-{\isom} && [\HH^{0,\diamond}/\Z_p^\ast]. 
        }
    \]
In the diagram:
\begin{enumerate}
    \item the map labeled $\det_{\LT}$ is induced from:
    \begin{enumerate}
        \item the structure map $\LT_{n,\eta}\to \LT_{1,\eta}\isom \Spa(K,W)$;
        \item the determinant map $\G_n\to \G_1$, explained in the remarks below,
    \end{enumerate}
    \item the map labeled $\det_{\HH}$ is induced from:
    \begin{enumerate}
        \item the structure map $\HH^{n-1}\to \HH^0\isom \Spa(\Q_p,\Z_p)$;
        \item the determinant map $\GL_n(\Z_p)\to \Z_p^\ast$.
    \end{enumerate}
    \end{enumerate}
\end{thm}

Before beginning the proof, we offer some remarks.  For a perfectoid space $S$ over $\Spa(\Q_p,\Z_p)$, the $S$-points of the morphism $\det\from \SW_n\to \SW_1$ should input a formal group $H$ of dimension 1 and height $n$ over $S$, and output a formal group $\bigwedge^nH$ over $S$ of dimension 1 and height 1. The crucial computation that makes this work comes from the case of Dieudonn\'e modules over $\overline{\F}_p$:  if $M(G_n)$ is the Dieudonn\'e module of $G_n$, then the top exterior power $\bigwedge_{W(\overline{\F}_p)}^n M(G_n)$ is isomorphic to $M(G_1)$.  

Considering the action of $\G_n=\Aut(G_n,\overline{\F}_p)$ on $\bigwedge_{W(\overline{\F}_p)}^n M(G_n)\isom M(G_1)$, we find a determinant map
    \[ 
        \det\from \G_n\to \G_1.
    \]
This map is isomorphic to the profinite completion of the reduced norm map $D_n^\ast\to \Q_p^\ast$, where $D_n=\End G_n \otimes_{\Z_p}\Q_p$.  

Similar considerations on the Fargues--Fontaine curve give an isomorphism $\bigwedge^n_{\OO_{X_S}}\OO(1/n)\isom \OO(1)$ for any perfectoid space $S$ over $\overline{\F}_p$, which is compatible with $\det\from \G_n\to\G_1$. 

\begin{rem} One might wonder whether general exterior powers exist beyond the determinant.  
    For a general $p$-divisible group $G$ of height $n$ over a perfect field $k$ of characteristic $p$, the exterior powers $\bigwedge^i_{W(k)} M(G)$ are Dieudonn\'e modules if and only if $\dim G\leq 1$, so one can only expect exterior powers of $G$ itself to exist in this case.  A  construction of exterior powers of $p$-divisible groups of dimension $\leq 1$ over the ring of integers in a nonarchimedean local field appeared previously for odd primes $p$ in \cite{Hedayatzadeh} and extended to $p=2$ in \cite{hopkins2013ambidexterity}.
\end{rem}

\begin{proof}[Proof of \cref{ThmDeterminant}] 

As we saw in the proof of \cref{ThmEquivalenceOfStacks}, there is an isomorphism of v-stacks $\SW_n\isom [\Sht_n/\G_n^0]$.  For a perfectoid space $S$ over $\overline{\F}_p$, the $S$-points of $[\Sht_n/\G_n^0]$ are modifications of vector bundles on $X_S$:
\[ 0 \to T\otimes_{\Z_p}\OO_{X_S} \stackrel{\alpha}{\to} \E \to i_* L\to 0,\]
where $T$ is a pro-\'etale $\Z_p$-local system of rank $n$, $\E$ is a vector bundle which is isoclinic of slope $1/n$ at every geometric point of $S$, $i$ is the inclusion of $S^\sharp$ into $X_S$ for some untilt $S^\sharp$ of $S$ to characteristic 0, and $L$ is a locally free $\OO_{S^\sharp}$-module of rank 1.  Given such data, one obtains an $S$-point of $[\Sht_1/\G_1^\circ]\isom \SW_1$ by taking determinants:
\[ 0\to \bigwedge^n_{\Z_p} T \otimes_{\Z_p} \OO_{X_S} \stackrel{\det\alpha}{\to} \bigwedge^n_{\OO_{X_S}} \E\to i_*i^*\bigwedge^n_{\OO_{X_S}} \E \to 0. \] 
(Explanation:  $\alpha$ has a zero along $i$, so the quotient map $\bigwedge^n_{\OO_{X_S}}\E\to \coker \det \alpha$ factors through a surjective map $\beta\from i_*i^*\bigwedge^n_{\OO_{X_S}} \E \to \coker \det\alpha$.  
On the other hand, since $\det\alpha$ is a map between line bundles on $X_S$ which are degree 0 and 1 at every geometric point of $S$, $\coker \det\alpha$ must be length 1 at every geometric point of $S$. Since $i_*i^*\bigwedge^n_{\OO_{X_S}} \E$ is already length 1 at every geometric point of $S$,  this forces $\beta$ to be an isomorphism.)

Thus we have constructed a morphism $\det\from \SW_n\to \SW_1$.  The commutativity of the right square of the diagram is by construction:  both compositions output the $S$-point of $[\HH^{0,\diamond}/\Z_p^\ast]\isom \Spd(\Q_p,\Z_p) \times [\Spd \F_p/\Z_p^\ast]$ defined by the untilt $S^\sharp$ and the pro-\'etale $\Z_p^\ast$-torsor $\bigwedge_{\Z_p}^n T$.  The commutativity of the left square amounts to the fact noted above, that there is an isomorphism $\bigwedge_{\OO_{X_S}}\OO(1/n) \isom \OO(1)$ compatible with the determinant map $\det\from \G_n\to\G_1$.  
\end{proof}

\section{The fundamental exact sequence}\label{sec:exactsequence}

The goal of this section is to construct what we dub the \emph{fundamental exact sequence}, describing the group $H^1_{\cts}(\G_n,A_n^{\ast\ast})$ in terms of $H^1_{\cts}(\G_1,A_1^{\ast\ast})$ and $H^1_{\proet}(\HH_C^{n-1},\OO^{\ast\ast})^{\Pi_n}$, where $\Pi_n=\Gal(\overline{\Q}_p/\Q_p)\times \GL_n(\Z_p)$.  (We intend no confusion with the fundamental exact sequence of $p$-adic Hodge theory appearing in \cite{FarguesFontaine}.) 

The key inputs are the isomorphism of towers \cref{ThmEquivalenceOfStacks}, the determinant map, as well as a logarithm exact sequences for the sheaf of principal units (\cref{lem:log_ses}) to compute the terms. Throughout, special care is needed for the ``anomalous'' case $n=p=2$. 

The objects of study in this section are closely related to those considered by Ertl, Gilles, and Nizio{\l} in \cite{VPicardGroup}, whose methods inspired the approach in this section. 

\subsection{Pro-\'etale sheaves and condensed abelian groups}

Suppose that $X$ is a v-stack.  Let us say that a morphism of v-stacks $Y\to X$ is pro-\'etale (resp., a pro-\'etale cover) if its base change through a map $U\to X$ from a representable object is also pro-\'etale (resp., a pro-\'etale cover).  
Define the pro-\'etale site $X_{\proet}$ as follows:  the objects are pro-\'etale morphisms $Y\to X$, and the covers are pro-\'etale covers.  Let us assume the existence of a pro-\'etale cover $Y\to X$ from a representable object $Y$: then $X_{\proet}$ admits a basis consisting of pro-\'etale morphisms $U\to X$ from affinoid perfectoid $U$.

We will encounter various sheaves of sets or abelian groups on $X_{\proet}$, for instance:
\begin{itemize}
    \item The sheaf $h_Y$ on $X_{\proet}$ represented by a diamond $Y$.
    \item The structure sheaves $\OO_X^+\subset \OO_X$, whose sections over an affinoid perfectoid $U=\Spa(R,R^+)$ are $R^+\subset R$.
    \item When $X$ is fibered over $\Spd \Q_p$, any affinoid perfectoid $U=\Spa(R,R^+)$ in $X_{\proet}$ has an untilt $U^\sharp=\Spa(R^\sharp,R^{\sharp +})$, and then we may define sheaves $\OO_X^{\sharp +}\subset \OO_X^{\sharp}$, whose value on $U$ is $R^{\sharp+}\subset R^{\sharp}$.  
\end{itemize}

An interesting feature of such sheaves is that they automatically carry a condensed structure.  This because there is a morphism of sites
\[ \lambda\from X_{\proet}\to \ast_{\proet} \]
onto the pro-\'etale site of a point, whose objects are profinite sets $S$;  the morphism $\lambda$ pulls back $S$ to the object $X\times S$ of $X_{\proet}$.  As a result, if $\mathcal{F}$ is a sheaf of sets on $X_{\proet}$, then $\Gamma_{\cond}(X_{\proet},\mathcal{F}):=\lambda_*\mathcal{F}$ lies in the category $\Cond$ of condensed sets (= sheaves of sets on $\ast_{\proet}$).  Explicitly, if $S$ is a profinite set, we can describe the $S$-points: $\Gamma_{\cond}(X_{\proet},\mathcal{F})(S)=\Gamma((X\times S)_{\proet},\mathcal{F})$. 

If $\mathcal{F}$ is a sheaf of abelian groups on $X_{\proet}$, then $H^0_{\cond}(X_{\proet},\mathcal{F})=\lambda_*\mathcal{F}$ lies in $\Cond(\Ab)$, the category of condensed abelian groups.  Once again, the $S$-points are $H^0((X\times S)_{\proet},\mathcal{F})$; note that this can also be described as $\Hom(\Z[S],H^0(X_{\proet},\mathcal{F}))$, where $\Z[S]$ is the free condensed abelian group on $S$.  The higher cohomology groups of $\mathcal{F}$ also carry a condensed structure:   Let $R\Gamma_{\cond}(X,\mathcal{F}):=R\lambda_*\mathcal{F}$, an object of the derived category $D(\Cond(\Ab))$, and then each cohomology group $H^i_{\cond}(X,\mathcal{F}):=R^i\lambda_*\mathcal{F}$ is condensed.
We can describe an object $A$ of $D(\Cond(\Ab))$ in terms of its ``derived $S$-points'' $A(S):=\RHom(\Z[S],A)$, for any profinite $S$.  The derived $S$-points of $R\Gamma_{\cond}(X,\mathcal{F})$ are
\begin{equation}
    \label{EqDerivedSPoints}
  \RHom(\Z[S],R\Gamma_{\cond}(X_{\proet},\mathcal{F})) = R\Gamma((X\times S)_{\proet},\mathcal{F}).  
\end{equation}
It may happen that $\mathcal{F}$ is already a sheaf of topological abelian groups on $X_{\proet}$.  In that case one can package $\mathcal{F}$ together with its topology into a sheaf $\mathcal{F}_{\cond}$ of condensed abelian groups on $X_{\proet}$, and then $R\Gamma(X,\mathcal{F}_{\cond})$ would be an object of $D(\Cond(\Ab))$, a priori distinct from $R\Gamma_{\cond}(X,\mathcal{F})$.  For $X$ fibered over $\Spd \Q_p$, the untilted integral structure sheaf $\OO_X^{\sharp+}$ is an example:  for $U=\Spa(R,R^+)$ an affinoid perfectoid object of $X_{\proet}$, we have $\OO_X^{\sharp+}(U)=R^{\sharp+}$, a topological abelian group.   In this case the two condensed enhancements agree \cite[Lemma 3.6.1]{BSSW}:  $R\Gamma_{\cond}(X_{\proet},\OO^{\sharp+})\isom R\Gamma(X_{\proet},\OO^{\sharp+}_{\cond})$.  

Recall the functor $X\mapsto X^{\diamond}$ from rigid-analytic adic spaces over $\Q_p$ to diamonds over $\Q_p$.  In this situation we have a map of sites
\[ \nu\from X_{\proet}^{\diamond}\to X_{\an},\]
where $X_{\an}$ is the analytic site of $X$, and a pullback map $\OO_X\to \nu_* \OO_{X^\diamond}^{\sharp}$ of sheaves of topological rings.  When $X$ is smooth, this map is an isomorphism \cite[Theorem 8.2.3]{kedlaya2019relativepadichodgetheory}.  (We remark that when $X=\Spa \Q_p$, this result states that the subfield of $\Gal(\overline{\Q}_p/\Q_p)$-invariants in $C=\widehat{\overline{\Q}}_p$ is $\Q_p$, which is \cite[Theorem 1]{TatePDivisibleGroups}.) 

\begin{lemma} \label{LemFullyFaithful} Let $X$ and $Y$ be rigid-analytic spaces over $\Q_p$, with $X$ smooth.  The map $\Hom_{\Q_p}(X,Y)\to \Hom_{\Spd \Q_p}(X^\diamond,Y^\diamond)$ is a bijection.  There is even an isomorphism of condensed enhancements: for any profinite $S$, the map $\Hom_{\Q_p}(X\times S,Y)\to \Hom_{\Spd\Q_p}(X^{\diamond}\times S,Y^{\diamond})$ is a bijection.
\end{lemma}

\begin{proof} The non-condensed statement is \cite[10.2.3]{ScholzeLectures};  we review it only to point out the condensed enhancement.  It suffices to assume that $X=\Spd(R,R^\circ)$ is affinoid, with $R$ a Tate algebra over $\Q_p$.  The result cited above states that $R=\OO_X(X)\to \OO_{X^\diamond}^{\sharp}(X^{\diamond})$ is an isomorphism of topological rings, and so for any profinite $S$, we have that  $\OO_X(X\times S)\isom C_{\cts}(S,\OO_X(X))\isom C_{\cts}(S,\OO_{X^\diamond}^{\sharp}(X^\diamond))\isom \OO_{X^\diamond}^{\sharp}(X^\diamond\times S)$ is an isomorphism.  Now apply the functor of continuous maps from $\OO_Y(Y)$ to obtain the result. 
\end{proof}

Next, we document the interaction between pro-\'etale cohomology and continuous group cohomology.  First, let us discuss the condensed version of group cohomology for a profinite group $G$.  Let $\mathcal{M}$ be a condensed $G$-module.  This means $\mathcal{M}$ is a condensed abelian group together with an action map $G\times \mathcal{M}\to \mathcal{M}$ in the category of condensed sets, satisfying the usual axioms.  For any profinite set $S$, let $\Z[S]$ be the free condensed abelian group on (the condensed set associated to) $S$.  Consider the standard resolution of the trivial $G$-module $\Z$ in condensed $\Z[G]$-modules:
\[ 
        \cdots\to\Z[G\times G]\to\Z[G].
    \]
Applying the functor $\Hom_G(-,\mathcal{M})$ in the category of condensed $G$-modules, we arrive at a complex of abelian groups $R\Gamma(G,\mathcal{M})$, whose $n$th term is $\Hom(\Z[G^n],\mathcal{M})\isom \mathcal{M}(G^n)$.  Note that if $\mathcal{M}$ is the condensed abelian group associated to a topological $G$-module $M$, then $\mathcal{M}(G^n)=C_{\cts}(G^n,M)$ is the group of continuous $n$-cocycles, and thus $R\Gamma(G,\mathcal{M})$ computes the continuous cohomology groups $H^i_{\cts}(G,M)$.

We extend the definition of $R\Gamma(G,\mathcal{M})$ and $H^i(G,\mathcal{M})$ to the setting where $\mathcal{M}$ is a complex of condensed $G$-modules, by taking the totalization of the double complex whose $n$th column is $\RHom(\Z[G^n],\mathcal{M})$.  In this setting one formally gets a spectral sequence
\[ H^i(G,H^j(\mathcal{M}))\implies H^{i+j}(G,\mathcal{M}). \]
    
\begin{prop}  
\label{PropOxxDescent}  Let $G$ be a profinite group, and let $Y\to X$ be a $G$-torsor in v-stacks.   Let $\mathcal{F}$ be a sheaf of abelian groups on $X_{\proet}$. Then there is a natural isomorphism
\[ R\Gamma(X_{\proet},\mathcal{F})\isomto R\Gamma(G,R\Gamma_{\cond}(Y_{\proet},\mathcal{F})). \]
In particular there is a spectral sequence
\[ H^i(G,H^j_{\cond}(Y_{\proet},\mathcal{F}))\implies H^{i+j}(X_{\proet},\mathcal{F}). \] 
\end{prop}

\begin{proof} The proof copies \cite[Proposition 3.6.3]{BSSW}, which proves the same result for the integral structure sheaf $\OO^+$ on $X_{\proet}$.  

Since $Y\to X$ is pro-\'etale, the pro-\'etale cohomology of $X$ can be computed by means of the simplicial cover: 
\[ \cdots\mathrel{\substack{\textstyle\rightarrow\\[-0.6ex]
                      \textstyle\rightarrow \\[-0.6ex]
                      \textstyle\rightarrow}} Y \times_X Y \rightrightarrows Y \]
Namely, $R\Gamma(X_{\proet},\mathcal{F})$ is quasi-isomorphic to the totalization of the corresponding double complex
    \[ 
    R\Gamma(Y_{\proet},\mathcal{F}) \to R\Gamma((Y\times_X Y)_{\proet},\mathcal{F}) \to\cdots.
    \] 
Since $Y\to X$ is a $G$-torsor, the $n$th column in the double complex is isomorphic to 
\[ R\Gamma((Y\times G^n)_{\proet}, \mathcal{F}) \isom 
\RHom(\Z[G^n], R\Gamma_{\cond}(Y_{\proet},\mathcal{F})) \]
by \eqref{EqDerivedSPoints}.  
This is exactly the double complex whose totalization computes $R\Gamma(G,R\Gamma_{\cond}(Y_{\proet},\mathcal{F}))$.
\end{proof}

\subsection{The sheaf of principal units $\OO^{\xx}$, and the logarithm exact sequence}

Let $X$ be a v-stack over $\Spd \Q_p$ admitting a pro-\'etale cover by a representable object.  
By the tilting equivalence, we can think of $X$ as a sheaf of groupoids on the category of perfectoid spaces over $\Q_p$.
Then a basis for $X_{\proet}$ consists of affinoid perfectoid $U=\Spa(R,R^+)$ over $\Q_p$ admitting a pro-\'etale morphism to $X$. Let us define the sheaf of principal units $\OO^{\xx}$ by
\[ \OO^{\xx}(U)=1+R^{\circ\circ},\]
where $R^{\circ\circ}\subset R$ is the subset of topologically nilpotent elements.  Note that since $R$ is complete, $R^{\circ\circ}$ is a group:  for $x\in R^{\circ\circ}$, the inverse of $1+x$ is $1-x+x^2-\dots$. 

(The sheaf $\OO^{\xx}$ would more properly be called $\OO^{\sharp\xx}$, since it is a subsheaf of the sheaf we have called $\OO^{\sharp}$.  We have chosen to use $\OO^{\xx}$ to ease the notational burden.)

\begin{example} \label{ExDisc} Let $K$ be a nonarchimedean field containing $\Q_p$, let $A=\OO_K\powerseries{T_1,\dots,T_d}$, and let $D=\Spf A$.  Let $D_\eta$ be the adic generic fiber of $D$, so that $D_\eta$ is the nonvanishing locus of $\abs{p}$ in $\Spa(A,A)$.  Then $H^0(D_\eta^{\diamondsuit}, \OO^{\ast\ast})\isom A^{\xx}=1+M$, where $M\subset A$ is the maximal ideal.  Applied to Lubin--Tate space, we find $H^0(\LT_{n,\eta}^{\diamondsuit},\Oxx)\isom A_n^{\xx}$.
\end{example}

The following {\em logarithm exact sequence} will allow us to access the cohomology of $\OO^{\ast\ast}$ on various rigid-analytic spaces.

\begin{lemma}\label{lem:log_ses} Let $X$ be a v-stack over $\Spd \Q_p$ admitting a pro-\'etale cover by a representable object.  There is an exact sequence of sheaves of abelian groups on $X_{\proet}$:
    \begin{equation}\label{eq:logseq}
        0 \to \mu_{p^\infty} \to \OO^{\ast\ast} \stackrel{\log}{\to} \OO \to 0. 
    \end{equation}
Here $\mu_{p^\infty}$ is the sheaf of $p$th power roots of 1.  
\end{lemma}

\begin{proof} Suppose $U\to X$ is an affinoid perfectoid object of $X_{\proet}$, with $U=\Spa(R,R^+)$.  The usual logarithm series converges on $\OO^{\ast\ast}(U)$ to an element of $\OO(U)$.  

We first claim that the kernel of $\log\from \OO^{\ast\ast}(U)\to \OO(U)$ is $\mu_{p^\infty}(U)$. It is clear that $\mu_{p^\infty}(U)$ is contained in the kernel. For the other inclusion, let $f\in \OO^{\ast\ast}(U)$, then there exists $n\geq 0$ large enough so that $f^{p^n}-1\in p^2R^+$.  Then 
\[ p^n \log f =\log f^{p^n} =\sum_{i\geq 1} (-1)^{i-1}\frac{(f^{p^n}-1)^i}{i}\in (f^{p^n}-1)(1+pR^+). \]
(Here we have used $2(i-1)\geq v_p(i)+1$ for all $i\geq 2$, where $v_p$ is the $p$-adic valuation on integers.)  
Thus if $\log f = 0$, we must have $f^{p^n}=1$.

Now we claim that $\log\from \OO^{\ast\ast}\to \OO$ is surjective in the pro-\'etale topology. Let $U=\Spa(R,R^+)$ as above.  For $f\in \OO(U)=R$, let $n\geq 0$ be large enough so that $p^nf \in p^2 R^+$.  Then $\exp(p^nf)$ converges to an element $e\in A^{\ast\ast}$ such that $\log e = p^n f$.  Let $S=R[e^{1/p^n}]:= R[T]/(T^{p^n}-e)$;  since $e\in R^\ast$ and $p$ is invertible in $R$, we have that $S$ is a faithfully flat and finite \'etale $R$-algebra.  We find that if $V=\Spa(S,S^+)$ (with $S^+=$ integral closure of $R^+$ in $S$), then $V\to U$ is an \'etale cover for which $\OO^{\ast\ast}(V)$ contains a preimage of $f$ under $\log$, namely $e^{1/p^n}$.  
\end{proof}

\subsection{The fundamental exact sequence}

In this section we leverage the isomorphisms between stacks in \cref{ThmEquivalenceOfStacks} to control the continuous cohomology  $H^1_{\cts}(\G_n,A_n^{\ast\ast})$. 

Throughout, we will suppress the subscripts ``$\proet$'' and ``cts'' to ease notational burden;  it will always be understood that the cohomology on diamonds is with respect to the pro-\'etale topology, and that the cohomology of profinite groups is meant to be continuous.   

Suppose we are in the situation of Proposition \ref{PropOxxDescent}:  A profinite group $G$ acts on a diamond $Y$ over $\Spd\Q_p$, with quotient $X=[Y/G]$.  Then we have a spectral sequence 
\[ 
H^i(G,H^j_{\cond}(Y,\Oxx)) \implies H^{i+j}(X,\Oxx) 
\] 
whose low-degree terms form an exact sequence:
    \begin{equation}\label{EqLowDegree}
        0\to H^1(G,\Oxx(Y)) \to H^1(X,\OO^{\xx}) \to H^1(Y,\OO^{\xx})^G. 
    \end{equation}
    Here, $\Oxx(Y)$ means $H^0_{\cond}(Y,\Oxx)$, a condensed abelian group with $G$-action.

The formation of the exact sequence in \eqref{EqLowDegree} is functorial in the following sense:  If a morphism of pairs $(Y_1,G_1)\to (Y_2,G_2)$ is interpreted to mean a homomorphism $\phi\from G_1\to G_2$ and $\psi\from Y_1\to Y_2$ such that $\psi(gy)=\phi(g)\psi(y)$ for $g\in G_1$, $y\in Y_1$, then such a morphism induces a morphism of spectral sequences and hence a morphism between the corresponding low-degree exact sequences.

Now suppose there are two such presentations for $X$:
    \[ 
        \mathfrak{X}\isom [Y_1/G_1]\isom [Y_2/G_2]. 
    \]
The corresponding low-degree exact sequences assemble to give a diagram
\begin{equation}
    \label{EqPushoutCross}
    \begin{gathered}
        \xymatrix{
            & & 0 \ar[d] & \\
            & P \ar[r] \ar[d] & H^1(G_1,\OO^{\xx}(Y_1)) \ar[d] & \\
            0 \ar[r] & H^1(G_2,\OO^{\xx}(Y_2)) \ar[r] & H^1(X,\OO^{\xx}) \ar[r] \ar[d] & H^1(Y_{2},\OO^{\xx})^{G_2} \\
            & & H^1(Y_{1},\OO^{\xx})^{G_1}, & 
        }
    \end{gathered}
\end{equation}
where $P$ is the indicated pullback in the diagram. In particular, there is an exact sequence
\begin{equation}
    \label{EqTwoPresentationsExactSequence}
0 \to P \to H^1(G_1,\Oxx(Y_{1})) \to H^1(Y_{2},\OO^{\xx})^{G_2} 
\end{equation}
as well as a similar one which swaps the roles of the two presentations.  The formation of the diagrams in \eqref{EqPushoutCross} and \eqref{EqTwoPresentationsExactSequence} are functorial on the category of data consisting of $(Y_1,G_1)$,$(Y_2,G_2)$, and an isomorphism $[Y_1/G_1]\isomto [Y_2/G_2]$ and with morphisms consisting of a pair of maps of pairs and a $2$-commuting square of stacks.
  
Let $C$ be the completion of an algebraic closure of $\Q_p$, and let $\Gamma_{\Q_p}=\Gal(\overline{\Q}_p/\Q_p)$,
so that $\Spd C\to\Spd \Q_p$ is a pro-\'etale $\Gamma_{\Q_p}$-torsor in diamonds.   \Cref{ThmEquivalenceOfStacks} shows that for each $n\geq 1$ there is a pair of presentations of the stack $\SW_n$:
\begin{equation}
    \label{EqTwoPresentations}
    \SW_n \isom [\LT_{n,\eta}^{\diamond}/\G_n] \isom [\HH_C^{n-1,\diamond}/(\Gamma_{\Q_p}\times\GL_n(\Z_p))]. 
\end{equation}
Here, we have replaced $\HH^{n-1,\diamond}$ with its equivalent presentation $\HH_C^{n-1,\diamond}/\Gamma_{\Q_p}$. It will be convenient to use the notation 
    \begin{equation}\label{eq:notationpin}
        \Pi_n = \Gamma_{\Q_p} \times \GL_n(\Z_p).
    \end{equation}

Recall from \cref{ThmDeterminant} that the determinant morphism $\det\from \SW_n\to \SW_1$ is compatible with the two presentations of $\SW_n$ given in \eqref{EqTwoPresentations}.  Applying \eqref{EqPushoutCross} functorially to the determinant morphism, we obtain a commutative diagram:
\begin{equation}
    \label{EqBigAssDiagram}
\adjustbox{scale=0.65,center}{
\begin{tikzcd}
	&&& 0 \\
	& {P_1} && {H^1(\Pi_1,\Oxx(\HH_C^{0,\diamond}))} & 0 \\
	&& {P_n} && {H^1(\Pi_n,\Oxx(\HH_C^{n-1,\diamond}))} \\
	0 & {H^1(\G_1,\Oxx(\LT_{1,\eta}^{\diamond}))} && {H^1(\SW_1,\Oxx)} && {H^1(\LT_{1,\eta}^{\diamond},\Oxx)^{\G_1}} \\
	& 0 & H^1(\G_n,\Oxx(\LT_{n,\eta}^{\diamond})) && {H^1(\SW_n,\Oxx)} && {H^1(\LT_{n,\eta}^{\diamond},\Oxx)^{\G_n}} \\
	&&& {H^1(\HH^{0,\diamond}_C,\Oxx)^{\Pi_1}} \\
	&&&& {H^1(\HH^{n-1,\diamond}_C,\Oxx)^{\Pi_n}.}
	\arrow["q"', from=2-2, to=3-3]
	\arrow[from=3-5, to=5-5]
	\arrow[from=2-5, to=3-5]
	\arrow[from=5-5, to=7-5]
	\arrow[from=3-3, to=3-5]
	\arrow[from=2-4, to=4-4]
	\arrow["\det_{\HH}^\ast", from=2-4, to=3-5]
	\arrow[from=2-2, to=2-4]
	\arrow[from=1-4, to=2-4]
	\arrow[from=4-4, to=5-5]
	\arrow[from=4-4, to=6-4]
	\arrow[from=6-4, to=7-5]
	\arrow[from=3-3, to=5-3]
	\arrow[from=5-3, to=5-5]
	\arrow[from=5-5, to=5-7]
	\arrow["\det_{\LT}^{\ast}",from=4-2, to=5-3]
	\arrow["r", from=2-2, to=4-2]
	\arrow[from=4-2, to=4-4]
	\arrow[from=4-4, to=4-6]
	\arrow["\det_{\LT}^{\ast}", from=4-6, to=5-7]
	\arrow[from=4-1, to=4-2]
	\arrow[from=5-2, to=5-3]
\end{tikzcd}}
\end{equation}
We proceed with seven lemmas concerning objects in the diagram \eqref{EqBigAssDiagram}.

\begin{lemma}\label{lem:liouville}As condensed abelian groups we have $\OO^{\xx}(\HH^{n-1,\diamond}_C)\isom C^{\xx}$, and $\OO^{\xx}(\LT_{n,\eta}^{\diamond})\isom A_n^{\xx}$.
\end{lemma}

\begin{proof}    
On rigid-analytic spaces $X$ over $C$, the sheaf $\OO^{\xx}$ is represented by the rigid-analytic fiber  of the formal multiplicative group.   Now apply \cref{LemFullyFaithful} to find that for smooth $X$, $\OO^{\ast\ast}(X)\isom \OO^{\ast\ast}(X^{\diamond})$ as condensed abelian groups.  Therefore the lemma is reduced to calculating $\OO^{\ast\ast}(X)$ as a topological abelian group.   

In the case $X=\HH^{n-1}_C$, we appeal to \cite[Lemma 3]{MR1360770}, which states that every bounded function on $\HH^{n-1}_C$ is constant.  In particular $\OO^{\ast\ast}(X)=C^{\xx}$.  In the case $X=\LT_{n,\eta}$, which is a rigid-analytic open ball, we have $\OO^{\xx}(\LT_{n,\eta})=A_n^{\xx}$.  
\end{proof}

\begin{lemma}\label{lem:SLncohom}
    Considering $\Q_p^{\xx}=1+p\Z_p$ with its trivial $\SL_n(\Z_p)$-action, we have:
        \[
            H^1(\SL_n(\Z_p),\Q_p^{\xx}) \isom 
                \begin{cases}
                    0, & (n,p) \neq (2,2); \\ 
                    \ZZ/2 & (n,p) = (2,2).
                \end{cases}
        \]
In the case $(n,p) = (2,2)$, the group $H^1(\SL_n(\Z_p),\Q_p^{\xx})$ is generated by the homomorphism given by the composition 
    \[
        \SL_2(\Z_2)\xrightarrow{\mod 2\;\;\;} \SL_2(\Z/2)\isom S_3\xrightarrow{\mathrm{sign}} \{\pm 1\} \subset \QQ_2^{\xx}. 
    \]
\end{lemma}
\begin{proof} Since $\Q_p^{\xx}$ is abelian, $H^1(\SL_n(\Z_p),\Q_p^{\xx})$ is isomorphic to the group of continuous homomorphisms $\SL_n(\Z_p)^{\overline{\ab}}\to\Q_p^{\xx}$, where 
    \[\SL_n(\Z_p)^{\overline{\mathrm{ab}}} = \SL_n(\Z_p) / \overline{[\SL_n(\Z_p),\SL_n(\Z_p)]}.\]
Since $\SL_n(\ZZ)$ is dense in $\SL_n(\Z_p)$, the image of $\SL_n(\ZZ)^{\mathrm{ab}} \to \SL_n(\Z_p)^{\overline{\mathrm{ab}}}$ is also dense.  We apply the well-known result \cite[1.2.11, 1.2.15, Theorem 4.3.22]{hahn2013classical}:
    \[
\SL_n(\ZZ)^{\mathrm{ab}} = \begin{cases}
0, & n \neq 2; \\ 
\ZZ/12 & n =2. 
\end{cases}\]
This settles the claim for $n\neq 2$, and shows that in all cases $\SL_n(\Z_p)^{\overline{\mathrm{ab}}}$ is a (possibly trivial) torsion cyclic group. For $p \neq 2$, $\QQ_p^{\xx} \cong \ZZ_p$ is torsion free, so this settles the claim for $p \neq 2$.

We are left with the case $(n,p) = (2,2)$. As $\QQ_2^{\xx} \cong \ZZ_2\oplus \ZZ/2$ and $\SL_2(\Z_2)^{\overline{\mathrm{ab}}}$ is torsion cyclic we see that $H^1(\SL_2(\Z_2),\Q_2^{\xx})$ is of order at most $2$. To conclude we observe that the map $\SL_2(\Z_2)\to\Q_2^{\xx}$ described above determines a non-trivial element in $H^1(\SL_2(\Z_2),\Q_2^{\xx})$.
\end{proof}

\begin{lemma} \label{lem:detH1iso}
For $(n,p)\neq (2,2)$, the map
\[ \det\!_{\HH}^\ast \from H^1(\Pi_1,\Oxx(\HH_C^0))
\to 
H^1(\Pi_n,\Oxx(\HH_C^{n-1})) \]
is an isomorphism.  For $(n,p) = (2,2)$, $\det\!_{\HH}^{\ast}$ is an inclusion of a direct summand with complement $\Z/2$ generated by the class $\alpha\in H^1(\Pi_2,C^{\xx})$ defined as the composite 
    \[
        \Pi_2 \to \GL_2(\ZZ_2) \xrightarrow{\mod 2\;}  \GL_2(\ZZ/2) \simeq S_3 \xrightarrow{\mathrm{sign}} \{\pm 1\} \subset C^{\xx},
    \]
where the first map is the canonical projection.
\end{lemma}

\begin{proof}
In light of Lemma \ref{lem:liouville}, the map $\det_{\HH}^\ast$ agrees with the map appearing below, in the inflation-restriction sequence for $C^{\xx}$ relative to the map $\Pi_n \to \Pi_1$ induced by the determinant map $\GL_n(\Z_p)\to\Z_p^\ast$:
    \[
        0\to H^1(\Pi_1, (C^{\xx})^{\SL_n(\Z_p)}) \xrightarrow{\det_{\HH}^{\ast}} H^1(\Pi_n,C^{\xx}) \to H^1(\SL_n(\Z_p),C^{\xx})^{\Pi_1}. 
    \]
The action of $\SL_n(\Z_p)$ on $C^{\xx}$ is trivial, and \cite[Theorem 1]{TatePDivisibleGroups} shows that $C^{\Gamma_{\Q_p}}=\Q_p$.  We obtain an exact sequence 
    \[ 
        0\to H^1(\Pi_1, C^{\xx}) \xrightarrow{\det_{\HH}^{\ast}} H^1(\Pi_n,C^{\xx}) \to H^1(\SL_n(\Z_p),\Q_p^{\xx}). 
    \]
For $(n,p) \neq (2,2)$, the claim follows from \cref{lem:SLncohom}. In the case $(n,p) = (2,2)$, it follows from \cref{lem:SLncohom} and the fact that the class $\alpha \in H^1(\Pi_2,C^{\xx})$ is a lift of the nontrivial class in $H^1(\SL_2(\Z_2),\Q_2^{\xx})$ with $2\alpha=0$, that the above sequence is split short exact.
\end{proof}

\begin{lemma}\label{lem:risom}
The map
    \[ 
        r\from P_1\to H^1(\G_1,\Oxx(\LT_{1,K})) 
    \]
is an isomorphism.
\end{lemma}
\begin{proof}
    Consider the following exact sequence, which is an instance of \eqref{EqTwoPresentationsExactSequence} applied to \eqref{EqBigAssDiagram}:
        \[
            0 \to P_1 \xrightarrow{r} H^1(\G_1,\Oxx(\LT_{1,K})) \to H^1(\HH^0_C,\Oxx)^{\Pi_1}.
        \]
    Since $\HH_C^0 = \Spa(C,\OO_C)$, we have $H^1(\HH^0_C,\Oxx) = 0$.
\end{proof}

\begin{lemma} \label{lem:LTdetinjective}
    The map 
        \[ 
            \det\!_{\LT}^{\ast}\from H^1(\LT_{1,\eta}^{\diamond},\Oxx)^{\G_1} \to H^1(\LT_{n,\eta}^{\diamond},\Oxx)^{\G_n} 
        \]
is injective.
\end{lemma}
\begin{proof}  
It is enough to show that the map 
    \[
        \det\!_{\LT}^{\ast}\from H^1(\LT_{1,\eta}^{\diamond},\Oxx)\to H^1(\LT_{n,\eta}^{\diamond},\Oxx) 
    \]
is injective. This follows from the observation that any point $\Spa(K,W(\overline{\F}_p))\to \LT_{n,\eta}$ induces a retraction for $\det\!_{\LT}^{\ast}$.
\end{proof}

\begin{lemma} \label{lem:22isspecial}
Assume that $(n,p)=(2,2)$. The map
    \[
        H^1(\Pi_2,\Oxx(\HH_C^{1,\diamond})) \to H^1(\LT_{2,\eta}^{\diamond},\Oxx)
    \]
appearing in \eqref{EqBigAssDiagram} sends the class $\alpha$ from \cref{lem:detH1iso} to $0$.  Therefore, the class $\alpha$ lifts to $P_2$.
\end{lemma}
\begin{proof}
Consider the short exact sequence of sheaves on any v-stack over $\Spd\Q_2$:
    \[
        0\to \mu_2 \to \Oxx \xrightarrow{(\;)^2} \Oxx\to 0.
    \]
This short exact sequence gives rise to the following diagram with exact right vertical column:
\[\begin{tikzcd}
	&& {\Oxx(\LT_{2,\eta}^{\diamond})} \\
	&& {\Oxx(\LT_{2,\eta}^{\diamond})} \\
	{H^1(\Pi_2,\mu_2(\HH_C^{1,\diamond}))} & {H^1(\SW_2,\mu_2)} & {H^1(\LT_{2,\eta}^{\diamond},\mu_2)} \\
	{H^1(\Pi_2,\Oxx(\HH_C^{1,\diamond}))} & {H^1(\SW_2,\Oxx)} & {H^1(\LT_{2,\eta}^{\diamond},\Oxx).} 
	\arrow["{(\;)^2}"', from=1-3, to=2-3]
	\arrow["{\partial}"',from=2-3, to=3-3]
	\arrow[from=3-1, to=3-2]
	\arrow[from=3-1, to=4-1]
	\arrow[from=3-2, to=3-3]
	\arrow[from=3-2, to=4-2]
	\arrow[from=3-3, to=4-3]
	\arrow[from=4-1, to=4-2]
	\arrow[from=4-2, to=4-3]
\end{tikzcd}\]
The group $\mu_2(\HH^{1,\diamond}_C)=\set{\pm 1}$ is discrete, so $H^1(\Pi_2,\mu_2(\HH_C^{1,\diamond}))\isom H^1_{\cts}(\Pi_2,\set{\pm 1})$. 
Let $\tilde{\alpha}\in H^1(\Pi_2,\mu_2(\HH_C^{1,\diamond}))$ be the class corresponding to the composite 
    \[
        \Pi_2 \to \GL_2(\ZZ_2) \xrightarrow{\mod 2}  \GL_2(\ZZ/2) \simeq S_3 \xrightarrow{\mathrm{sign}} \{\pm 1\}.
    \]
The image of $\tilde{\alpha}$ under the map $H^1(\Pi_2,\mu_2(\HH_C^{1,\diamond}))\to {H^1(\Pi_2,\Oxx(\HH_C^{1,\diamond}))}$ is $\alpha$.
Let $\overline{\alpha}\in H^1(\LT_{2,\eta}^{\diamond},\mu_2)$ be the image of $\tilde{\alpha}$  under the composition of the top horizontal maps. It now suffices to show that $\overline{\alpha}$ lies in the image of $\partial\colon \Oxx(\LT_{2,\eta}^{\diamond}) \to H^1(\LT_{2,\eta}^{\diamond},\mu_2)$.

The group $H^1(\LT_{2,\eta}^{\diamond},\mu_2)$ classifies \'etale $\mu_2$-torsors (i.e., double covers) of $\LT_{2,\eta}^{\diamond}$.  The \'etale sites of a rigid-analytic space and its diamond are equivalent \cite[Lemma 15.6]{ScholzeDiamonds}, so all the same this group classifies double covers of $\LT_{2,\eta}$.  With respect to this identification, $\partial$ sends a principal unit in $A_2^{\xx}$ to the double cover of $\LT_{2,\eta}$ obtained by adjoining its square root.  The $\mu_2$-torsor corresponding to $\overline{\alpha}$ can be described this way:  Let $H^{\univ}_2$ be the universal deformation of $H_2$ over $A_2$.  The 2-torsion subgroup $H^{\univ}_2[2]$ is a group scheme of order 4 over $\LT_2=\Spf A_2$, whose generic fiber $H^{\univ}_2[2]_{\eta}$ is \'etale over $\LT_{2,\eta}$.  Trivializations of $H^{\univ}_2[2]$ over $\LT_{2,\eta}$ (or, what is the same, markings of its three nonzero sections) form a torsor under the group $\GL_2(\Z/2)\isom S_3$.  Finally, $\overline{\alpha}$ is the quotient of this torsor corresponding to the sign homomorphism $S_3\to \mu_2$, which is to say it is the torsor classifying markings of the three nonzero sections of $H^{\univ}_2[2]$ up to cyclic permutation.  

Let $E$ be a supersingular elliptic curve over $\overline{\F}_2$, with universal deformation $E^{\univ}$.  By Serre--Tate theory \cite{SerreTate1968}, there is an isomorphism between $\LT_2$ and the formal scheme parameterizing deformations of $E$, which identifies $H_2^{\univ}$ with the completion of $E^{\univ}$ along its origin.  In particular $H_2^{\univ}[2]$ is identified with the $2$-torsion group $E^{\univ}[2]$, and $\overline{\alpha}$ is the double cover of $\LT_{2,\eta}$ which classifies markings of the three nonzero sections of $E^{\univ}[2]$ up to cyclic permutation.

A choice of nonvanishing differential $\omega$ on $E^{\univ}$ determines a Weierstrass model for $E^{\univ}$ over the local ring $A_2$.  Over $A_2[1/2]$, the model can be brought into reduced form, with affine equation $y^2=f(x)$, for a monic cubic polynomial $f(x)\in A_2[1/2][x]$.  The nonzero 2-torsion sections of $E^{\univ}$ correspond to roots $e_1,e_2,e_3$ of $f(x)$.  Therefore $\overline{\alpha}$ is the double cover of $\LT_{2,\eta}$ obtained by adjoining the square root of $\Delta(E^{\univ},\omega)$, where $\Delta(E^{\univ},\omega)=\prod_{i<j} (e_i-e_j)^2\in A_2^\ast$ is the discriminant of the Weierstrass model determined by $\omega$.  For a scalar $\lambda\in A_2^\ast$ we have $\Delta(E^{\univ},\lambda\omega)=\lambda^{12}\Delta(E^{\univ},\omega)$;  since every element is a 12th power in $\overline{\F}_2^\ast$, we can find an $\omega$ with $\Delta(E^{\univ},\omega)\in A_2^{\ast\ast}$.    We conclude that $\overline{\alpha}$ lies in the image of $\partial$;  indeed it is the image of $\Delta(E^{\univ},\omega)$. 
\end{proof}

\begin{lemma}\label{lem:qisom}
For $(n,p) \neq (2,2)$, the map $q\from P_1\to P_n$ is an isomorphism. When $(n,p) = (2,2)$, the map $q$ is an inclusion of a direct summand with complement $\Z/2$ generated by $\alpha$.
\end{lemma}    
\begin{proof}
Formation of the exact sequence \eqref{EqTwoPresentationsExactSequence} is functorial, so we find a diagram
\[
\xymatrix{
0  \ar[r] & P_1 \ar[r] \ar[d]_q & H^1(\Pi_1,\Oxx(\HH_C^{0,\diamond})) \ar[d]_{\det_{\HH}^\ast} \ar[r] & H^1(\LT_{1,\eta}^{\diamond},\Oxx)^{\G_1} \ar[d]_{\det_{\LT}^\ast} \\
0  \ar[r] & P_n \ar[r]  & H^1(\Pi_n,\Oxx(\HH_C^{n-1,\diamond})) \ar[r] & H^1(\LT_{n,\eta}^{\diamond},\Oxx)^{\G_n}. 
}
\]
When $(n,p) \neq (2,2)$, the map $\det_{\HH}^\ast$ is an isomorphism by \cref{lem:detH1iso} and $\det_{\LT}^{\ast}$ is injective by \cref{lem:LTdetinjective}, so $q$ is an isomorphism. In the case $(n,p) = (2,2)$, the map $\det_{\HH}^\ast$ is the inclusion of a direct summand with complement $\Z/2$ generated by the class $\alpha$, by \cref{lem:detH1iso}.  By \cref{lem:22isspecial}, $\alpha$ is a $2$-torsion element of $P_n$ that is not in the image of $q$.  We conclude that $q$ is the inclusion of a direct summand with complement $\Z/2$ generated by $\alpha$.
\end{proof}

When $(n,p)=(2,2)$, we denote the image of $\alpha$ under the map $P_2 \to H^1_{\cts}(\G_2,A_2^{\ast\ast})$ by $\hat{\alpha}$.
Combining the lemmas above, we deduce the following theorem, which appears in the introduction as \cref{ThmFundamentalExactSequenceIntro}.

\begin{thm}\label{ThmFundamentalExactSequence}
    For $(n,p) \neq (2,2)$, there is an exact sequence
        \[
            0 \to H^1_{\cts}(\G_1,A_1^{\ast\ast}) \xrightarrow{\det_{\LT}^{\ast}} H^1_{\cts}(\G_n,A_n^{\ast\ast}) \stackrel{b}{\to} H^1_{\proet}(\HH_C^{n-1},\OO^{\ast\ast})^{\Pi_n}, 
        \]
    where $b$ is the composite $H^1_{\cts}(\G_n,A_n^{\ast\ast}) \to H^1_{\proet}(\SW_n,\OO^{\ast\ast}) \to H^1_{\proet}(\HH_C^{n-1},\OO^{\ast\ast})^{\Pi_n}$ in \eqref{EqBigAssDiagram}. 
    
    When $(n,p) = (2,2)$ there is an exact sequence
        \[
            0 \to H^1_{\cts}(\G_1,A_1^{\ast\ast}) \oplus \ZZ/2 \xrightarrow{\det_{\LT}^{\ast}\oplus \hat{\alpha}} H^1_{\cts}(\G_2,A_2^{\ast\ast}) \stackrel{b}{\to} H^1_{\proet}(\HH_C^{1},\OO^{\ast\ast})^{\Pi_2}, 
        \]
    where $b$ is the composite $H^1_{\cts}(\G_2,A_2^{\ast\ast}) \to H^1_{\proet}(\SW_2,\OO^{\ast\ast}) \to H^1_{\proet}(\HH_C^{1},\OO^{\ast\ast})^{\Pi_2}$ in \eqref{EqBigAssDiagram}.     
\end{thm}         
\begin{proof}  We have $\OO^{\xx}(\LT_{n,\eta}^{\diamondsuit})=A_n^{\xx}$, as discussed in \cref{ExDisc}.  Consider the following commutative diagram, which is part of \eqref{EqBigAssDiagram}:
        \[
            \xymatrix{P_n \ar[r] \ar[d] & H^1_{\cts}(\Pi_n,\Oxx(\HH_C^{n-1})) \ar[d] \\
            H^1_{\cts}(\G_n,A_n^{\ast\ast}) \ar[r] \ar[rd]_b & H^1_{\proet}(\SW_n,\Oxx) \ar[d] \\
             & H^1_{\proet}(\HH_C^{n-1},\OO^{\ast\ast})^{\Pi_n}.}
        \]
    As an instance of \eqref{EqTwoPresentationsExactSequence}, we obtain an exact sequence
        \[
            0 \to P_n \to H^1_{\cts}(\G_n,A_n^{\ast\ast}) \xrightarrow{b} H^1_{\proet}(\HH_C^{n-1},\OO^{\ast\ast})^{\Pi_n}.
        \]
    \Cref{lem:risom} and \cref{lem:qisom} provide isomorphisms
        \[
            P_n \cong 
                \begin{cases}
                    H^1_{\cts}(\G_1,A_1^{\ast\ast}) & \text{if } (n,p) \neq (2,2); \\
                    H^1_{\cts}(\G_1,A_1^{\ast\ast}) \oplus \ZZ/2 & \text{if } (n,p) = (2,2).
                \end{cases}
        \]
    Together with the commutativity of the left square in \eqref{EqBigAssDiagram}, this gives the two exact sequences of the theorem. 
\end{proof}

When $n=1$, the theorem is uninteresting:  $\det_{\LT}^\ast$ is an isomorphism, and the target of $b$ is $0$ for dimension reasons.  The next objective is to identify $H^1_{\proet}(\HH_C^{n-1},\OO^{\xx})^{\Pi_n}$ for $n\geq 2$.  
In \S5.5 we will prove that this group is a free $\Z_p$-module of rank 1 and the map $b$ is surjective, generated by the image of $\varepsilon_p(\Sigma^2\mathbb{S}_{K(n)})$.   Together with the identification of $H^1_{\cts}(\G_1,A_1^{\xx})$ in \cref{lem:height1contribution}, this claim will complete the proof of  \cref{thm:A**main}.

\section{The cohomology of Drinfeld symmetric space $\HH_C^{n-1}$}

Via the fundamental exact sequence in \cref{ThmFundamentalExactSequence}, we have reduced the crucial \cref{thm:A**main} to a question about $H^1_{\proet}(\HH^{n-1}_C,\Oxx)$, the cohomology of $\Oxx$ on the pro-\'etale site of Drinfeld's symmetric space $\HH_C^{n-1}$.  Recall that $\HH_C^{n-1}$ is the complement in rigid-analytic $\mathbb{P}^{n-1}_C$ of all $\Q_p$-rational hyperplanes;  as such it admits an action of $\GL_n(\Q_p)$.   In particular $\HH_C^1=\mathbb{P}^1_C\backslash\mathbb{P}^1(\Q_p)$ is analogous to the (upper and lower) half-plane $\mathbb{P}^1(\mathbb{C})\backslash\mathbb{P}^1(\mathbb{R})$.  

The cohomology of $\HH_C^{n-1}$ was studied first by Drinfeld \cite{DrinfeldEllipticModules} in the case $n=2$, and then by Schneider--Stuhler \cite{SchneiderStuhler} for general $n$.  The main theorem of \cite{SchneiderStuhler} is that if $H$ is a cohomology theory on rigid-analytic varieties satisfying certain reasonable axioms (for instance, homotopy-invariance relative to the open unit disc $D$), then $H^i(\HH_C^{n-1})$ is dual to a generalized Steinberg representation of $\GL_n(\Q_p)$ for $0\leq i\leq n$, and is 0 otherwise.  (The construction of the generalized Steinberg representations is reviewed below.)  The result applies to $\ell$-adic \'etale cohomology ($\ell\neq p$) and to de Rham cohomology.

In contrast, \'etale cohomology with $\Z_p$-coefficients does not satisfy homotopy-invariance; indeed, $H^1_{\et}(D,\Z_p)$ is very large.  One therefore might expect that there cannot be a meaningful description of $H^i_{\et}(\HH_C^{n-1},\Z_p)$, let alone $H^i_{\proet}(\HH_C^{n-1},\Q_p)$.  Nonetheless, both groups were identified by Colmez--Dospinescu--Nizio\l \cite{CDNStein}, \cite{CDNIntegral}.  In the latter setting, $H^i_{\proet}(\HH_C^{n-1},\Q_p)$ is an extension of a space of differential forms by a generalized Steinberg representation.  

Combining these results with the logarithm exact sequence (\cref{lem:log_ses}), we gain access to $H^i_{\proet}(\HH_C^{n-1},\Oxx)$, see \cref{ThmOxxCohomologyOfH} below.  Specializing to the case $i=1$, we find that $H^1_{\proet}(\HH^{n-1}_C,\Oxx)^{\Pi_n}$ is (for $n\geq 2$) isomorphic to the space of $\GL_n(\Z_p)$-invariants in a dual Steinberg representation with coefficients in $\Q_p/\Z_p$.  This space of invariants is shown to be a free rank 1 $\Z_p$-module in \cref{LemFixedPartOfSp}.

\subsection{Generalized Steinberg representations: definitions} The following is a rephrasing of the constructions appearing in \cite{CDNStein}.  For $0\leq r\leq n$, let $X_r$ be the set of flags
    \[ 
        V_1\subset V_2\subset \cdots \subset V_r\subseteq \Q_p^{\oplus n}, 
    \]
where $V_i$ is a $\Q_p$-vector space of dimension $i$. Note that we in particular allow the empty flag. For  $1\leq j\leq r$, let $X_{r,j}$ be set of flags as above with the $j$th vector space omitted.  The set $X_r$ (resp., $X_{r,j}$) is a quotient of $\GL_n(\Q_p)$ by a parabolic subgroup, and so it naturally has the structure of a profinite sets admitting a continuous action of $\GL_n(\Q_p)$; note also the quotient map $X_r\to X_{r,j}$.  

For a profinite set $S=\varprojlim S_i$ and a ring $A$, let $\LC(S,A)$ denote the ring of locally constant functions on $S$ valued in $A$.  If $A$ carries a topology, we give $\LC(S,A)=\varinjlim \LC(S_i,A)$ the colimit topology, where each $\LC(S_i,A)$ (a finite free $A$-module) has the natural topology.  For a quotient map $S\to S'$, the natural map $\LC(S',A)\to \LC(S,A)$ identifies $\LC(S',A)$ with a closed $A$-submodule of $\LC(S,A)$. Define the {\em generalized Steinberg (or special)} representation by
\[ \St_r(A)=\frac{\LC(X_r,A)}{\sum_{j=1}^r \LC(X_{r,j},A)} \]
(with its quotient topology). 
For each $A$, $\St_r(A)$ is a topological $A$-module admitting a continuous action of $\GL_n(\Q_p)$. 
We will make use of the dual Steinberg representation
\[ \St_r(A)^\ast=\Hom_{A,\cts}(\St_r(A),A), \]
where $\Hom_{A,\cts}$ means continuous $A$-module homomorphisms. One can think of $\St_r(A)^\ast$ as the module of $A$-valued  distributions on $X_r$ whose value is $0$ on functions pulled back from any $X_{r,j}$.  

As a special case,
\[ \St_1(A)=\frac{\LC(\mathbb{P}^{n-1}(\Q_p),A)}{A}. \]
So $\St_1(A)^\ast$ is the module of $A$-valued distributions $\mu$ on $\mathbb{P}^{n-1}(\Q_p)$ with $\mu(\mathbb{P}^{n-1}(\Q_p)) = 0$. 
Thus, when we identify $\mathbb{P}^{n-1}(\Q_p)$ with the set of linear forms in $n$ variables up to scaling by $\Q_p^\ast$ and let  $\delta_{\ell}$ denote the Dirac distribution at $\ell\in \mathbb{P}^{n-1}(\Q_p)$, any difference $\delta_{\ell}-\delta_{\ell'}$ determines an element of $\St_1(A)^{\ast}$.  

Since every $\Z_p$-valued distribution is in particular a $\Q_p$-valued distribution we get an equivariant  map  $ \St_r(\Z_p)^{\ast}\to\St_r(\Q_p)^{\ast}$. We use this map to give an ad hoc definition of the $\GL_n(\Q_p)$-module $\St_r(\Q_p/\Z_p)^{\ast}$, defining it to be the quotient:
\begin{equation}
    \label{EqSpExactSequence}
 0 \to \St_r(\Z_p)^{\ast}\to\St_r(\Q_p)^{\ast}\to \St_r(\Q_p/\Z_p)^{\ast} \to 0.
 \end{equation}
Thus for instance elements of $\St_1(\Q_p/\Z_p)^{\ast}$ are $\Q_p/\Z_p$-valued distributions on $\PP^{n-1}(\Q_p)$ with total measure 0.

\begin{lemma} \label{LemFixedPartOfSp} Assume $n\geq 2$. When $A$ is torsion-free, then 
\[\left(\St_1(A)^{\ast}\right)^{\GL_n(\Z_p)} = 0.\]
Moreover, the subset of $\St_1(\Q_p/\Z_p)^{\ast}$ fixed by $\GL_n(\Z_p)$ is a free $\Z_p$-module of rank 1.  
\end{lemma}

\begin{proof} 
To give a uniform argument, we denote by $M$ either the ring of coefficients $A$ or $\Q_p/\Z_p$. Let $\mu$ be an $M$-valued distribution on $\PP^{n-1}(\Q_p) = \PP^{n-1}(\Z_p)$. For each $m\geq 1$, we have a $\GL_n(\Z_p)$-equivariant surjective map 
    \begin{equation}\label{eq:pim}
        \pi_m\from \PP^{n-1}(\Z_p)\to \PP^{n-1}(\Z/p^m\Z). 
    \end{equation}
As $m$ varies, the fibers of the $\pi_m$ form a basis of open subsets of $\PP^{n-1}(\Z_p)$, so that $\mu$ is determined by its values on the fibers.  Now suppose $\mu$ is fixed by $\GL_n(\Z_p)$.  Since $\GL_n(\Z_p)$ acts transitively on $\PP^{n-1}(\Z/p^m\Z)$, the value of $\mu$ is the same on all fibers of $\pi_m$;  let $a_m$ denote the common value.  Then by the additivity of $\mu$:
    \[ 
        \mu(\PP^{n-1}(\Z_p)) = \frac{p^n-1}{p-1}a_1=\frac{p^n-1}{p-1}p^{n-1}a_2 = \frac{p^n-1}{p-1}p^{2(n-1)}a_3 = \cdots 
    \]
Via $\mu\mapsto (a_1,a_2,\dots)$, the subgroup of $\St_1(M)^{\ast}$ fixed by $\GL_n(\Z_p)$ may be identified with the subgroup of $\prod_{m\geq 1}M$ satisfying 
    \[ 
        0 = a_1 = p^{n-1}a_2 = p^{2(n-1)}a_3 = \cdots.
    \]
When $M$ is torsion free this is clearly zero, while for $M=\Q_p/\Z_p$,  this group is $\varprojlim_m \Z/p^{(n-1)m}\Z\isom\Z_p$.
\end{proof}

By the first part of \cref{LemFixedPartOfSp}, $\left(\St_1(A)^{\ast}\right)^{\GL_n(\QQ_p)} = 0$, so taking invariants under $\GL_n(\Z_p)$ in the exact sequence \eqref{EqSpExactSequence}, we find an injective  connecting map
\begin{equation}
\label{EqMapDelta}
     \partial \from H^0(\GL_n(\Z_p),\St_1(\Q_p/\Z_p)^\ast) \to H^1(\GL_n(\Z_p),\St_1(\Z_p)^\ast).
\end{equation}
Let $\mu\in H^0(\GL_n(\Z_p),\St_1(\Q_p/\Z_p)^{\ast})$ be the generator of this group with $a_m=1/p^{(n-1)(m-1)}$ for all $m\geq 2$.

\begin{lemma} \label{LemFixedPartOfSpInH1} Let $\ell\in\PP^{n-1}(\Q_p)$ be arbitrary.  Then $\partial(\mu)$ is represented by the cocycle
    \[ 
        g\mapsto \delta_{\ell}-\delta_{g(\ell)}.
    \]
\end{lemma}
\begin{proof}
The measure $\mu$ can be lifted to a (non-$\GL_n(\Z_p)$-invariant) measure $\tilde{\mu}\in \St_1(\Q_p)^\ast$ defined by the formula (for $m\geq 2$)
\[ \tilde{\mu}(\pi_m^{-1}(x)) = \begin{cases}
    
\frac{1}{p^{(n-1)(m-1)}}-1,& \pi_m(\ell)=x; \\
\frac{1}{p^{(n-1)(m-1)}},& \text{otherwise},
\end{cases}
\]
where $\pi_m$ is the map defined in \eqref{eq:pim}. Then $\partial(\mu)$ is represented by the cocycle sending $g\in \GL_n(\Z_p)$ to $g(\tilde{\mu})-\tilde{\mu}$, and this is easily seen to agree with $\delta_{\ell}-\delta_{g(\ell)}$.
\end{proof}

\subsection{Results of Colmez--Dospinescu--Nizio\l}\label{ssec:cdn}

We turn now to some crucial results of Colmez--Dospinescu--Nizio\l on the \'etale and pro-\'etale cohomology of Drinfeld symmetric space $\HH_C^{n-1}$ over a complete algebraically closed field $C/\Q_p$.  

Recall the difference between \'etale and pro-\'etale cohomology:  whenever $X$ is a scheme (resp. stack, rigid-analytic space, etc.), one has the constant sheaf $\Z/p^n\Z$ on the \'etale site of $X$, and then one {\em defines}
\begin{eqnarray*}
H^i_{\et}(X,\Z_p)&=&\varprojlim H^i_{\et}(X,\Z/p^n\Z) \\
H^i_{\et}(X,\Q_p)&=& H^i_{\et}(X,\Z_p)\otimes_{\Z_p}\Q_p.
\end{eqnarray*}
We remark that $H^i_{\et}(-,\Z_p)$ might not fit into a long exact sequence without finiteness assumptions on the cohomology. Moreover, $H^i_{\et}(X,\Q_p)$ is not the cohomology of any sheaf ``$\Q_p$'' on $X_{\et}$.  On the other hand, one has the constant sheaf $\Q_p$ on the pro-\'etale topology $X_{\proet}$, defined by sending an object $U\in X_{\proet}$ to the ring of continuous maps $\abs{U}\to\Q_p$, and then $H^i_{\proet}(X,\Q_p)$ really is the sheaf cohomology of $\Q_p$ on $X_{\proet}$.  Suppose $X$ is a rigid-analytic space; then the projection $X_{\proet}\to X_{\et}$ induces an isomorphism on cohomology 
\[ H^i_{\et}(X,\Z_p)\to H^i_{\proet}(X,\Z_p) \] 
\cite[Corollary 3.17]{Scholze2013pAdic}, but the map on rational cohomology 
\[ H^i_{\et}(X,\Q_p)\to H^i_{\proet}(X,\Q_p) \]
can fail to be an isomorphism if $X$ is not quasi-compact (for instance, if $X$ is the discrete infinite union of points).  There are also such comparison maps for general $\Z_p$-sheaves, in particular the Tate twists $\Z_p(m)$.  

The cohomology groups $H^i_{\et}(\HH_C^{n-1},\Z_p) \simeq H^i_{\proet}(\HH_C^{n-1},\Z_p)$ and $H^i_{\proet}(\HH_C^{n-1},\Q_p)$ were computed in \cite{CDNStein} and \cite{CDNIntegral}, respectively.  For $\mathbb{Z}_p$ and other $p$-complete sheaves we shall always use $H^*_{\proet}$ to denote the groups $H^*_{\proet}\simeq H^*_{\et}$.
We now review these results, starting with $H^i_{\et}$.  
We need one more ingredient before we can give the statements, namely the Kummer map
\[\kappa\from \OO^\ast(\HH_C^{n-1})\to H^1_{\proet}(\HH_C^{n-1},\Z_p(1)).\]
This map arises as follows: Consider the diagram of short exact sequences on $\HH_{C,\proet}$:
    \[
        \xymatrix{0 \ar[r] & \mu_{p^{n+1}} \ar[r] \ar[d] & \OO^\ast \ar[r]^-{(-)^{p^{n+1}}} \ar[d]_{(-)^p} &  \OO^\ast \ar[r] \ar[d]^-{\mathrm{id}} & 0 \\
        0 \ar[r] & \mu_{p^{n}} \ar[r] & \OO^\ast \ar[r]^-{(-)^{p^{n}}} &  \OO^\ast \ar[r]  & 0.} 
    \]
Taking limits in $n$ results in another short exact sequence, 
    \[
        \xymatrix{0 \ar[r] & \Z_p(1) \ar[r] & \varprojlim_{(-)^p}\OO^\ast \ar[r] & \OO^\ast \ar[r] & 0,}
    \]
and $\kappa$ is the boundary map in cohomology of this short exact sequence.

\begin{thm} \label{ThmCDNEtale}  There is a $\Gamma_{\Q_p}\times \GL_n(\Q_p)$-equivariant isomorphism
    \[ 
        r_i\from \St_i(\Z_p)^\ast \isomto H^i_{\proet}(\HH_C^{n-1},\Z_p(i))
    \]    
such that (in the $i=1$ case) the relation
    \[ 
        r_1(\delta_{\ell_1}-\delta_{\ell_2}) = \kappa(\tilde{\ell}_1/\tilde{\ell}_2)
    \]
holds for all $\ell_1,\ell_2\in\mathbb{P}^{n-1}(\Q_p)$.  Here the $\tilde{\ell}_i$
are $\Q_p$-rational linear forms in $n$ variables that represents $\ell_i$, so  $\tilde{\ell}_1/\tilde{\ell}_2\in \OO^\ast(\HH_C^{n-1})$. In particular, the right hand side does not depend on the choice of representatives $\tilde{\ell}_i$.
\end{thm}

\begin{proof} The map $r_i$ is defined in \cite[Proposition 4.7]{CDNIntegral};  for $i=1$ it is characterized by the relation $r_1(\delta_{\ell_1}-\delta_{\ell_2})=\kappa(\tilde{\ell}_1/\tilde{\ell}_2)$.  (There is a characterization of $r_i$ for general $i$ in terms of a ``regulator map.'') The isomorphy of $r_i$ is \cite[Theorem 5.1]{CDNIntegral}. 
\end{proof}

Let us turn now to pro-\'etale cohomology.  For a rigid space $X$ over $C$, the short exact sequence of abelian sheaves on $X_{\proet}$ from \Cref{lem:log_ses},
    \[  
        0 \to \mu_{p^\infty} \to \OO^{\ast\ast} \stackrel{\log}{\to} \OO \to 0, 
    \]
gives rise, after taking sequential limits along the multiplication by $p$ map to a short exact sequence of sheaves of $\Q_p$-vector spaces on $X_{\proet}$
    \begin{equation}\label{eq:logg}
         0 \to \Q_p(1)\to \varprojlim_{\times p} \Oxx \stackrel{\widetilde{\log}}{\to} \OO\to 0.
     \end{equation}
We thus get a connecting map $\partial_{\widetilde{\log}}\colon\OO[-1]\to \Q_p(1)$ in the derived category of sheaves on $X_{\proet}$.  The pro-\'etale cohomology of $\OO$ was studied in {\cite[Proposition 3.23]{scholze_perfectoidspacesasurvey}}:  If $X$ is a smooth rigid space over $C$ and  $\nu\from X_{\proet}\to X_{\et}$ is the projection, there is a natural isomorphism $R^i\nu_*\OO \isom \Omega^i(-i)$, where $\Omega^i=\Omega^i_{X/C}$ is the sheaf of differential $i$-forms.  As a result there is a spectral sequence
\[ H^i_{\et}(X,\Omega^j(-j))\implies H^{i+j}_{\proet}(X,\OO). \]
In the special case that $X$ is a Stein space, the coherent sheaves $\Omega^j_{X/C}$ are acyclic, and so 
    \begin{equation}\label{eq:forms}
        H^i_{\proet}(X,\OO)\isom \Omega^i(X)(-i).
    \end{equation}
Therefore the connecting map $\OO[-1]\to \Q_p(1)$ induces a map
    \begin{equation}\label{eq:preexp}
       \exp\from \Omega^{i-1}(X)(1-i) \to H^i_{\proet}(X,\Q_p(1)).
    \end{equation}  
The map $\exp$ is related to the Bloch--Kato exponential, see \cite[\S 3.2]{VPicardGroup} for a discussion.   When $X$ admits a semistable formal model over the ring of integers of a local field, \cite{CDNStein} identifies the kernel of $\exp$ with the module of closed forms $\ker d$, and identifies the cokernel of $\exp$ in terms of the Hyodo--Kato cohomology of the special fiber of the model.  
We are now ready to state the theorem of Colmez--Dospinescu--Nizio\l on the pro-\'etale cohomology of $\HH_C^{n-1}$.

\begin{thm} \label{ThmCDNProetale} There is a $\Gamma_{\Q_p}\times\GL_n(\Q_p)$-equivariant exact sequence of $\Q_p$-vector spaces
    \[
        0\to \frac{\Omega^{i-1}(\HH_C^{n-1})}{\ker d} \stackrel{\exp}{\to} H^i_{\proet}(\HH_C^{n-1},\Q_p(i)) \to \St_i(\Q_p)^\ast \to 0,
    \]
compatible with \cref{ThmCDNEtale}, in the sense that the evident diagram involving $H^i_{\proet}(\HH_C^{n-1},\Z_p(i))\to H^i_{\proet}(\HH_C^{n-1},\Q_p(i))$ commutes.
\end{thm}

We readily deduce the pro-\'etale cohomology of $\mu_{p^\infty}$ on $\HH_C^{n-1}$:

\begin{cor}\label{Cor:CDNproetale}
    There is a $\Gamma_{\Q_p}\times\GL_n(\Q_p)$-equivariant exact sequence
        \[
            0\to \frac{\Omega^{i-1}(\HH_C^{n-1})}{\ker d} \stackrel{\exp}{\to} H^i_{\proet}(\HH_C^{n-1},\mu_{p^\infty}(i-1)) \to \St_i(\Q_p/\Z_p)^\ast \to 0.
        \]
\end{cor}
\begin{proof}
    \Cref{ThmCDNEtale} and \cref{ThmCDNProetale} together supply a map of exact sequences
        \[
            \xymatrix{ & 0 \ar[r] \ar[d] &  H^i_{\proet}(\HH_C^{n-1},\Z_p(i)) \ar[r]^-{\cong} \ar[d] &  \St_i(\Z_p)^\ast \ar[r] \ar[d] & 0 \\
            0 \ar[r] & \frac{\Omega^{i-1}(\HH_C^{n-1})}{\ker d} \ar[r]^-{\exp} &  H^i_{\proet}(\HH_C^{n-1},\Q_p(i)) \ar[r] & \St_i(\Q_p)^\ast \ar[r] & 0.}
        \]
    Note that the right vertical map is injective with quotient $\St_i(\Q_p/\Z_p)^\ast$, see \eqref{EqSpExactSequence}.
    Thus, for every $i$ and $j$ the map 
    \[H^i_{\proet}(\HH_C^{n-1},\Z_p(j)) \to H^i_{\proet}(\HH_C^{n-1},\Q_p(j))\]
    is injective (note that $j$ does not change the underlying groups). From the long exact sequence associated to the $(i-1)$-fold twist of the short exact sequence
        \begin{equation}\label{eq:canseq}
            0 \to \Z_p(1)\to \Q_p(1) \to \mu_{p^\infty} \to 0,
        \end{equation} 
    we learn that there is an isomorphism
    \[H^i_{\proet}(\HH_C^{n-1},\mu_{p^\infty}(i-1)) \simeq H^i_{\proet}(\HH_C^{n-1},\Q_p(i))/H^i_{\proet}(\HH_C^{n-1},\Z_p(i)).\]
    The snake lemma then provides the short exact sequence of the corollary.
\end{proof}

\subsection{Consequences for $H^i_{\proet}(\HH_{C}^{n-1},\Oxx)$}
We gather here some consequences of \cref{ThmCDNEtale} and \cref{ThmCDNProetale}. The next results concern the $\Oxx$- and $\lim_p\Oxx$-cohomology of $\HH_C^{n-1}$.  

\begin{thm} 
\label{ThmOxxCohomologyOfH}
For $i\geq 0$, there is a map of $\Pi_n$-equivariant short exact sequences
    \[
            \xymatrix{0 \ar[r] & \St_i(\Q_p)^{\ast}(1-i) \ar[r] \ar[d] &  H^i_{\proet}(\HH_C^{n-1},\lim_p \Oxx) \ar[r] \ar[d]  & \Omega^{i,\cl}(\HH_C^{n-1})(-i) \ar[r] \ar[d]^{\simeq}  & 0 \\
            0 \ar[r] & \St_i(\Q_p/\Z_p)^*(1-i) \ar[r] & H^i_{\proet}(\HH_C^{n-1},\Oxx) \ar[r] & \Omega^{i,\cl}(\HH_C^{n-1})(-i) \ar[r] & 0,}
    \]
where $\Omega^{i,\cl}$ denotes the sheaf of closed differential $i$-forms on $\HH_C^{n-1}$ and the vertical maps are the canonical ones.
\end{thm}
\begin{proof}
    We begin with the construction of the lower short exact sequence. Let $\partial_{\log}\colon \OO \to \mu_{p^\infty}[1]$ be the connecting homomorphism  of the logarithm exact sequence \eqref{eq:logseq}, and write
        \[
            \partial_{\log}^i\colon H^i_{\proet}(\HH_C^{n-1},\OO) \to H^{i+1}_{\proet}(\HH_C^{n-1},\mu_{p^\infty})
        \]
    for the induced map on cohomology. It follows from the long exact sequence associated to the logarithm sequence that there is a short exact sequence
        \[
            0 \to \coker(\partial_{\log}^{i-1}) \to H^i_{\proet}(\HH_C^{n-1},\Oxx) \to \ker(\partial_{\log}^{i}) \to 0. 
        \]
    Using \cref{Cor:CDNproetale}, we can identify the kernel and cokernel of $\partial_{\log}^i$. To this end, consider the cofiber sequence
    \[\OO^{**} \to \OO \to \mu_{p^\infty}[1.]\] 
    Applying $\varprojlim_{\times p}$, we get a commutative diagram 
     \[
        \xymatrix{\varprojlim_{p} \OO^{**}\ar[d]\ar[r] & \OO\ar@{=}[d]\ar[r]^-{\partial_{\widetilde{\log}}} & \Q_p(1)[1] \ar[d]\\
        \OO^{**}\ar[r]& \OO\ar[r]^-{\partial_{\log}} & \mu_{p^\infty}[1],}
    \]
    so we get a  lift
        \begin{equation}\label{eq:deltafactorization}
            \xymatrix{& \OO \ar[d]^{\partial_{\log}} \ar@{-->}[dl]_-{\partial_{\widetilde{\log}}} & \\
            \Q_p(1)[1] \ar[r] & \mu_{p^\infty}[1] \ar[r] & \Z_p(1)[2],}
        \end{equation}
    where the bottom exact sequence is a rotation of the triangle coming from \eqref{eq:canseq}. Note that this lift is unique  as all maps $\OO \to \Z_p(1)[1]$ are trivial, since the target is $p$-complete while the $p$ acts invertibly on the source. By the discussion after \Cref{ThmCDNEtale}, on cohomology, the lift $\partial_{\widetilde{\log}}$ gives precisely the map \eqref{eq:preexp}.

    Combining this observation with \cref{Cor:CDNproetale} and \eqref{eq:forms}, we thus obtain a map of exact sequences
        \[
            \xymatrix{0 \ar[r] & \Omega^i(\HH_C^{n-1})(-i) \ar[r]^-{\cong} \ar@{->>}[d] & H^i_{\proet}(\HH_C^{n-1},\OO) \ar[r] \ar[d]^-{\partial_{\log}^i} & 0 \ar[d] \\
            0 \ar[r] & \frac{\Omega^{i}(\HH_C^{n-1})}{\ker d}(-i) \ar[r]^-{\exp} & H^{i+1}_{\proet}(\HH_C^{n-1},\mu_{p^\infty}) \ar[r] &  \St_{i+1}(\Q_p/\Z_p)^\ast(-i) \ar[r] & 0}
        \]
    By our discussion above, the left vertical map is the canonical quotient map. The snake lemma then implies that
        \[
            \ker(\partial_{\log}^i) \cong \Omega^{i,\cl}(\HH_C^{n-1})(-i) \qquad \text{and} \qquad \coker(\partial_{\log}^i) \cong \St_{i+1}(\Q_p/\Z_p)^\ast(-i),
        \]
    as desired.

    The factorization \eqref{eq:deltafactorization} induces a map of short exact sequences, in which we have already identified the bottom sequence:
        \[
            \xymatrix{0 \ar[r] & \coker(\partial_{\widetilde{\log}}^{i-1}) \ar[r] \ar[d] &  H^i_{\proet}(\HH_C^{n-1},\lim_p \Oxx) \ar[r] \ar[d]  & \ker(\partial_{\widetilde{\log}}^{i}) \ar[r] \ar[d]  & 0 \\
            0 \ar[r] & \coker(\partial_{\log}^{i-1}) \ar[r] & H^i_{\proet}(\HH_C^{n-1},\Oxx) \ar[r] & \ker(\partial_{\log}^{i}) \ar[r] & 0.}
        \]
    The same argument applied to the sequence \eqref{eq:logg} and using \cref{ThmCDNProetale} instead of \cref{Cor:CDNproetale} establishes the analogous computation for cohomology with coefficients in $\lim_p \Oxx$, i.e., it identifies also the top sequence. This gives the stated map of short exact sequences.
\end{proof}

\begin{cor} \label{lem:sescoh}
    Consider the short exact sequence of pro-\'etale sheaves on $\HH_C^{n-1}$:
        \[
            0\to \Z_p(1) \to \lim_{p}\Oxx \xrightarrow{} \Oxx \to 0.
        \]
    The associated long exact sequence for cohomology has zero boundary maps. 
    Therefore for every $i \geq 0$ we obtain a map of short exact sequences of $\Pi_n$-modules
\[\begin{tikzcd}
	0 & {\St_i(\Z_p)^{\ast}}(1-i) & {\St_i(\Q_p)^{\ast}}(1-i) & {\St_i(\Q_p/\Z_p)^{\ast}}(1-i) & 0 \\
	0 & {H^i_{\proet}(\mathcal{H}_C^{n-1},\Z_p(1)) } & {H^i_{\proet}(\mathcal{H}_C^{n-1},\lim_p \Oxx) } & { H^i_{\proet}(\mathcal{H}_C^{n-1},\Oxx )} & 0.
	\arrow[from=1-1, to=1-2]
	\arrow[from=1-2, to=1-3]
	\arrow["{\simeq }", from=1-2, to=2-2]
	\arrow[from=1-3, to=1-4]
	\arrow[from=1-3, to=2-3]
	\arrow[from=1-4, to=1-5]
	\arrow[from=1-4, to=2-4]
	\arrow[from=2-1, to=2-2]
	\arrow[from=2-2, to=2-3]
	\arrow[from=2-3, to=2-4]
	\arrow[from=2-4, to=2-5]
\end{tikzcd}\]   
    \end{cor}
\begin{proof}
The splitting of the long exact sequence into short exact sequences is equivalent to the surjectivity of  the map
    \[
        H^i_{\proet}(\mathcal{H}_C^{n-1},\lim_p \Oxx) \to  H^i_{\proet}(\mathcal{H}_C^{n-1},\Oxx )
    \] 
for all $i$. Note that this map factors as
    \[
        H^i_{\proet}(\mathcal{H}_C^{n-1},\lim_p \Oxx) \to \lim_p H^i_{\proet}(\mathcal{H}_C^{n-1},\Oxx ) \to H^i_{\proet}(\mathcal{H}_C^{n-1},\Oxx ).
    \]
Since $\Oxx \xrightarrow{\times p} \Oxx$ is surjective, we have $\lim_p \Oxx \simeq R\lim_p \Oxx$. So \[R\Gamma_{\proet}(\mathcal{H}_C^{n-1}, \lim_p\Oxx)\simeq R\lim_p R\Gamma_{\proet}(\mathcal{H}_C^{n-1}, \Oxx).\]
Thus, by the Milnor sequence (which exists because the pro-\'etale topos is replete), the  map above  
    \[
        H^i_{\proet}(\mathcal{H}_C^{n-1},\lim_p \Oxx) \to  \lim_p H^i_{\proet}(\mathcal{H}_C^{n-1},\Oxx )
    \]
is a surjection. The second map is surjective as well, since 
$H^i_{\proet}(\mathcal{H}_C^{n-1},\Oxx)$ is $p$-divisible. This follows from \cref{ThmOxxCohomologyOfH}, which exhibits $H^i_{\proet}(\mathcal{H}_C^{n-1},\Oxx)$ as an extension of $p$-divisible abelian groups.  

From \Cref{ThmOxxCohomologyOfH}, we obtain a commutative square of $\Pi_n$-modules:
    \[
        \begin{tikzcd}
	{\St_i(\Q_p)^{\ast}}(1-i) & {\St_i(\Q_p/\Z_p)^{\ast}}(1-i) \\
	{H^i_{\proet}(\mathcal{H}_C^{n-1},\lim_p \Oxx) } & { H^i_{\proet}(\mathcal{H}_C^{n-1},\Oxx )}.
	\arrow[from=1-1, to=1-2]
        \arrow[from=1-1, to=2-1]
        \arrow[from=1-2, to=2-2]
	\arrow[from=2-1, to=2-2]
        \end{tikzcd}
    \]
This diagram extends to the desired map of short exact sequences by \eqref{EqSpExactSequence} and the vanishing of the boundary map established above.
\end{proof}

Finally, we identify the $\Pi_n$-invariants in $H^1_{\proet}(\HH_C^{n-1},\Oxx)$.

\begin{cor} \label{CorRank1Invariants} 
There is a canonical isomorphism
    \[ 
        H^1_{\proet}(\HH_C^{n-1},\Oxx)^{\Pi_n}\isom \left(\St_1(\Q_p/\Z_p)^{\ast}\right)^{\GL_n(\Z_p)}.
    \]
Thus if $n\geq 2$, then by \cref{LemFixedPartOfSp}, $H^1_{\proet}(\HH_C^{n-1},\Oxx)^{\Pi_n}$ is a free $\Z_p$-module of rank 1 with distinguished generator. When $n=1$ we have $H^1_{\proet}(\HH_C^{0},\Oxx)^{\Pi_1}=0$.
\end{cor}

\begin{proof}  We claim that if $X$ is any smooth rigid-analytic space over $\Q_p$, and $j\in\Z$ is nonzero, then $\Omega^1(X_C)(-j)$ has no $\Gamma_{\Q_p}$-invariants.  See the proof of \cite[Theorem 5.1.2]{BSSW} for the proof of this claim.  Now take $\Pi_n$-invariants in the exact sequence in the $i=1$ case of \cref{ThmOxxCohomologyOfH}.  We find an isomorphism
    \[ 
        H^0_{\cts}(\Pi_n,\St_1(\Q_p/\Z_p)^\ast)\isom H^0_{\cts}(\Pi_n,H^1_{\proet}(\HH_C^{n-1},\Oxx)). 
    \]
The group $\Pi_n$ acts on $\St_1(\Q_p/\Z_p)^\ast$ through its quotient $\GL_n(\Z_p)$, and we have seen in \cref{LemFixedPartOfSp} that the module of invariants is rank 1 over $\Z_p$ for $n\geq 2$. For $n=1$, $\HH_C^0 = \Spa(C,\OO_C)$ and thus $H^1_{\proet}(\HH_C^0,
\mathcal{F})$ is zero for any sheaf $\mathcal{F}$.
\end{proof}

\subsection{A primitivity result} Assume throughout that $n\geq 2$. 
Let us recall the fundamental exact sequence from \cref{ThmFundamentalExactSequence}, applying what we have learned from \cref{CorRank1Invariants}:
\[
\xymatrix{
0 \ar[r] &  H^1_{\cts}(\G_1,A_1^{\ast\ast}) \ar[r] & H^1_{\cts}(\G_n,A_n^{\ast\ast}) \ar[r]^-{b} & H^1_{\proet}(\HH_C^{n-1},\Oxx)^{\Pi_n} \ar[d]_{\isom} \\
& & & \Z_p
}
\]
Our current goal is to prove that the map labeled $b$ is surjective.  To do this, it suffices to show that the image under $b$ of the class $\varepsilon_p(\Sigma^2\mathbb{S}_{K(n)})$ is primitive, i.e., that it is a generator of the $\Z_p$.  More precisely, we will prove:

\begin{thm}
\label{ThmSphereClassIsPrimitive} The map $b$ in \cref{ThmFundamentalExactSequence} sends $\varepsilon_p(\Sigma^2\mathbb{S}_{K(n)})$ to the distinguished generator of the $\Z_p$-module $H^1_{\proet}(\HH_C^{n-1},\Oxx)^{\Pi_n}$ from \cref{CorRank1Invariants}.  
\end{thm}

The players in the proof of \cref{ThmSphereClassIsPrimitive} fit into the following commutative diagram:
\begin{equation}\label{EqDiagramChase}
    \begin{gathered}
    \xymatrix{H^1_{\cts}(\G_n,A_n^{\xx}) \ar[d]_{g_0} \ar@/^-4pc/[dd]_{b} & H^1_{\proet}(\SW_n,\OO^{\ast}) \\
    H^1_{\proet}(\SW_n,\Oxx) \ar[d]_{e_1} \ar[ru]^-{g_1} \ar[rd]_{f_1} & H^1_{\cts}( \Pi_n, \OO^{\ast}(\HH_C^{n-1})) \ar[u]_{g_2} \ar[d]^{f_2} \\
    H^0_{\cts}(\Pi_n,H^1_{\proet}(\HH_C^{n-1},\Oxx)) \ar[r]_-{\phi_2} \ar[d]_{i_1}^{\cong} & H^1_{\cts}(\Pi_n, H^1_{\proet}(\mathcal{H}_C^{n-1},\Z_p(1))) \ar[d]^{i_2}_{\cong} \\
    H^0_{\cts}(\GL_n(\Z_p),\St_1(\Q_p/\Z_p)^\ast) \ar[r]^-{\partial} & H^1_{\cts}(\GL_n(\Z_p),\St_1(\Z_p)^{\ast}).}
    \end{gathered}
\end{equation}
We explain the maps appearing in \eqref{EqDiagramChase}. The map labeled $g_0$ arises from the equivalence $\SW_n\isom [\LT_{n,\eta}^{\diamond}/\G_n]$, $g_1$ arises from the inclusion $\Oxx \to \OO^\ast$, and $g_2$ arises from the equivalence $\SW_n\isom [\HH_C^{n-1}/\Pi_n]$. The map labeled $\phi_2$ appears as a connecting map when applying the long exact sequence in continuous $\Pi_n$-cohomology to the short exact sequence of $\Pi_n$-modules in 
\cref{lem:sescoh}.  The bottom map labeled $\partial$ is the connecting map from \eqref{EqMapDelta}; it is injective as noted in the discussion around \eqref{EqMapDelta}. \Cref{lem:sescoh} implies that the bottom square in \eqref{EqDiagramChase} commutes. The map $i_1$ is the isomorphism from \cref{CorRank1Invariants}, and $i_2$ arises from the identification of $H^1_{\proet}(\HH_C^{n-1},\Z_p(1))$ with $\St_1(\Z_p)^{\ast}$ in \cref{ThmCDNEtale}.   In fact $i_2$ is an isomorphism, as one sees by applying the inflation-restriction sequence to the normal subgroup $\Gamma_{\Q_p}\subset \Pi_n$, using the fact that $\St_1(\Z_p)^\ast$ has no $\GL_n(\Z_p)$-invariants. 
Finally, $e_1$ appears in \eqref{EqBigAssDiagram}, as an example of the low-degree exact sequence \eqref{EqLowDegree}, and $b = e_1 \circ g_0$ and $f_1 = \phi_2 \circ e_1$ are the composites. 

For the construction of $f_2$, start with the short exact sequence of $\Pi_n$-equivariant sheaves 
on $\HH_C^{n-1}$:
\[
    0 \to \Z_p(1) \to \lim_p \OO^* \to \OO^* \to 0.
\]
Consider the boundary map 
\[
    \OO^*(\mathcal{H}_{C}^{n-1}) \to H^1_{\proet}(\mathcal{H}_{C}^{n-1},\Z_p(1))
\]
which is a map of $\Pi_n$-modules. Applying $H^1_{\cts}(\Pi_n,-)$ gives
\[
    f_2 \colon H^1_{\cts}( \Pi_n, \OO^{\ast}(\HH_C^{n-1})) \to H^1_{\cts}(\Pi_n, H^1_{\proet}(\mathcal{H}_C^n,\Z_p(1))).
\]

\begin{lem}\label{LemTwoImagesAgree} Let $a_1=g_0\varepsilon(\Sigma^{2(p^n-1)}\mathbb{S}_{K(n)})$, and let $a_2\in H^1_{\cts}(\Pi_n,\OO^\ast(\HH_C^{n-1}))$ be the class of the cocycle $(\sigma,g)\mapsto (\chi(\sigma)g(\ell)/\ell)^{-(p^n-1)}$, where $\ell\in H^0(\PP^n_{\Q_p},\OO(1))$ is any $\Q_p$-rational linear form, and $\chi\from \Gamma_{\Q_p}\to \Z_p^\times$ is the cyclotomic character.  Then $g_1(a_1)=g_2(a_2)$.
\end{lem}

\begin{proof} Let $H=H_n^{\univ}$ be the universal deformation of $H_n$.  By \cref{ThmEquivalenceOfStacks}, 
 there is an isomorphism $\SW_n\isom [\HH_C^{n-1}/\Pi_n]$ which identifies the line bundles $L:=(\Lie H)^{\diamondsuit}_{\eta}$ and $\OO(-1)^{\diamondsuit}\otimes_{\Q_p}\Q_p(-1)$.  
Since $\ell$ is a nowhere vanishing section for $\OO(1)\vert_{\HH_C^{n-1}}$, the image of $a_2$ in $H^1_{\proet}(\SW_n,\OO^\ast)$ is the class of $L^{\otimes (p^n-1)}$.  This equals $g_1(a_1)$:  indeed $\varepsilon(\Sigma^2\mathbb{S}_{K(n)})\in H^1_{\cts}(\G_n,A_n^\ast)$ is the class of $L$ as a $\G_n$-equivariant invertible $A_n$-module, and then after multiplying by $2(p^n-1)$ one gets the claimed equality.
\end{proof} 

\begin{prop}\label{PropTwoMaps}
    Consider two classes $x_1\in H^1_{\proet}(\SW_n,\Oxx)$ and $x_2 \in H^1_{\cts}( \Pi_n, \OO^{\ast}(\HH_C^{n-1}))$.
If $g_1(x_1) = g_2(x_2)$, then $f_1(x_1)  = f_2(x_2)$.
\end{prop}

The proof of \cref{PropTwoMaps} is postponed to the next subsection.

\begin{proof}[Proof of \cref{ThmSphereClassIsPrimitive} assuming \cref{PropTwoMaps}]  By \cref{LemTwoImagesAgree} and \cref{PropTwoMaps}, we have
\[ f_1g_0\varepsilon(\Sigma^{2(p^n-1)}\mathbb{S}_{K(n)})=(p^n-1)f_2(c),\]
where $c$ the class of the cocycle $(\sigma,g)\mapsto (\chi(\sigma)g(\ell)/\ell)^{-1}$ in $H^1_{\cts}(\Pi_n,\OO^\ast(\HH_C^{n-1}))$.  The prime-to-$p$ part of $H^1_{\cts}(\G_n,A_n^\ast)$ has order $(p^n-1)$, and so $(p^n-1)\varepsilon_p(\Sigma^2\mathbb{S}_{K(n)})=\varepsilon(\Sigma^{2(p^n-1)}\mathbb{S}_{K(n)})$.  Since $(p^n-1)$ acts invertibly on the $\Z_p$-module $H^1_{\cts}(\Pi_n,H^1_{\proet}(\HH_C^{n-1},\Z_p(1)))$, we must have $f_1g_0\varepsilon_p(\Sigma^2\mathbb{S}_{K(n)})=f_2(c)$.  Apply $i_2$ and use the commutativity of the diagram in \eqref{EqDiagramChase} to find that $\partial i_1b\varepsilon_p(\Sigma^2\mathbb{S}_{K(n)})$ equals the cocycle $g\mapsto \delta_{\ell}-\delta_{g(\ell)}$ from \cref{LemFixedPartOfSpInH1}.  Since $\partial$ is injective and $i_1$ is an isomorphism, we conclude that $b\varepsilon_p(\Sigma^2\mathbb{S}_{K(n)})$ is the distinguished generator of $H^1_{\proet}(\HH_C^{n-1},\Oxx)^{\Pi_n}$.  
\end{proof}

\subsection{Proof of \cref{PropTwoMaps} and \cref{thm:A**main}}
We begin with some preparatory lemmas.

\begin{lem}\label{LemTruncation:alt}
Assume that $C$ and $D$ are complexes in $D(\Ab)$ such that $C \simeq \tau^{[0,1]}C$ and $D \simeq \tau^{[0,1]}D$ (i.e., $C,D$ are concentrated in cohomological degrees $0$ and $1$). Assume that $H^*(C)$ is $p$-divisible and $H^*(D)$ is derived $p$-complete.   Then 
every map $f \colon C \to D$ factors as 
\[
C \to H^1(C)[-1] \to H^0(D) \to D,
\]
where $H^1(C)$ and $H^0(D)$ are viewed as complexes concentrated in degree $0$.
\end{lem}
\begin{proof}
Since the cohomology groups $H^*(D)$ are derived $p$-complete, $D$ is derived $p$-complete, so the map $f$ factors as
    \[
        C \to \widehat{C} \to D,
    \]
where $\widehat{C}$ is the derived $p$-completion of $C$. Since $H^*(C)$ is $p$-divisible, we have $H^i(\widehat{C}) = \widehat{H^{i+1}(C)}$. In particular, we have $\widehat{C} \simeq \tau^{[-1,0]}\widehat{C}$. We thus get a decomposition 
    \[
        C \to \widehat{C} \to H^0(\widehat{C}) \to H^0(D) \to D
    \]
and  $H^0(\widehat{C}) \cong \widehat{H^1(C)[-1]}$. We get that the composition $C \to \widehat{C} \to H^0(\widehat{C})$ factors
as 
    \[ 
        C \to H^{1}(C)[-1] \to \widehat{H^{1}(C)[-1]} \simeq H^0(\widehat{C}).
    \]
Combining the last two displayed factorizations gives the claim.
\end{proof}

Let $A$ be a $p$-divisible pro-\'etale sheaf.  There is a short exact sequence of sheaves
\[ 0 \to T_pA \to \varprojlim_p A \to A \to 0,\]
where $T_pA=\varprojlim_m A_{p^m}$. This induces a map $A \to T_p A[1]$.  Let $R\Gamma(A)$ denote the sheaf cohomology of $A$.  Assume that the corresponding boundary maps $H^{i}_{\proet}(A)\xrightarrow{\partial} H^{i+1}_{\proet}(T_pA)$ for $i=0,1$ are zero. Then $H^*(A)$ are quotients of $p$-divisible groups for $i=0,1$.  
Then, the map $\tau^{[0,1]}{R\Gamma_{\proet}}(A) \to \tau^{[0,1]}{R\Gamma_{\proet}}(T_pA[1])$ satisfies the requirements of \cref{LemTruncation:alt}.  
This provides us with a factorization:
\[
\xymatrix{
\tau^{[0,1]}{R\Gamma_{\proet}}(A) \ar[r] \ar[d] & H^1_{\proet}(A)[-1] \ar[d] \\
\tau^{[0,1]}{R\Gamma_{\proet}}(T_pA[1]) & H^1_{\proet}(T_pA)[0]. \ar[l]
}
\]

\begin{lem} \label{LemH1InducesIsom} In the derived category of condensed abelian groups with $\Pi_n$-action,
    the map 
    \[
        H^1_{\proet}(\HH_C^{n-1}, \Z_p(1))[0] \to \tau^{[0,1]}{R\Gamma_{\proet}}(\HH_C^{n-1},  \Z_p(1)[1])
    \]
    is sent to an isomorphism  via the functor
    \[ A\mapsto H^1_{\cts}(\Pi_n,A):= H^1(A^{h\Pi_n}). \]
\end{lem}
\begin{proof}
    Consider the cofiber sequence 
        \[
            H^1_{\proet}(\mathcal{H}_C^{n-1}, \Z_p(1))[0] \to \tau^{[0,1]}{R\Gamma_{\proet}}(\HH_C^{n-1},  \Z_p(1)[1]) \to H^2_{\proet}(\HH_C^{n-1}, \Z_p(1))[-1].
        \]
    It is enough to show that 
        \[
            H^0_{\cts}(\Pi_n,H^2_{\proet}(\HH_C^{n-1}, \Z_p(1))) = 0.
        \]
From \cref{ThmCDNEtale} we have $H^2_{\proet}(\HH_C^{n-1},\Z_p(1))\isom \St_2(\Z_p)^{\ast}(-1)$, which has no nonzero $\Pi_n$-fixed vectors because $\Z_p(-1)^{\Gamma_{\Q_p}}=0$.
\end{proof}

\begin{proof}[Proof of \cref{PropTwoMaps}]
Consider the following diagram, in which all cohomology is computed on the pro-\'etale site of $\HH_C^{n-1}$:

\[
\xymatrix{
& \tau^{[0,1]}{R\Gamma_{\proet}}(\Oxx) \ar[ddl] \ar[ddd] \ar[rr] & & H^1(\Oxx)[-1] \ar[ddd] \\
&&&\\
\tau^{[0,1]}{R\Gamma_{\proet}}(\OO^{\ast}) \ar[dr] & & H^0(\OO^\ast)[0] \ar@{->}'[l][ll] \ar[dr] & \\
& \tau^{[0,1]}{R\Gamma_{\proet}}(\Z_p(1)[1]) & & H^0(\Z_p(1)[1]). \ar[ll]
}
\]

The diagram is commutative: the front rectangle commutes by \cref{LemTruncation:alt}, the bottom parallelogram commutes by naturality of the edge maps, and the left triangle arises from the composition $\OO^{\ast\ast}\to \OO^{\ast}\to \Z_p(1)[1]$.  

Applying the functor $A\mapsto H^1(A^{h\Pi_n})$ to the above diagram yields: 

\[
\resizebox{\textwidth}{!}{
\xymatrix{
& H^1(\SW_n, \Oxx) \ar[ddl]_{g_1} \ar[ddd]_{\phi_1} \ar[rr]^{e_1} & & H^1(\HH_C^{n-1},\Oxx)^{\Pi_n} \ar[ddd]_{\phi_2} \\
&&&\\
H^1(\SW_n, \OO^{\ast}) \ar[dr]_{\delta} & & H^1(\Pi_n,\OO^{\ast}(\HH_C^{n-1})) \ar@{->}'[l][ll]_{\;\;\;\;\;\;g_2} \ar[dr]_{f_2} & \\
& H^1(\Pi_n,{R\Gamma_{\proet}}(\HH_C^{n-1},\Z_p(1)[1])) & & H^1(\Pi_n,H^1(\HH_C^{n-1},\Z_p(1)[1])). \ar[ll]_{e_2}^{\sim}
}}
\]
Here $e_2$ is an isomorphism by \cref{LemH1InducesIsom}.
We can now prove \cref{PropTwoMaps}, in which $f_1=\phi_2\circ e_1$.  Given that $g_1(a_1)=g_2(a_2)$, we have 
    \[ 
        f_1(a_1)=\phi_2e_1(a_1)=e_2^{-1}\delta g_1(a_1) =e_2^{-1}\delta g_2(a_2)=f_2(a_2),
    \]
as desired.
\end{proof}

\Cref{ThmSphereClassIsPrimitive} is now proved.  We explain how it implies \cref{thm:A**main}, which in turn implies our main theorem (\cref{thm:picalg}), see the discussion following \cref{thm:A**main}.  

\begin{proof}[Proof of \cref{thm:A**main}]
    We first consider the case $(n,p)\neq (2,2)$. \Cref{ThmFundamentalExactSequence} together with \cref{ThmSphereClassIsPrimitive} state that there is an exact sequence
        \[
            0\to H^1_{\cts}(\G_1,A_1^{\xx}) \xrightarrow{\det_{\LT}^{\ast}} H^1_{\cts}(\G_n,A_n^{\xx})\to \Z_p \to 0 
        \]
    in which $\varepsilon_p(\Sigma^2\mathbb{S}_{K(n)})$ is sent to a generator of the $\Z_p$.  Since $H^1_{\cts}(\G_1,A_1^{\xx})$ and $\Z_p$ are $p$-complete the exact sequence is a sequence of $\Z_p$-modules. So as $\Z_p$ is free as $\Z_p$-module, the exact sequence splits:
        \begin{equation}\label{eq:proofofthm:A**main1}
            H^1_{\cts}(\G_n,A_n^{\xx})\isom \Z_p\oplus H^1_{\cts}(\G_1,A_1^{\xx}).
        \end{equation} 
    Similarly for $(n,p)= (2,2)$, we have an isomorphism
        \begin{equation}\label{eq:proofofthm:A**main2}
            H^1_{\cts}(\G_2,A_2^{\xx})\isom \Z_2\oplus H^1_{\cts}(\G_1,A_1^{\xx})\oplus\ZZ/2,
        \end{equation} 
    where in both cases the $\Z_p$ summand is generated by $\varepsilon_p(\Sigma^2\mathbb{S}_{K(n)})$. The term $H^1_{\cts}(\G_1,A_1^{\xx})$ has been computed in \cref{lem:height1contribution}; substituting this into \eqref{eq:proofofthm:A**main1} yields isomorphisms
        \[
            H^1_{\cts}(\bG_n,A_n^{\ast\ast}) \cong 
                \begin{cases}
                    \Z_p^2  & \text{if } p>2; \\
                    \Z_2^2 \oplus (\Z/2)^{\oplus 2} & \text{if } p=2 \text{ and } n>2; \\
                    \Z_2^2 \oplus (\Z/2)^{\oplus 3} & \text{if } p=2 \text{ and } n=2.
                \end{cases}
        \]
    As explained in \Cref{ThmFundamentalExactSequence}, a generator of the torsion-free part of $H^1_{\cts}(\G_1,A_1^{\xx})$ is sent to the determinant in $H^1_{\cts}(\G_n,A_n^{\xx})$, thus finishing the proof.
\end{proof}

\biblio
\bibliography{bib}\bibliographystyle{alpha}
\end{document}